\documentclass[english,11pt, reqno]{amsart}
\usepackage{amsmath}
\usepackage{amssymb}
\usepackage{amsthm}
\usepackage{mathtools}
\usepackage{bbm,bm}
\usepackage{stmaryrd}
\usepackage{xcolor}
\usepackage[colorlinks=true,linkcolor=blue!95!black, citecolor = green!55!black,bookmarksdepth=3,pdfusetitle,pdfauthor={}]{hyperref}
\usepackage[margin=1in]{geometry}
\usepackage{multicol}
\usepackage{etoolbox}
\usepackage{enumitem}
\usepackage{float}
\usepackage[mathscr]{euscript}
\usepackage{graphicx}
\usepackage{subcaption}
\usepackage{tikz}

\allowdisplaybreaks

\parskip 	\smallskipamount

\newcommand{\hair}{\ifmmode\mskip1mu\else\kern0.08em\fi}
\renewcommand{\P}{\mathbb{P}}

\newcommand{\E}{\mathbb{E}}

\newcommand{\R}{\mathbb{R}}
\newcommand{\C}{\mathbb{C}}
\newcommand{\N}{\mathbb{N}}
\newcommand{\Z}{\mathbb{Z}}

\newcommand{\one}{\mathbbm{1}}

\newcommand{\mrm}{\mathrm}
\newcommand{\msf}{\mathsf}

\newcommand{\iid}{i.i.d.\ }
\newcommand{\diff}{\mathrm d}
\newcommand{\mugeo}{\mu^q_{\mrm{Geo}}}

\DeclareMathAlphabet{\mathdutchcal}{U}{dutchcal}{m}{n}

\usetikzlibrary{calc}
\usetikzlibrary{math}

\newtheorem{theorem}{Theorem}[section]
\newtheorem*{theorem*}{Theorem}
\newtheorem*{proposition*}{Proposition}
\newtheorem{proposition}[theorem]{Proposition}
\newtheorem*{corollary*}{Corollary}

\newtheorem{lemma}[theorem]{Lemma}

\theoremstyle{definition}
\newtheorem{definition}[theorem]{Definition}

\newtheorem{remark}[theorem]{Remark}

\title{The lower tail of $q$-pushTASEP}

\author{Ivan Corwin}

\author{Milind Hegde}

\newcommand{\resc}{X^{\mrm{sc}}_N}

\begin{document}

\begin{abstract}
We study $q$-pushTASEP, a discrete time interacting particle system whose distribution is related to the $q$-Whittaker measure. We prove a uniform in $N$ lower tail bound on the fluctuation scale for the location $x_N(N)$ of the right-most particle at time $N$ when started from step initial condition. Our argument relies on a map from the $q$-Whittaker measure to a model of periodic last passage percolation (LPP) with geometric weights in an infinite strip that was recently established in \cite{imamura2021skew}. 
By a path routing argument we bound the passage time in the periodic environment in terms of an infinite sum of independent passage times for standard LPP on $N\times N$ squares with geometric weights whose parameters decay geometrically. To prove our tail bound result we combine this reduction with a concentration inequality, and a crucial new technical result---lower tail bounds on $N\times N$  last passage times uniformly over all $N \in \N$ and all the geometric parameters in $(0,1)$. This technical result uses Widom's trick \cite{widom2002convergence} and an adaptation of an idea of Ledoux introduced for the GUE \cite{ledoux2005deviation} to reduce the uniform lower tail bound to uniform asymptotics for very high moments, up to order $N$, of the Meixner ensemble. This we accomplish by first  obtaining sharp uniform estimates for factorial moments of the Meixner ensemble from an explicit combinatorial formula of Ledoux \cite{ledoux2005distributions}, and translating them to polynomial bounds via a further careful analysis and delicate cancellation.
\end{abstract}

\maketitle

\setcounter{tocdepth}{1}
\tableofcontents

\section{Introduction, main results, and proof ideas}

The Kardar-Parisi-Zhang (KPZ) universality class consists of a large variety of models, all of which are believed to exhibit certain universal behaviors; for example, common scaling limits. Most progress in this area has been in the setting of certain models that are known as \emph{exactly solvable} or \emph{integrable}, which possess certain algebraic structure that makes their analysis within reach, in comparison to non-integrable models.

In such models, it is often important for applications to have control on the upper and lower tails of the KPZ observable on the fluctuation scale. Of the two, it is more challenging to obtain this control on the lower tail (i.e., the one that typically has cubic tail exponent), though in what are known as zero-temperature models, such as TASEP and last passage percolation, a variety of techniques have been developed over the last two decades to do this (e.g. Riemann-Hilbert methods or analysis of determinantal representations; see Section~\ref{s.intro.prior work} for a detailed discussion). In contrast, for \emph{positive} temperature models such as the KPZ equation, stochastic six vertex model, ASEP, and polymer models, only a few techniques have recently been developed to approach this problem. Further, each technique only seems to be applicable in particular cases; due to fundamental limitations, there is no broad coverage.

In this paper we study the exactly solvable, discrete time interacting particle system model of \emph{geometric $q$-pushTASEP}, a positive temperature model, but one for which the few methods available to obtain lower tails in positive temperature do not seem applicable. We develop a new technique for lower tail estimates on the position of the right-most particle, harnessing recently discovered connections between it and last passage percolation.

We start by introducing the model of study and our main results.

\subsection{Principal objects and models of study}

\subsubsection{Some notation and distributions}
The \emph{$q$-Pochhammer symbol} $(z;q)_n$ is given by
\begin{align*}
(z;q)_n = \prod_{i=0}^{n-1}(1-zq^i) \quad \text{ for } n=0,1, \ldots,
\end{align*}
with $(z;q)_{\infty}$ defined by replacing $n-1$ by $\infty$.
The \emph{$q$-binomial coefficient} is given by
\begin{align}\label{e.q-binomial coefficient}
\binom{n}{k}_{\!\!q} = \frac{(q;q)_n}{(q;q)_k(q;q)_{n-k}}.
\end{align}

The \emph{q-deformed beta binomial distribution} is a distribution with parameters $q$, $\xi$, $\eta$, and $m$. Here $m\in\Z_{\geq 0}$ and the distribution is defined on $\{0,1, \ldots, m\}$; the other parameters are non-negative real numbers, and are restricted to more specific domains in certain cases that we will describe. For $s\in\{0, \ldots, m\}$, the probability mass function at $s$ is given by
\begin{align*}
\varphi_{q,\xi, \eta}(s\mid m) = \xi^s \frac{(\eta/\xi; q)_s(\xi; q)_{m-s}}{(\eta; q)_m}\cdot \binom{m}{s}_{\!\!q}.
\end{align*}
We refer the reader to \cite[Section~6.1]{matveev2016q} for more information regarding this distribution, including a discussion on why the above expression sums (over $s=0, \ldots, m$) to $1$ when the expression is well-defined and non-negative.

A special case is the $q$-Geometric distribution of parameter $\xi$ (denoted $q$-Geo($\xi$)), obtained by taking $m=\infty$, $\eta=0$, and $q,\xi\in(0,1]$, so that the probability mass function at $s\in\Z_{\geq 0}$ is given by
\begin{align*}
\varphi_{q,\xi, \eta}(s\mid \infty) = \xi^s \frac{(\xi; q)_{\infty}}{(q;q)_s}.
\end{align*}

\subsubsection{The model of $q$-pushTASEP}\label{s.q-pushTASEPmodel} The $q$-pushTASEP is a discrete time interacting particle system on $\Z$ first introduced in \cite{matveev2016q}.
We have $N\in\N$ many particles which occupy distinct sites in $\Z$, and we label their position at time $T\in\Z_{\geq 0}$ in increasing order as $x_1(T) < x_2(T) < \ldots < x_N(T)$; we denote the collection of these random variables by $x(T)$. We also specify a collection of parameters $a_1, \ldots, a_N$ and $b_1, b_2, \ldots$, all lying in $(0,1)$.

The evolution from time $T$ to $T+1$ is as follows. The particle positions are updated from left to right: for $k\in\{1, \ldots, N\}$,
\begin{align*}
x_k(T+1) = x_k(T) + J_{k,T} + P_{k,T},
\end{align*}
where $J_{k,T}$ and $P_{k,T}$ are independent random variables with $J_{k,T}\sim q\text{-Geo}(a_kb_{T+1})$ (encoding a \emph{jump} contribution) and
$$P_{k,T} \sim \varphi_{q^{-1}, \xi= q^{\mrm{gap}_k(T)}, \eta=0}\bigl(\cdot \mid x_{k-1}(T+1) - x_{k-1}(T)\bigr),$$
(encoding a \emph{push} contribution) where $\mrm{gap}_k(T) = x_k(T) - x_{k-1}(T)-1$, $x_0(T) = -\infty$ by convention and, by a slight abuse of notation, $\sim$ means the LHS is distributed according to the measure which has probability mass function given by the RHS. In other words, $P_{k,T}$ is a $q$-deformed beta binomial random variable  with parameters $q^{-1}, \xi= q^{\mrm{gap}_k(T)}, \eta = 0$, and $m=x_{k-1}(T+1) - x_{k-1}(T)$. Note in particular that $x_1$'s motion does not depend on that of any other particle, i.e., marginally it follows a random walk. Further, the process is an exclusion process, i.e., particles always occupy distinct sites and also remain ordered.\footnote{Let $\Delta_{k-1}(T+1)=x_{k-1}(T+1) - x_{k-1}(T)$ and $j$ be the value the probability mass function in the RHS of the previous display is evaluated at. When $\Delta_{k-1}(T+1)\geq \mrm{gap}_k(T)$, the probability mass function is zero unless $j\geq \Delta_{k-1}(T+1) - \mrm{gap}_k(T)$, as otherwise the factor $(q^{\mrm{gap}_k(T)}; q^{-1})_{\Delta_{k-1}(T+1) - j}$ is zero. It is immediate from the definition that this implies $x_k(T)+ j\geq x_{k-1}(T+1) + 1$, thus ordering and exclusion are maintained.}

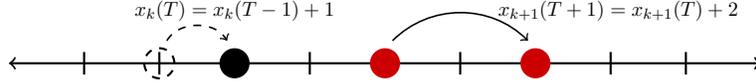
\begin{figure}
\begin{tikzpicture}
\draw[<->, thick] (-5,0) -- (5,0);

\foreach \x in {-4,..., 4}
\draw[thick] (\x, 0.15) -- ++(0,-0.3);

\fill[red!80!black] (0,0) circle (0.2cm);
\fill[red!80!black] (2,0) circle (0.2cm);

\draw[black, dashed, thick] (-3,0) circle (0.2cm);
\fill[black] (-2,0) circle (0.2cm);

\draw[->, dashed, semithick] (-2.9, 0.3) to[out=60, in=120] (-2.1,0.3);
\draw[->, semithick] (0.1, 0.3) to[out=45, in=135] (1.9,0.3);

\node[scale=0.7] at (3.1, 0.7) {$x_{k+1}(T+1) = x_{k+1}(T) + 2$};
\node[scale=0.7] at (-2, 0.7) {$x_{k}(T) = x_{k}(T-1) + 1$};
\end{tikzpicture}
\caption{A depiction of one step in the evolution of $q$-pushTASEP. The dotted circle is the position of the $k$\textsuperscript{th} particle at time $T-1$, and the solid black circle is it after it moves to its position at time $T$. The left red circle is the $(k+1)$\textsuperscript{th} particle at time $T$ and the right one the same at time $T+1$. The movement of the $k$\textsuperscript{th} particle in the previous step effects the $P_{k,T}$ contribution to the total jump size of $2$ of the $(k+1)$\textsuperscript{th} particle at time $T$.}
\end{figure}

This model is integrable. More precisely, the distribution of $x_N(T)$, started from a special initial condition known as \emph{step initial condition} where $x_k(0) = k$, can be related to a marginal of the $q$-Whittaker measure, a measure on partitions (equivalently, Young diagrams) defined in terms of $q$-Whittaker polynomials. This connection will be important for our arguments, and we will discuss it more in Section~\ref{s.intro.q-pushtasep to LPP}.

Apart from integrability, another reason $q$-pushTASEP is of interest is because it degenerates to other well-known models. Indeed, in the $q\to 1$ limit, when appropriately renormalized, $x_N(T)$ converges to the free energy of the log-gamma polymer model. While our results will not carry over to this limit, we will make some further remarks about this relationship between the models in Section~\ref{s.log gamma remark}. We also mention that after speeding up time and letting the jump rates go to zero, $q$-pushTASEP converges to the continuous time $q$-pushTASEP \cite{borodin2016nearest} which, when $q=0$, is the well-known pushTASEP \cite{borodin2008large}.

\subsubsection{Law of large numbers and asymptotic Tracy-Widom fluctuations of $q$-pushTASEP}

In this work we will focus on $q$-pushTASEP when the parameters are equal, i.e., $u = a_i = b_j$ for all $i=1, \ldots, N$ and $j=1, 2, \ldots$ for some $u\in(0,1)$, and when the initial condition is $x_k(0) = k$ for $k=1, \ldots, N$. In this setting, and under some additional restrictions on the parameters, \cite{vetHo2022asymptotic} proved a law of large numbers for $x_N(T)$, which in the $T=N$ case states (with the convergence being in probability) that
\begin{align}\label{e.fq definition}
\lim_{N\to\infty} \frac{x_N(N)}{N} = 2\cdot\frac{\psi_q(\log_q u) + \log(1-q)}{\log q} + 1 =: f_q;
\end{align}
here $\log_q u = \log u/\log q$ is the logarithm to the base $q$ and $\psi_q$ is the $q$-digamma function, given by
\begin{equation}\label{e.q-digamma}
\psi_q(x) = \frac{1}{\Gamma_q(x)}\frac{\partial \Gamma_q(x)}{\partial x},
\end{equation}
where $\Gamma_q(x) = \frac{(q;q)_\infty}{(q^x;q)_{\infty}}(1-q)^{1-x}$ is the $q$-gamma function.

Note that our definition of $q$-pushTASEP differs from that of \cite{vetHo2022asymptotic} and \cite{matveev2016q}, in that particles move to the right for us rather than the left, thus introducing an extra negative sign in the law of large numbers. Our definition agrees with the one given in \cite{imamura2022solvable}.

\cite{vetHo2022asymptotic} also proves that the asymptotic fluctuation of $x_N$ converges to the GUE Tracy-Widom distribution. To state this, let us consider the rescaled observable
\begin{align}\label{e.definition of X^sc}
\resc = \frac{x_N(N) -f_qN}{(-\psi_q''(\log_q u))^{1/3}(\log q^{-1})^{-1}N^{1/3}};
\end{align}
note that the denominator is a positive quantity, since $\psi_q''(x)<0$ for all $x>0$ (see e.g. \cite{mansour2009some}). 
Now for $q,u\in(0,1)$, and under certain restrictions on those parameters that are used to simplify the analysis there,
\cite[Theorem~2.2]{vetHo2022asymptotic} asserts that $\resc \Rightarrow F_{\mrm{GUE}}$, where $F_{\mrm{GUE}}$ is the GUE Tracy-Widom distribution. 

The proof given in \cite{vetHo2022asymptotic} relies on certain formulas for $q$-Laplace transforms of particle positions proved in \cite{borodin2015height}. The recent work \cite{imamura2022solvable} gives different Fredholm determinant formulas for randomly shifted versions of $x_N(T)$ (see Corollary 5.1 there) from which it should also be possible to extract the above distributional convergence, with a perhaps simpler analysis; indeed, the analogous convergence is demonstrated for a half-space version of the model in \cite[Theorem~6.11]{imamura2022solvable}.

It remains a question whether the conditions assumed in \cite{vetHo2022asymptotic} are necessary for this convergence to hold; our techniques suggest it should hold for any $q, u\in(0,1)$. In particular, our results will hold for all $q,u\in(0,1)$.

\subsection{Main results}

Our main theorem bounds the lower tail of the fluctuations of the centred and scaled position $\resc$ (as defined in \eqref{e.definition of X^sc}) of the $N$\textsuperscript{th} particle of $q$-pushTASEP, as introduced in Section~\ref{s.q-pushTASEPmodel}.

\begin{theorem}\label{mt.q-pushtasep bound}
Let $q, u\in(0,1)$ and let $a_i=b_j = u$ for all $i,j$. There exist positive absolute constants $c'$, C, and $N_0$ (independent of $q$ and $u$) such that, with $\theta_0 = C(1-u)^{-1/3}(1\vee (\log q^{-1})^{-2/3})$ and $c = (1-u)^{1/2}c'$, and for $N\geq N_0$ and $\theta>\theta_0$,
\begin{align*}
\P\left(\resc <  -\theta\right) \leq \exp\bigl(-c\theta^{3/2}\bigr).
\end{align*}

\end{theorem}

We believe the true lower tail behavior to be $\exp(-c\theta^3)$, at least for $\theta \ll N^{2/3}$, i.e., smaller than the large deviation regime, similar to other models in the KPZ class. We discuss ahead in Remark~\ref{r.geometric lpp tail exponent} why our arguments do not achieve this, and also how with additional different arguments it should be possible to attain the full exponent of $3$.

It is an interesting question whether Theorem~\ref{mt.q-pushtasep bound} can be extended to the case of general $a_i$ and $b_j$. As we will see, a crucial connection to last passage percolation that we rely on continues to hold, so the general scheme is broadly applicable. However certain moment formulas available for the last passage percolation problem in the case of homogeneous parameters are not known for the general case, and this is where the strategy stops being directly applicable.

\subsection{Lower tails of KPZ observables} \label{s.intro.prior work}

In the past decade, integrable tools have been combined with other perspectives to slowly push out of strictly integrable settings. Examples include studies of geometric properties in models of last passage percolation (e.g. \cite{J00,slow-bond,coalescence,riddhi-aging,BHS18,basu2021temporal,watermelon,sarkar2021infinite,schmid2022mixing}),  process-level regularity properties (e.g. \cite{corwin2014brownian,corwin2016kpz,hammond2016brownian,hammond2017patchwork,calvert2019brownian,sarkar2020brownian,dauvergne2023wiener}) of processes whose finite-dimensional distributions are accessible via exactly solvable tools, recent progress on constructing the ASEP speed process \cite{aggarwal2022asep}, edge and bulk scaling behavior of tiling or dimer models \cite{aggarwal2019universality,huang2021edge,aggarwal2021edge,huang2023pearcey}, as well as the construction of the directed landscape and convergence of LPP models to it \cite{dauvergne2018basic,dauvergne2018directed,dauvergne2021scaling}. In these works, a crucial input from the integrable side has repeatedly been bounds on the tails of the relevant statistic. In the following, by zero temperature models we will mean models which can be embedded in determinantal point processes, e.g. integrable models of last passage percolation.

In the case of zero temperature models (e.g. geometric or exponential last passage percolation or the totally asymmetric simple exclusion process), these tail bounds had been studied two decades ago, with arguments relying in an essential way on determinantal structure possessed by these models. In positive temperature (e.g. the KPZ equation or the asymmetric simple exclusion process), where such structure is not directly available, progress on obtaining these important tail inputs has only been made in the last few years, but their availability promises to create the opportunity to bring the zero temperature successes to the positive temperature.

\subsubsection{The relative difficulty of upper and lower tail bounds} From a physical perspective, it is easy to see that the upper and lower tails should have different rates of decay, with the lower tail decaying faster.  This is because the upper tail concerns making a single, ``largest'' object even larger---in $q$-pushTASEP, making the right-most particle lie even further to the right, which can be accomplished by demanding a single large jump of the right-most particle. So, in particular, the other particles are not a barrier. In contrast, in the lower tail exactly the opposite happens: for the right-most particle to lie atypically to the left, \emph{all} the other particles must also do so---in particular, many jumps, including those of other particles, must be suppressed. However, this intuition does not reveal the fact that, typically, it is technically much more challenging to obtain lower tail bounds than upper ones. Further, while this intuition turns out to be well-suited for arguments to understand large deviations behavior (e.g. \cite{basu2017upper} for LPP), i.e., deviations on scale $N$, in applications one needs bounds on the fluctuation scale (i.e., deviations on scale $N^{1/3}$).

For solvable zero temperature models, which have determinantal descriptions, the difference in the difficulty of upper and lower tails can be seen from the fact that the upper tail bounds follow directly from bounds on the kernel of the associated determinantal point process, while this is not the case for lower tail bounds.
Nevertheless, as mentioned above, in these models, a number of approaches have been developed over the last two decades. These include the Riemann-Hilbert approach (e.g. \cite{baik2001optimal}), methods based on explicit formulas for moments (e.g. \cite{ledoux2005deviation}), connections to random matrix theory (e.g. \cite{johansson2000shape,ledoux2010,basu2019lower,ramirez2011beta} as well as large deviation work e.g. \cite{guionnet2021asymptotics,augeri2021large,cook2023full}), and abstract concentration (see e.g. \cite[Section 3]{auffinger201750} for applications to first passage percolation).

The toolbox for the upper tail is already fairly well developed in positive temperature (i.e., non-determinantal but exactly solvable models). Here too one often has determinantal formulas for the distribution of the observable, and one can extract the upper tail by establishing decay of the kernel in these determinantal formulas. An instance where this is done is \cite[Theorem~1.4]{barraquand2021fluctuations}, in the context of the log-gamma polymer. Besides this approach, one can also try to extract upper tail estimates from the moments of the exponential of the random variable of interest (e.g. for the KPZ equation, this corresponds to moments of the stochastic heat equation, or, for our model, the analogue is $q$-moments). An example of this method is captured in \cite[Proposition~4.3 and Lemma~4.5]{corwin2018kpz}, where tail estimates for the narrow-wedge KPZ equation are obtained through estimates on the $k$\textsuperscript{th} moment of the stochastic heat equation. We also mention recent works \cite{ganguly2022sharp} and \cite{landon2022tail} which respectively make use of Gibbs properties and special structure of stationary versions of the relevant models (the KPZ equation and O'Connell-Yor polymer respectively) to prove upper tail estimates, but these methods by their nature are specific to models which have such probabilistic structure.

For our model of $q$-pushTASEP, a Fredholm determinant formula of the type the first approach relies on can be found in \cite[Theorem 3.3]{borodin2015height} (via our model's connection to the $q$-Whittaker measure, see Section~\ref{s.intro.q-pushtasep to LPP} ahead or the discussion in \cite[Section~7.4]{matveev2016q}), and a different one in \cite[Corollary 5.1]{imamura2022solvable}. A formula for the $q$-moments in our model is also available, though there is a subtlety in that not all $q$-moments are finite; see Section 7.4 in \cite{matveev2016q} for the formula and a brief discussion of this point.

Having said this, while there are well-established approaches to obtain such upper tail estimates, it is certainly not a triviality to actually do so in any model, and we do not pursue them for $q$-pushTASEP in this work. However, we plan to revisit this question as part of subsequent work in which we will need both tail bounds to study further aspects of this model.

The toolbox for the lower tail in positive temperature is smaller but is being actively developed, and we briefly review some of the tools now. However, these do not seem applicable to our model, and so we are ultimately led to develop a new technique. %

\subsubsection{Work using determinantal representations of Laplace transforms} The first class of techniques for lower tails in positive temperature models gives a determinantal representation for the Laplace transform (or $q$-Laplace transform) of the observable. This approach was initiated in \cite{corwin2020lower}, which obtained fluctuation-scale lower tail bounds for the narrow wedge solution to the KPZ equation. \cite{corwin2020lower} used a formula from \cite{borodin2016moments,amir2011probability} which equates the Laplace transform of the fundamental solution of the stochastic heat equation (which is related to the KPZ equation via the Cole-Hopf transform) to an expectation of a multiplicative functional of the Airy point process, which is determinantal. That it is a multiplicative functional (as well as its precise form) is very useful as it allows the lower tail of the KPZ equation to be bounded in terms of the lower tail behavior of the particles at the edge of the Airy point process, which in turn can be controlled via determinantal techniques as outlined above.

The Laplace transform identity that this argument relies on can be seen as a special case of a general matching proved in \cite{borodin2018stochastic} between the stochastic six vertex model's height function and a multiplicative functional of the row lengths of a partition sampled according to the Schur measure. The stochastic six vertex model and the Schur measure are each known to specialize to a number of models also of interest; for example, one degeneration of the former is the asymmetric simple exclusion process (ASEP), and the analogous one for the latter is the discrete Laguerre ensemble, a determinantal process. This yields an identity between the $q$-Laplace transform of ASEP and a multiplicative functional of the discrete Laguerre ensemble \cite{borodin2017asep}, which was used to obtain a lower tail bound for the former in \cite{aggarwal2022asep}, using the latter's connection to TASEP.

Unfortunately, not all degenerations to models of interest play nicely on both sides. For instance, for the O'Connell-Yor polymer, the Schur measure side of the stochastic six vertex model identity degenerates to an average of a multiplicative functional with respect to a point process whose measure is a \emph{signed} measure instead of a probability measure \cite{imamura2016determinantal}. Typical modes of analysis break down in the context of signed measures. For other models too, including ours, this issue of signed measures seems to arise.

A related recent approach brings in the machinery of Riemann-Hilbert problems and has been developed in \cite{cafasso2022riemann}. There, in the setting of the KPZ equation, the mentioned expectation of the multiplicative functional of the Airy point process is expressed as a Fredholm determinant, and then the latter is written as a Riemann-Hilbert problem. So far this approach has only been developed at the level of the KPZ equation, and so it remains to be seen how broadly it can be applied.

\subsubsection{Coupling and geometric methods in polymer models}\label{s.recent work in polymer}  
For the semi-discrete O'Connell-Yor and log-gamma polymer models, recent work \cite{landon2022upper,landon2022tail} has obtained lower tail estimates via a mixture of exact formulas, coupling arguments, and geometric considerations. This builds on methods developed in the zero temperature model of exponential last passage percolation \cite{emrah2020right,emrah2021optimal,emrah2022coupling}. The program has so far been implemented in full in the semi-discrete O’Connell-Yor model and in part for the log-gamma polymer. First, \cite{landon2022upper} obtains the bounds for a stationary version of the model (where one can prove an explicit formula for the Laplace transform of the free energy), and then, for the O'Connell-Yor case, these are translated to the original model using geometric considerations of the polymer measure in \cite{landon2022tail}. In fact, the
Laplace transform bound obtains a lower tail exponent of 3/2, which is then upgraded to the sharp exponent of 3 by adapting geometric arguments from \cite{ganguly2020optimal}. Since this method relies heavily on the polymer geometry, it is unclear how it could be extended to address the model of $q$-pushTASEP which only has a particle interpretation.

In summary, while there are a variety of methods in zero-temperature models to obtain lower tail bounds, so far only a handful of tools are available for positive temperature models. The ones available do not seem immediately applicable to our model. For this reason, we introduce a new method which does not rely on polymer structure or identities between $q$-Laplace transforms and multiplicative functionals of determinantal point processes, which are not directly available in $q$-pushTASEP. We rely instead on the recent work \cite{imamura2021skew} which relates the $q$-Whittaker measure on partitions to a model of periodic geometric last passage percolation. In this way, we are able to use both the polymer techniques and determinantal structure which \emph{are} available in geometric last passage percolation to analyze $q$-pushTASEP. Our methods may also be useful in studying the lower tails of $q$-pushTASEP with certain other special initial data that also have connections to marginals of $q$-Whittaker measures, or $q$-pushTASEP with particle creation \cite{barraquand2020half,imamura2022solvable} which should have a description in terms of a similar last passage percolation problem via half-space $q$-Whittaker measures.

To explain our broad approach, we next describe this model of periodic last passage percolation.

\subsection{Last passage percolation}\label{s.lpp}

We first describe the environment in which our last passage percolation (LPP) problem will exist. We consider a sequence of $N\times T$ ``big rectangles'' indexed by $k\in\N\cup\{0\}$, each of which contains $NT$ ``small squares'' inside. These are arranged in a periodic strip as shown in Figure~\ref{f.infinite LPP}. We use the coordinates $(i,j;k)$ (with $(i,j)\in\{1, \ldots, N\}\times\{1, \ldots, T\}$ and $k\in\{0,1, \ldots, \}$) to denote the small square with coordinates $(i,j)$ in the $k$\textsuperscript{th} big square. The site $(i,j;k)$ is associated with an independent non-negative random variable $\smash{\xi_{(i,j;k)}}$ which we call a \emph{site weight}.

The distribution of the randomness of the site weights is as follows. We will say $X\sim\mrm{Geo}(z)$ if $X$ is a random variable such that $\P(X\geq k) = z^k$ for $k=0,1,2 \ldots $; in other words, $z$ is the \emph{failure} probability in repeated independent trials and $X$ is the number of failures before the first success. Then the site weights are specified as follows: $\smash{\xi_{(i,j;k)}}$ are independent across all $i,j,k$, and distributed as $\mrm{Geo}(a_ib_j q^k)$ for $k=0,1,2, \ldots$ and $(i,j)\in\{1, \ldots, N\}\times \{1, \ldots, T\}$.

Note that in our model with the specialization $u=a_i=b_j$, the site weights in the same big square, i.e., with the same value of $k$, are identically distributed as $\mrm{Geo}(u^2q^k)$. For $s = (i,j;k)$ a site in the strip, we may also write $\xi_s$ for $\smash{\xi_{(i,j;k)}}$.

\begin{remark}\label{r.ultimately zero}
Observe that $\P(\xi_{(i,j;k)} \neq 0) = a_ib_jq^k$ for all $i$, $j$, $k$, which is summable over $(i,j)\in\{1, \ldots, N\}\times\{1, \ldots, T\}$ and $k = 0,1, \ldots$. So by the Borel-Cantelli lemma, almost surely, for all large enough $k$ and all $(i,j)\in\{1, \ldots, N\}\times\{1, \ldots, T\}$, $\smash{\xi_{(i,j;k)}}$ will be zero.
\end{remark}

We consider downward paths which are allowed to wrap around the strip, i.e., paths which are at $(i,T;k)$ may move to $(i, 1; k+1)$ and paths which are at $(N,j;k)$ may move to $(1, j; k+1)$ in the next step; again see Figure~\ref{f.infinite LPP}. Each such path $\gamma$ is assigned a weight $w(\gamma)$ given by $\sum_{v\in\gamma} \xi_{v}$. Note that while a priori the weight of $\gamma$ could be infinite if $\gamma$ is an infinite path, in our setting the environment will only have finitely many non-zero site weights almost surely (as noted above in Remark~\ref{r.ultimately zero}), and so this possibility will not arise.

Now, the last passage value $L_{v,w}$ between $v$ and $w$ small squares in the strip is defined as
$$L_{v,w} := \max_{\gamma: v\to w} w(\gamma),$$
where the maximum is over all downward paths from $v$ to $w$, assuming at least one such path exists. If not, we define $L_{v,w}$ to be $-\infty$; we say this only to give a logically complete definition, but such cases will not actually arise in this paper.

This model of LPP is similar to other models of periodic LPP considered in the literature (e.g. \cite{baik2018fluctuations,baik2019multipoint,baik2021periodic,betea2021peaks,schmid2022mixing} though perhaps with slightly different selections of parameters or distributions of the random variables), and also has connections to the periodic Schur measure \cite{borodin2007periodic,betea2021peaks}.

\begin{figure}[h]
  \begin{center}
  \begin{tikzpicture}[scale=0.56]
  \newcommand{\sidelength}{2}

  \begin{scope}[shift={(10,1.75)}]

    \begin{scope}
    \clip (-\sidelength-1.7, 0.1) rectangle (\sidelength+2, -7);

    \fill[yellow, opacity=0.25] (-\sidelength, -\sidelength) -- (0,0) -- (\sidelength, -\sidelength) -- (0, -2*\sidelength) -- cycle;

    \fill[orange, opacity=0.25] (\sidelength, -\sidelength) -- ++(0,-4) -- ++(-2, 2) -- cycle;
    \fill[orange, opacity=0.25] (-\sidelength, -\sidelength) -- ++(0,-4) -- ++(2, 2) -- cycle;

    \fill[purple!50!blue, opacity=0.25] (-\sidelength, -3*\sidelength) -- ++(\sidelength,-\sidelength) -- ++(\sidelength, \sidelength) -- ++(-\sidelength, \sidelength) --cycle;

    \fill[green!40, opacity=0.25] (-\sidelength, -3*\sidelength) -- ++(0,-1) -- ++(1,0) -- cycle;
    \fill[green!40, opacity=0.25] (\sidelength, -3*\sidelength) -- ++(0,-1) -- ++(-1,0) -- cycle;

    \draw[thick] (-\sidelength,-7) -- (-\sidelength,-\sidelength) -- (0,0) -- (\sidelength,-\sidelength) -- (\sidelength,-7);

    \draw[thick] (-\sidelength, -\sidelength) --coordinate[at end](R1) ++(2*\sidelength, -2*\sidelength);
    \draw[thick] (\sidelength, -\sidelength) --coordinate[at end](L1) ++(-2*\sidelength, -2*\sidelength);

    \draw[thick,dashed] (R1) -- ++(-2*\sidelength, -2*\sidelength);
    \draw[thick,dashed] (L1) -- ++(2*\sidelength, -2*\sidelength);

    \foreach \i [evaluate=\i as \x using \i-0.5] in {1, 1.5,...,\sidelength}
    {
      \draw[opacity=0.6] (-\x, -\x) -- ++(\sidelength+\x, -\sidelength-\x) -- ++(-2*\sidelength,-2*\sidelength);
      \draw[opacity=0.6] (\x, -\x) -- ++(-\sidelength-\x, -\sidelength-\x) -- ++(2*\sidelength,-2*\sidelength);
    }

    \end{scope}

    \draw[very thick, green!60!black] (0,-0.5) -- ++(1,-1) -- ++(-1, -1) -- ++(2, -2);
    \draw[very thick, green!60!black] (-\sidelength, -4.5) -- ++(1.5,-1.5) -- ++(-0.5, -0.5) -- ++(0.5, -0.5);

    \node[scale=0.8] (first-label) at (\sidelength+0.3, -1) {$\mathrm{Geo}(u^2)$};
    \node[scale=0.8] (second-label) at (\sidelength+1.4, -4) {$\mathrm{Geo}(u^2q)$};
    \node[scale=0.8] (third-label) at (-\sidelength-1.4, -6) {$\mathrm{Geo}(u^2q^2)$};

    \draw[->, semithick] (first-label) to[out=180, in=90] (0, -1.5);
    \draw[->, semithick] ($(second-label)+(-0.1,0.2)$) to[out=120, in=80] (\sidelength-1, -4.5);
    \draw[->, semithick] (third-label) to[out=60, in=90] (0, -5.5);

    \draw[thick, dashed, <-] (40:2cm and 0.4cm) arc   (40:-220:2cm and 0.4cm);

  \end{scope}

  \begin{scope}[shift={(9in, 0.66in)}]

	\begin{scope}
	\clip (-\sidelength-1.7, 0.1) rectangle (\sidelength+2.5, -7);

	\fill[yellow, opacity=0.25] (-\sidelength, -\sidelength) -- (0,0) -- (\sidelength+0.5, -\sidelength-0.5) -- ++(-2, -2) -- cycle;

	\fill[orange, opacity=0.25] (\sidelength+0.5, -\sidelength-0.5) -- ++(0,-4) -- ++(-2, 2) -- cycle;
	\fill[orange, opacity=0.25] (-\sidelength, -\sidelength) -- ++(0,-4) -- ++(0.5,-0.5) -- ++(2, 2) -- ++(-2.5, 2.5) -- cycle;

	\fill[purple!50!blue, opacity=0.25] (\sidelength+0.5, -\sidelength-4.5) -- ++(0,-0.5) -- ++(-3.5, 0) -- ++(-0.5, 0.5) -- ++(2,2) --cycle;
	\fill[purple!50!blue, opacity=0.25] (-\sidelength, -\sidelength-4) -- ++(0,-1) -- ++(0.5, 0.5) --cycle;

	\fill[green!40, opacity=0.25] (-\sidelength, -\sidelength-5) -- ++(1,0) -- ++(-0.5,0.5) -- cycle;

	\draw[thick] (-\sidelength,-7) -- (-\sidelength,-\sidelength) -- (0,0) -- (\sidelength+0.5,-\sidelength-0.5) -- (\sidelength+0.5,-7);

	\draw[thick] (-\sidelength, -\sidelength) --coordinate[at end](R1) ++(2*\sidelength+0.5, -2*\sidelength-0.5);
	\draw[thick] (\sidelength+0.5, -\sidelength-0.5) --coordinate[at end](L1) ++(-4.5, -4.5);

	\draw[thick] (-\sidelength, -\sidelength-4) -- ++(1, -1);

	\foreach \i [evaluate=\i as \x using \i] in {0.5, 1,...,\sidelength}
	{
		\draw[opacity=0.6] (-\x, -\x) -- ++(\sidelength+\x+0.5, -\sidelength-\x-0.5) -- ++(-2*\sidelength,-2*\sidelength);
		\draw[opacity=0.6] (\x, -\x) -- ++(-\sidelength-\x, -\sidelength-\x) -- ++(2*\sidelength,-2*\sidelength);
	}
	\end{scope}

	\node[scale=0.8] (first-label) at (\sidelength+1.8, -1.8) {$\mathrm{Geo}(a_2b_4)$};
	\node[scale=0.8] (second-label) at (\sidelength+1.9, -4) {$\mathrm{Geo}(qa_3b_1)$};
	\node[scale=0.8] (third-label) at (-\sidelength-1.85, -6) {$\mathrm{Geo}(q^2a_2b_2)$};

	\draw[->, semithick] (first-label) to[out=180, in=90] (1, -2.5);
	\draw[->, semithick] ($(second-label)+(-0.1,0.2)$) to[out=120, in=80] (\sidelength-0.5, -4);
	\draw[->, semithick] (third-label) to[out=60, in=90] (0.5, -6);

	\draw[<->, dotted, semithick] (0.2,0.2) -- (\sidelength+0.5+0.2, -\sidelength-0.5+0.2);
	\draw[<->, dotted, semithick] (-0.2,0.2) -- (-\sidelength-0.2, -\sidelength+0.2);

	\node[anchor=south] at (1.8, -1.7+0.3) {$T$};
	\node[anchor=south] at (-1.8, -1.7+0.3) {$N$};

	\draw[thick, dashed, <-] (40:2.8cm and 0.3cm) arc   (40:-220:2.35cm and 0.4cm);

  \end{scope}

  \end{tikzpicture}
  \end{center}
  \caption{The environment in which the infinite last passage percolation occurs. The dashed arrow on top indicates the direction in which the squares wrap around, and the solid green line on the left is a downward path which wraps around the strip. In the left panel we have the model under consideration in this article with the specialization of $T=N$ and $u=a_i=b_j$ for all $i,j$, while on the right is the general model.}\label{f.infinite LPP}
\end{figure}
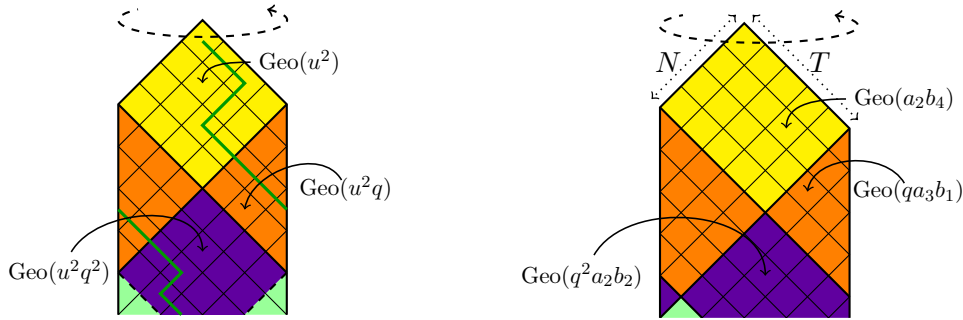

\subsection{Relation between $q$-pushTASEP and LPP}\label{s.intro.q-pushtasep to LPP}
We can now explain the exactly solvable connection between $q$-pushTASEP and the model of LPP in an infinite periodic strip just introduced that was recently discovered by Imamura-Mucciconi-Sasamoto, and on which our arguments crucially rely. We first state the equivalence precisely.

While this paper only considers the case $T=N$, we will state the LPP equivalence for general $T$. For this, we return to the LPP problem described in Section~\ref{s.lpp} where the infinite periodic environment has ``fundamental domain'' with $N\times T$. The parameters of the geometric random variables is as described there, i.e., site $(i,j;k)$ has parameter $a_ib_jq^k$.

\begin{theorem}\label{t.q-pushTASEP to lpp}
Let $L$ be the LPP value in the environment just described with $a_i,b_j\in(0,1)$ for all $(i,j)\in\{1, \ldots, N\}\times\{1, \ldots, T\}$. Let $x_N(T)$ be the position of the $N$\textsuperscript{th} particle at time $T$ in $q$-pushTASEP with the same parameters $a_i, b_j$ for all $(i,j)\in\{1, \ldots, N\}\times\{1, \ldots, T\}$ and step initial condition. Then
\begin{align*}
x_N(T) \stackrel{d}{=} L + N .
\end{align*}
\end{theorem}

The connection between $x_N(T)$ and $L$ runs through a measure on partitions known as the $q$-Whittaker measure, as mentioned above. More precisely, it was shown in \cite{matveev2016q} (and stated ahead as Theorem~\ref{t.q-pushTASEP and whittaker}) that $x_N(T)$, for any $N,T\in\N$ and with step initial condition (i.e., $x_k(0)= k$ for $k=1, \ldots, N$), is distributed as the length of the top row in a partition (encoded as a Young diagram) distributed according to the $q$-Whittaker measure (after a deterministic shift by $N$). So it remains to establish a distributional equality between the $q$-Whittaker measure's top row and the LPP value $L$; the proof of this was explained to us by Matteo Mucciconi and relies on results from \cite{imamura2021skew} and we discuss it next.

\cite{imamura2021skew} proves a relation between the $q$-Whittaker measure and the periodic LPP model. From the perspective of the needs of this paper, the main consequence of \cite{imamura2021skew} is the development of a bijection which generalizes the Robinson-Schensted-Knuth correspondence. In traditional LPP on $\Z^2$, the RSK correspondence associates to the LPP environment (with non-negative integer site weights) a pair of Young tableaux of the same shape, with the property that LPP statistics are encoded in the row lengths of the tableaux; for instance, the LPP value is exactly the length of the top row of the tableaux. It is well-known that, under this correspondence, the measure on the environment given by \iid geometric random variables gets pushed forward to give the Schur measure on partitions, i.e., Young diagrams.
The generalization of RSK established in \cite{imamura2021skew}, called skew RSK there, relates pairs of \emph{skew} Young tableaux to pairs of \emph{vertically strict tableaux} (tableaux where the ordering condition on the entries is imposed only along columns and not rows) along with some additional data.

In this bijection, LPP has not had a role to play. To involve LPP, we recall an earlier generalization of RSK known as the Sagan-Stanley correspondence \cite{sagan1990robinson}, which can be interpreted as giving a bijection between the LPP environment in an infinite strip (again with non-negative integer site weights) and pairs of skew Young tableaux. \cite{imamura2021skew} also shows that, if one composes this bijection with the skew RSK bijection, then the LPP value is exactly the length of the top row of the vertically strict tableaux coming from the skew RSK. 

It turns out that the generating function of vertically strict tableaux can be written in terms of the $q$-Whittaker polynomials. Using this fact, certain weight preservation properties of the bijection, and an argument similar to the well-known one that  establishes the above mentioned relationship between geometric LPP and the Schur measure, one can show that the LPP value when the infinite strip has site weights given by independent geometric variables with parameter specified above has the same distribution as the top row of a random partition from the $q$-Whittaker measure. Since this statement is not recorded explicitly in \cite{imamura2021skew}, we will give a proof using results from that paper in Appendix~\ref{app.q-whittaker and lpp}. In fact, one can relate the lengths of all the rows of the partition to LPP values involving multiple disjoint paths, and we prove this stronger statement in Theorem~\ref{t.full q-whittaker to lpp}. Theorem~\ref{t.q-pushTASEP to lpp} then follows by combining this theorem with Theorem~\ref{t.q-pushTASEP and whittaker} (relating $x_N(T)$ and the top row of the $q$-Whittaker measure).

\subsection{Proof ideas}\label{s.proof ideas}

We specialize to $N=T$ and $a_i =b_j = u\in(0,1)$ for $(i,j)\in\{1, \ldots, N\}\times\{1, \ldots, T\}$. 

To summarize, $x_N(N)$ is, up to a deterministic shift by $N$, the LPP value in an infinite periodic environment of inhomogeneous geometric random variables. Now, the weight of \emph{any} path in this environment is a lower bound on the last passage percolation value. We consider a specific path which allows us to utilize the homogeneity of the geometric variable parameters inside a big square (as well as tail information of geometric LPP in such homogeneous squares) along with the independence across big squares.

More specifically, we consider the path formed by concatenating paths from the top to bottom of the big squares along the center, i.e., the squares in which the geometric parameter is $u^2q^{2i}$ for some $i\in\N\cup\{0\}$. More precisely, we do not exactly concatenate the paths as they do not have a common site; we simply consider the sum of the weights of the paths, ignoring the positive weight of the extra site needed to actually join the paths. See Figure~\ref{f.concatenation}.

\begin{figure}
\begin{tikzpicture}

\fill[yellow, opacity=0.25] (0,0) -- ++(-1,1) -- ++(1,1) -- ++(1,-1)  --cycle;

\fill[orange, opacity=0.25] (0,0) -- ++(1,1) -- ++(1,-1) -- ++(-1,-1)  --cycle;
\fill[orange, opacity=0.25] (0,0) -- ++(-1,1) -- ++(-1,-1) -- ++(1,-1)  --cycle;

\fill[purple!50!blue, opacity=0.25] (0,0) -- ++(-1,-1) -- ++(1,-1) -- ++(1,1)  --cycle;

\draw[thick] (-1,1) -- (1,-1);
\draw[thick] (1,1) -- (-1,-1);

\draw[opacity=0.6] (-1.5, -0.5) -- ++(2,2);
\draw[opacity=0.6] (-0.5, -1.5) -- ++(2,2);
\draw[opacity=0.6] (-1.5, 0.5) -- ++(2,-2);
\draw[opacity=0.6] (-0.5, 1.5) -- ++(2,-2);

\draw[very thick, green!60!black] (0, 1.5) -- ++(-0.25,0.25);
\draw[very thick, green!60!black] (0, 1.5) -- ++(0.5,-0.5) -- ++(-0.5, -0.5);

\draw[very thick, green!60!black] (0, -0.5) -- ++(-0.5,-0.5) -- ++(0.75, -0.75);

\draw[very thick, green!60!black, dashed] (0, 0.5) -- ++(0.5,-0.5) -- ++(-0.5, -0.5);
\end{tikzpicture}
\caption{A depiction of the paths we consider near the boundary between different big squares. The two solid green paths go from the topmost site to the bottommost site in their respective big squares, where the environment is homogeneous. We do not consider the dotted green path needed to connect them, which is valid for proving an upper bound on the lower tail since including its weight will only increase the overall weight.}
\label{f.concatenation}
\end{figure}
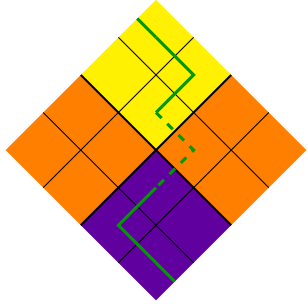

Let us calculate the law of large numbers of this path, i.e., its weight up to first order, using the knowledge of the LLN for geometric LPP. Indeed, in an $N\times N$ square with geometric parameter $u^2q^{2i}$, to first order in $N$, the LPP value is $2N\cdot\frac{uq^{i}}{1-uq^i}$ (see for example \cite{johansson2000shape} or Theorem~\ref{t.uniform lower tail} ahead), so that the overall LPP value of the path we have described is, again to first order,
\begin{align}\label{e.first form of LLN}
2N\cdot\sum_{i=0}^{\infty}\frac{uq^{i}}{1-uq^i} = 2N\cdot\sum_{i=0}^{\infty}\frac{q^{i+\log_q u}}{1-q^{i+\log_q u}}.
\end{align}
To evaluate this sum we need the $q$-digamma function $\psi_q$, defined in \eqref{e.q-digamma}.
Now, the $q$-digamma function $\psi_q$ is related to the sum in \eqref{e.first form of LLN} by the formula
$$\psi_q(x) = -\log(1-q)+\log q\cdot\sum_{i=0}^\infty \frac{q^{i+x}}{1-q^{i+x}}.$$
From this we see that \eqref{e.first form of LLN}  equals
\begin{align}\label{e.lln expression}
2N\cdot\frac{\psi_q(\log_q(u)) + \log(1-q)}{\log q} = N(f_q-1),
\end{align}
which is the first order term in the probability in Theorem~\ref{mt.q-pushtasep bound} (remember that $L$ and $x_N(N)$ differ by a constant term of $N$) and matches the LLN proved in \cite{vetHo2022asymptotic}.

\subsubsection{Uniform LPP control}

We have identified a concatenation of LPP problems which obtains the correct first order behaviour. Now, the order of fluctuations of geometric LPP of parameter $u^2q^{2i}$ in an $N\times N$ square is $u^{1/3}q^{i/3}(1-u^{2}q^{2i})^{-1}N^{1/3}$ (again see for example \cite{johansson2000shape} or Theorem~\ref{t.uniform lower tail} ahead). One can think of these fluctuations, once rescaled by this expression, as being approximately distributed according to the GUE Tracy-Widom distribution. The latter has a negative mean and, a calculation as in the previous subsection shows that the accumulated loss on the fluctuation scale across all the big squares is finite, in particular of order $(\log q^{-1})^{-1}|\log \log q^{-1}| N^{1/3}$ (ignoring the dependence on $u$). This means that if we can control the geometric LPP values across all the squares and use appropriate tools on concentration of sums of independent random variables, we will obtain a lower tail inequality for $x_N(N)$.

(In fact, the true behavior of $x_N(N)$ should be $f_q N - \Theta((\log q^{-1})^{-1}N^{1/3})$, i.e., the fluctuation term should not have the $\log\log$ factor. Our approach does not seem able to achieve this, and we discuss this more ahead in Section~\ref{s.log gamma remark}, along with some of the consequences of the appearance of the extra $\log\log$ factor.)

To apply concentration inequalities, we will need control over all the constituent geometric LPP problems. Observe that as the big squares get farther into the environment, the parameter $u^2q^{2i}$ of the geometric random variables goes to zero. So, in fact, we need a tail bound on the geometric LPP problems which is \emph{uniform} in the parameter $q$ essentially in the entire range $(0,1)$.

Now, the literature contains extremely sharp estimates on the upper and lower tails of geometric LPP for any fixed parameter $q$ \cite{baik2001optimal}. Unfortunately, these estimates are not stated uniformly in $q$ in the required range, and the method of proof does not seem like it would yield such an estimate. Indeed, the arguments rely on steepest descent analysis of contour integrals, and the resulting contours implicitly depend on $q$, thus making it difficult to extract uniform-in-$q$ estimates. Thus we need to prove new results. The following is our second main result and obtains a uniform lower tail in the entire parameter range of $q$. Here,
\begin{equation}\label{e.mu_q definition}
\mu_q = \frac{(1+q^{1/2})^2}{1-q}.
\end{equation}

\begin{theorem}\label{t.uniform lower tail}
Let $T_N$ be the LPP value from top to bottom of an $N\times N$ square in an environment given by \iid Geo($q$) random variables. There exist positive constants $c$, $x_0$, and $N_0$ such that, for $q\in(0, 1)$, $N\geq N_0$, and $x > x_0$,
\begin{align*}
\P\left(T_N \leq (\mu_q-1) N - x\cdot \frac{q^{1/6}}{1-q}N^{1/3}\right) \leq \exp(-cx^{3/2}).
\end{align*}

\end{theorem}

We next make some remarks on aspects of this result before outlining how to use Theorem~\ref{t.uniform lower tail} to complete the proof of Theorem~\ref{mt.q-pushtasep bound}.

\begin{remark}[Effective range of $x$]
Observe that $\mu_q-1 = \smash{\frac{2q^{1/2}(1+q^{1/2})}{1-q}}$, so the first order term $(\mu_q-1)N = O(q^{1/2}N/(1-q))$. Thus, for $x>C(q^{1/3}N^{2/3})$ for some fixed constant $C$, the probability is actually zero (since $T_N\geq 0$ always). For this reason it will be enough to prove the theorem for $x< \delta(q^{1/2}N)^{2/3}$ for some small $\delta>0$; then for $\delta(q^{1/2}N)^{2/3}\leq x \leq C(q^{1/2}N)^{2/3}$ one can obtain the claimed bound by modifying the constant $c$, and beyond that the bound holds trivially.
\end{remark}

\begin{remark}[Effective range of $q$]
Though $q$ is allowed to be arbitrarily close to zero, the statement is really only meaningful when $q$ is lower bounded by a constant times $N^{-2}$. This is simply because when $q=o(N^{-2})$, then $(\mu_q-1)N = O(q^{1/2}N)= o(1)$; similarly, the fluctuation scale is also $o(1)$. As a result the upper bound on $x$ of $\delta(q^{1/2}N)^{2/3}$ under which we need to prove the theorem also becomes $o(1)$, where it is trivial. This effective lower bound on $q$ reflects the fact that $q=\Theta(N^{-2})$ is the regime in which the number of points in $[1,N]^2\cap\Z^2$ where the geometric random variable is non-zero is $O(1)$, and, more precisely, converges to a Poisson random variable; thus the geometric LPP problem converges to Poissonian LPP (see \cite{johansson-toprows}).%
\end{remark}

\begin{remark}[Tail exponent of $3/2$]\label{r.geometric lpp tail exponent}
While the tail bound we obtain is $\exp(-cx^{3/2})$, the true lower tail behavior is $\exp(-cx^3)$ as proven in \cite{baik2001optimal} for fixed $q$. The same should be true uniformly in $q$ as well, i.e., we expect the inequality in Theorem~\ref{t.uniform lower tail} with the RHS replaced by $\exp(-cx^3)$ to be true.

As we said earlier, the more usual method of obtaining lower tail bounds via steepest descent analysis of Riemann-Hilbert problems does not appear to be suited to obtain uniform estimates. Instead, we utilize a method, often referred to in the literature as ``Widom's trick'',
which was first introduced by Widom \cite{widom2002convergence} to reduce the task to understanding the trace of the kernel operator of the Meixner ensemble, a determinantal point process associated to geometric LPP via the RSK correspondence. Widom's trick essentially treats the points of the Meixner ensemble as being independent, ignoring the repulsive behavior determinantal point processes exhibit. This simplifies the task of obtaining a lower tail bound, but at the cost of only yielding a tail bound with exponent $3/2$. 

This can likely be upgraded to the full cubic tail exponent uniformly in $q$ using bootstrapping arguments developed in \cite{ganguly2020optimal}, but one would first have to obtain similar uniform lower tail estimates (i.e., with a non-optimal tail exponent like $3/2$, though the framework in \cite{ganguly2020optimal} requires only stretched exponential tails) to points displaced (on the $N^{2/3}$ scale) from $(N,N)$. We will not pursue this here as it is not necessary for our bounds on $q$-pushTASEP, but extending our argument to other directions would involve considering more complicated formulas. Indeed, a formula for the factorial moments of the empirical distribution of the Meixner ensemble (see Section~\ref{s.intro.meixner analysis}) simplifies in the $(N,N)$ case (see \eqref{e.factorial moment formula}) and in the case of other directions one would need to perform asymptotic analysis on an additional layer of summation and handle an extra direction parameter in a uniform way in all of the estimates.
\end{remark}

\begin{remark}[Comparison to exponential LPP]
The cubic lower tail exponent has also been obtained in the model of exponential LPP \cite{ledoux2010,basu2019lower}; however, the method there is very different and crucially relies on the connection of that LPP value to the top eigenvalue of the Laguerre Unitary Ensemble random matrix theory model \cite{johansson2000shape} and a certain tridiagonal representation of the same. Geometric LPP does not seem to have any analogous connection to random matrix theory so such techniques are not applicable. We also note that, since exponential random variables have a scale invariance property, one obtains tail estimates for exponential LPP for any rate of the random variables by considering only rate 1. Thus the delicacy of uniformity in the $q$-parameter for geometric LPP that we must deal with also has no analogue in the exponential case.
\end{remark}

\subsubsection{Meixner ensemble analysis} \label{s.intro.meixner analysis}
Let us finally say a few words about what we need to know about the Meixner operator's trace. It is well-known (but proven here for completeness in Lemma~\ref{l.expression for trace}, see also \cite{ledoux2005deviation}) that the trace can be expressed in terms of the upper tail of the expected empirical distribution $\nu_{q,N}$ of the Meixner ensemble. We then need to obtain a lower bound on the upper tail of $\nu_{q,N}$. An argument of Ledoux given in the context of the GUE in \cite{ledoux2005deviation} suggests that this can be accomplished by obtaining sharp asymptotics for the moments of $\nu_{q,N}$. In our context, this means that the estimates need to be sharp in both their $q$ and $N$ dependencies. 

We obtain these estimates by first doing a careful analysis of formulas available for the factorial or Pochhammer moments of $\nu_{q,N}$ (i.e., $\E[X(X-1)\cdots(X-k+1)]$ for $X\sim \nu_{q,N}$) from \cite{ledoux2005distributions}, performing a Laplace method type argument for sums instead of integrals. We then develop arguments to convert these estimates into ones for the polynomial moments. These two tasks comprise the bulk of the technical content of the paper. The difficulty of obtaining these estimates comes primarily from the high order of the moments needed to adapt Ledoux's idea to our setting: indeed, it turns out that we need moments up to order $N$, unlike in the GUE case where order $N^{2/3}$ sufficed, and with a particular $q$-dependence for the error term (see Remark~\ref{r.q^1/2 in error term}). 

Further, in Ledoux's exposition in \cite{ledoux2005deviation} for the GUE, estimates were obtainable for the polynomial moments directly, unlike here where we must first start with factorial moments. Since $X\sim\nu_{q,N}$ with extremely high probability takes values of order $N$, the factor by which $X^k$ and $X(X-1)\cdots (X-k+1)$ typically differ when $k=O(N)$ becomes substantial, in fact, exponential in $N$. Thus the transfer between factorial and polynomial moments becomes delicate, and an exact cancellation needs to happen between factors which appear in the asymptotics for the factorial moments (coming from the value of the maximizer of the exponent in Laplace's method) and the exponential-in-$N$ discrepancy factor just mentioned. 
It would be interesting to see if there is some other more direct method to obtain the polynomial moment estimates which avoids these cancellations or approaches them in a more systematic way.

However, even with this cancellation, it does not seem tractable to move directly from factorial to polynomial moments. The strategy we instead adopt also makes use of the ``layer cake representation'' of the polynomial moments (as well as of other quantities): $\E[X^k] = \int_0^\infty kt^{k-1}\P(X>t)\,\diff t$. This equality shows that, if we have sharp upper bounds for the upper tail of $X$ (with the bounds being applicable in the entire tail), we can get sharp upper bounds for the polynomial moments of $X$; here by sharp we mean with the right dependencies on the various parameters in the exponent (and not the correct coefficient in the exponent, which we will not be concerned with). 

We obtain such probability tail bounds using Markov's inequality combined with the factorial moment asymptotics and the cancellation; in fact, even ignoring the fact that we need moments of order $N$ for our ultimate applications, we require estimates on the same order moments here too in order to obtain the sharp exponentially decaying upper tails (in particular, to get the right dependence on $N$ in Proposition~\ref{p.crude upper tail bound}). Next turning to the lower bound on the polynomial moments, we develop a related but slightly more complicated argument for which also only the upper bounds on the tail of $X\sim \nu_{q,N}$ suffice.

With these precise upper bounds on the upper tail of $\nu_{q,N}$ (which recall is the Meixner ensemble's mean empirical distribution) in hand, it is essentially immediate to also obtain uniform-in-$q$ upper tail estimates for geometric LPP with the correct $3/2$ tail exponent, and we record it below (though we do not need this estimate for any of our arguments concerning $q$-pushTASEP):

\begin{theorem}\label{t.uniform upper tail}
Let $T_N$ be the LPP value from top to bottom of an $N\times N$ square in an environment given by \iid Geo($q$) random variables. There exist positive constants $c$, $C$, $x_0$, and $N_0$ such that, for $q\in(0, 1)$, $N\geq N_0$, and $x_0 \leq x \leq (q^{1/2}N)^{2/3}$,
\begin{align*}
\P\left(T_N \geq (\mu_q-1) N + x\cdot \frac{q^{1/6}}{1-q}N^{1/3}\right) \leq C\exp(-cx^{3/2}).
\end{align*}
For $x\geq (q^{1/2}N)^{2/3}$, the inequality holds with the RHS replaced by $Cq^{-1/4}N^{-1/2}(1+q^{1/6}xN^{-2/3})^{-cN}$.
\end{theorem}

\subsubsection{Tying it together}

With Theorem~\ref{t.uniform lower tail}, the last ingredient is a concentration inequality. The inequality must take into account the fact that the scale of the random variables is decreasing. Typical concentration inequalities are for sub-Gaussian tail decay (while here we only have tail exponent $3/2$) and are for deviations from the mean (while our estimates are from the law of large numbers centering). While there are results in the literature for different tail decays, e.g. \cite{stretched-exp-concentration}, adapting these to address the second point directly to our setting results in a constant order loss for each term being summed, independent of the scale of the summand. This is too lossy as we have an infinite number of terms. Instead we redo the arguments establishing these bounds, which ultimately rely on estimates on the moment generating function, in such a way to fit our applications.
With this final step, we will obtain Theorem~\ref{mt.q-pushtasep bound}.

\begin{remark}[An argument for the lower tail of $x_N(T)$]
As we saw, the conceptual heart of the argument consisted of finding a concatenation of paths which, to first order, has the same weight as the law of large numbers \eqref{e.first form of LLN} for the model. Now, if we were interested in $x_N(T)$ for general $T$, there is also a representation of it in terms of a periodic LPP problem, where the environment consists of periodic rectangles of dimension $N\times T$ instead of $N\times N$ squares as here. However, in such an environment, it is not clear what concatenation of paths would achieve the correct first order weight, and this is why we restrict to $T=N$ in this paper. We leave the general $T$ case for future work.
\end{remark}

\subsection{A remark on convergence to the log-gamma free energy}\label{s.log gamma remark}
Though not needed for the results in this paper, we also note that, as proven in \cite{matveev2016q}, the $q\to 1$ limit of $\resc$, when renormalized correctly, is the free energy of the log-gamma polymer introduced in \cite{seppalainen2012scaling} (we refer to the reader to that paper for the precise definition of the model).
Indeed for example, setting $q=\exp(-\varepsilon)$ and $u=\exp(-A\varepsilon)$ for a fixed $A>0$, \cite[Theorem~8.7]{matveev2016q} tell us that $\varepsilon(x_N(N)-(2N-1)\varepsilon^{-1}\log \varepsilon^{-1})$ converges in distribution to the log-gamma free energy where the parameters of the inverse gamma random variables are all $2A$.
It can be checked that the appropriately normalized $q\to 1$ limit of $f_q$ (as defined in \eqref{e.fq definition}) is indeed the law of large numbers for the log-gamma polymer.

However, notice that the centering term for the convergence is $(2N-1)\varepsilon^{-1}\log\varepsilon^{-1}$, while the first order behavior (in $N$) we calculated in \eqref{e.lln expression} via the connection to LPP, when written in terms of $\varepsilon$, was $2N\varepsilon^{-1}\log\varepsilon^{-1}$. In other words, there is a discrepancy of $\varepsilon^{-1}\log\varepsilon^{-1}$. This comes from the earlier noted point that the fluctuation scale we are able to prove (when written in terms of $\varepsilon$) is $\varepsilon^{-1}\log\varepsilon^{-1}$, unlike the true fluctuation scale of $\varepsilon^{-1}$ suggested by \cite{matveev2016q}; equivalently, our lower tail bound (for $x_N(N)$ and not $\resc$) only kicks in after $\varepsilon^{-1}\log \varepsilon^{-1}N^{1/3}$ into the tail.
For this reason, unfortunately, our tail bounds do not survive in the limit to provide a tail bound on the log-gamma free energy.

The ultimate source of the discrepancy in the fluctuation scale that we are able to prove is that we are approximating the true LPP value in the infinite cylinder by a sum of LPP values in $N\times N$ big squares. In more detail, the portion of our path in the $i$\textsuperscript{th} big square from the top suffers a loss of order $q^{i/6}(1-q^{2i})^{-1} N^{1/3}$ (ignoring the $u$-dependence), essentially because this is the scale of fluctuations on which the LPP value in this box converges to the Tracy-Widom distribution, and the latter has a negative mean. Observing that $1-q^{2i}$ is approximately $\varepsilon i$ up to constants when $q=\exp(-\varepsilon)$, we see that the sum of this loss from $i=1$ to $\infty$ yields an overall loss of of order $N^{1/3}$ times $\varepsilon^{-1}\sum_{i=1}^\infty i^{-1}e^{-i\varepsilon/6} = \varepsilon^{-1}\log(1-e^{-\varepsilon/6}) \approx \varepsilon^{-1}\log(\varepsilon^{-1})$. Thus to avoid the lossy factor of $\log \varepsilon^{-1}$ it seems one would need a different scheme of approximation.

As mentioned earlier in Section~\ref{s.recent work in polymer}, very recent work \cite{landon2022upper} has established a bound (with tail exponent $3/2$) on the lower tail of the free energy of a \emph{stationary} version of the log-gamma model (as well as other polymer models such as the O'Connell-Yor model) using a Burke property enjoyed by the model (proved in \cite{OY01} which also introduced the model, and analogous properties in models such as exponential LPP were used to obtain exponentially decaying tail estimates earlier in \cite{emrah2020right,emrah2021optimal,emrah2022coupling}), which gives access to formulas for the moment generating function of the free energy. For the O'Connell-Yor model, in \cite{landon2022tail}, these bounds were transferred to the non-stationary version of the model using geometric arguments involving the polymer measure introduced in \cite{flores2014fluctuation}, and the tail exponent was upgraded to the optimal $3$ by adapting geometric methods from \cite{ganguly2020optimal}. One expects that a similar program would deliver the corresponding bounds in the log-gamma case as well.

\subsection*{Acknowledgements}
The authors thank Matteo Mucciconi for explaining the proof of Theorem~\ref{t.q-pushTASEP to lpp}, as well as Philippe Sosoe and Benjamin Landon for sharing their preprint \cite{landon2022tail} with us in advance. We also thank the anonymous referees for their thorough reading of the paper and helpful comments. I.C. was partially supported by the NSF through grants DMS:1937254, DMS:1811143, DMS:1664650,
as well as through a Packard Fellowship in Science and Engineering, a Simons Fellowship, and
a W.M. Keck Foundation Science and Engineering Grant. I.C. also thanks the Pacific Institute for the Mathematical Sciences (PIMS) and Centre de Recherches Mathematique (CRM), where some materials were developed in conjunction with the lectures he gave at the PIMS-CERM summer school in probability (which is partially supported by NSF grant DMS:1952466). M.H. was partially supported by NSF grant DMS:1937254.

\section{Widom's trick applied to the lower tail in geometric LPP}

In the next two sections we will prove Theorem~\ref{t.uniform lower tail}, which provides an upper bound on the lower tail of the LPP value in an \iid geometric environment, uniform in the parameter of the geometric random variables.

The argument relies on a trick introduced by Widom in \cite{widom2002convergence}, which we explain next.

\subsection{Widom's trick}

We first need to introduce the Meixner ensemble, the determinantal point process associated to geometric LPP via the RSK correspondence. The fact that it is determinantal is the crucial property for Widom's argument.

\begin{definition}[Meixner ensemble]
First let $\mugeo$ denote the Geo($q$) distribution on $\N_0 := \N\cup\{0\}$, i.e., the distribution with discrete weights given by
\begin{align*}
\mugeo(\{x\}) = (1-q)q^x.
\end{align*}
For $q\in(0,1)$ and $N\in\N$, the $N\times N$ \emph{Meixner ensemble} is a determinantal point process on $\N_0$ with kernel given, for $x,y\in\N_0$ with $x\neq y$ and with respect to $\mugeo$, by
\begin{align}\label{e.meixner kernel}
 K_{N}(x,y) = \frac{\kappa_{N-1}}{\kappa_N}\cdot\frac{M_N(x)M_{N-1}(y) - M_{N-1}(x)M_N(y)}{x-y}; %
 \end{align}
here $M_N=\kappa_N x^N + \kappa_{N-1}x^{N-1}+ \ldots +\kappa_0$ are the orthonormal polynomials (which we call the Meixner polynomials, though it differs from the classical Meixner polynomials by a constant multiple due to the normalization) with respect to $\mugeo$.
The second factor on the right-hand side of \eqref{e.meixner kernel} makes sense for $x,y\in\R$, and so the $x=y$ case can be defined by taking the appropriate limit.
\end{definition}

Here is the relation between the Meixner ensemble and the geometric LPP value.

\begin{proposition}[Proposition~1.3 of \cite{johansson2000shape}]\label{p.meixner and geo}
Fix $q\in(0,1)$ and $N\in\N$. Let $\lambda_1 \geq \lambda_2 \geq  \ldots \lambda_N$ be distributed according to the $N\times N$ Meixner ensemble and let $T_N$ be the LPP value in the environment of \iid Geo($q$) random variables. Then $\smash{T_N \stackrel{d}{=} \lambda_1-N+1}$.
\end{proposition}

With this background we may explain Widom's trick. The fact that $(\lambda_1, \ldots, \lambda_N)$ is determinantal with kernel $K_{N}$ given by \eqref{e.meixner kernel} implies, using the Cauchy-Binet formula, that, for any $t\in\R$,
\begin{align*}
\P\left(\lambda_1 \leq t\right) = \det\left(I_N - K^t_{N}\right),
\end{align*}
where $K^t_{N}$ can be written as the Gram matrix of the Meixner polynomials, i.e.,
\begin{align}
K^t_{N} = \bigl(\langle M_{\ell -1}, M_{k-1}\rangle_{\ell^2(\{t, t+1, \ldots\},\, \mugeo)}\bigr)_{1\leq k,\ell\leq N}. \label{e.K^t definition}
\end{align}
The fact that Gram matrices are positive semi-definite implies that the eigenvalues of $\smash{K^t_{N}}$ are non-negative; also, we may write, for any unit vector $\smash{u\in\R^N}$ and with $\smash{g(x) = \sum_{i=1}^N u_iM_{i-1}(x)}$,
\begin{align*}
1= \sum_{i=1}^N u_i^2 = \langle g,g\rangle_{\ell^2(\N_0, \mugeo)} &= \langle g\one_{\cdot<t},g\one_{\cdot<t}\rangle_{\ell^2(\N_0, \mugeo)} + \langle g\one_{\cdot\geq t},g\one_{\cdot\geq t}\rangle_{\ell^2(\N_0, \mugeo)}\\
&\geq \langle g\one_{\cdot\geq t},g\one_{\cdot\geq t}\rangle_{\ell^2(\N_0, \mugeo)}
= \langle g,g\rangle_{\ell^2(\{t, t+1,\ldots\}, \mugeo)} = u^T K^t_{N} u,
\end{align*}
which in turn implies that the eigenvalues of $K^t_{N}$ are at most 1.

Let us label the eigenvalues of $K^t_{N}$ as $\rho^t_1, \ldots, \rho^t_N$. Since $1-x\leq e^{-x}$ for $x\in[0,1]$,
\begin{align*}
\P\left(\lambda_1 \leq t\right) = \det\left(I_N - K^t_{N}\right) = \prod_{i=1}^N (1-\rho^t_i) &\leq \exp\left(-\sum_{i=1}^N\rho^t_i\right) = \exp\left(-\mrm{Tr}(K^t_{N})\right).
\end{align*}

Thus Widom's trick reduces the problem of bounding the lower tail to understanding the trace of an associated operator. This in turn can be accomplished by lower bounding the upper tail of the expected empirical distribution $\nu_{q,N}$ of the Meixner ensemble, defined precisely by
\begin{equation}\label{e.nu_q,N definition}
\nu_{q,N} = \E\left[\frac{1}{N}\sum_{i=1}^N\delta_{\lambda_i}\right].
\end{equation}
We record the connection between the operator's trace and the tail of $\nu_{q,N}$ next.

\begin{lemma}\label{l.expression for trace}
For any $t\in \N$, $\mrm{Tr}(K^t_{N}) = N \nu_{q,N}([t,\infty))$.
\end{lemma}

\begin{proof}%
First we observe that, from \eqref{e.K^t definition},
\begin{align}\label{e.trace formula}
\mrm{Tr}(K^t_{N}) = \sum_{\ell=0}^{N-1}\langle M_{\ell-1}, M_{\ell-1}\rangle_{\ell(\{t, t+1, \ldots\}, \mugeo)} = \int_{t}^\infty \sum_{\ell=0}^{N-1} M_{\ell}^2\,\diff \mugeo.
\end{align}
Now since $(\lambda_1, \ldots, \lambda_N)$ is determinantal with kernel $K_N$ with respect to $\mugeo$, it is a standard fact of the theory of determinantal point processes (or see \cite[Proposition~1.2]{ledoux2005deviation}) that, for any bounded measurable $f:\N_0\to\R$,
\begin{align*}
\E\left[\prod_{i=1}^N[1+f(\lambda_i)]\right] = \sum_{r=0}^N \frac{1}{r!}\int_{\N_0^r}\prod_{i=1}^r f(x_i)\det(K_{N}(x_i, x_j))_{1\leq i,j\leq r}\,\diff \mugeo(x_1)\cdots \diff\mugeo(x_r).
\end{align*}
Replacing $f$ by $\varepsilon f$, taking the $\varepsilon\to 0$ limit, and thereby equating the order $\varepsilon$ terms on both sides (since the constant-in-$\varepsilon$ terms on both sides are easily seen to be $1$), we obtain that
\begin{align*}
\E\left[\sum_{i=1}^N f(\lambda_i)\right] = \int_{\N_0} f(x) K_N(x,x) \,\diff \mugeo(x).
\end{align*}
By the Christoffel-Darboux formula and \eqref{e.meixner kernel}, $K_N(x,x) = \sum_{\ell=0}^{N-1} M_{\ell}(x)^2$. With this, taking $f(x) = \one_{x\geq t}$ and using \eqref{e.trace formula} yields the claim.
\end{proof}

So the task is now to obtain a lower bound on the upper tail of $\nu_{q,N}$. The bound we prove is stated in the next theorem, and its proof will be the main goal of the remainder of this section as well as of the next two.

\begin{theorem}\label{t.mean empirical law lower bound}
Let $X$ be distributed as $\nu_{q,N}$ as defined in \eqref{e.nu_q,N definition} and let $\mu_q$ be as in \eqref{e.mu_q definition}. There exist positive absolute constants $c$, $C$, and $N_0$ such that, for $N\geq N_0$, $\varepsilon\in[CN^{-2/3},1]$, and $q\in[\varepsilon^3, 1)$,
\begin{align*}
\P\left(X\geq \mu_qN(1-q^{1/6}\varepsilon)\right) \geq c\varepsilon^{3/2}.
\end{align*}

\end{theorem}

In fact, the lower bound of order $\varepsilon^{3/2}$ is sharp and we prove a matching order upper bound in Proposition~\ref{p.upper tail bound inner}.

With Theorem~\ref{t.mean empirical law lower bound} and Widom's trick, Theorem~\ref{t.uniform lower tail}'s proof is straightforward.

\begin{proof}[Proof of Theorem~\ref{t.uniform lower tail}]
Recall from Proposition~\ref{p.meixner and geo} that if $\lambda_1$ is distributed as the largest particle of the Meixner ensemble, then $\smash{T_N \stackrel{d}{=} \lambda_1-N+1}$. From Lemma~\ref{l.expression for trace}, for any $t\in \N$,
\begin{align*}
\P\left(\lambda_1\leq t\right) \leq \exp\Bigl(-N\nu_{q,N}([t,\infty))\Bigr).
\end{align*}
So we see that, for any $\varepsilon>0$,
\begin{align*}
\P\Bigl(T_N \leq (\mu_q-1)N - \mu_qNq^{1/6}\varepsilon)\Bigr)
&= \P\Bigl(\lambda_1 \leq \mu_qN(1-q^{1/6}\varepsilon)\Bigr)\\
&\leq \exp\Bigl\{-N\nu_{q,N}\bigl([\mu_qN(1-q^{1/6}\varepsilon), \infty)\bigr)\Bigr\}.
\end{align*}
By Theorem~\ref{t.mean empirical law lower bound}, there exist positive constants $c$ and $c_0$ such that, for all $q\in[\varepsilon^3, 1)$ and $C N^{-2/3}\leq \varepsilon\leq 1$,
\begin{align*}
\nu_{q,N}\bigl([\mu_qN(1-q^{1/6}\varepsilon), \infty)\bigr) \geq c\varepsilon^{3/2}.
\end{align*}
Putting the above together with $\varepsilon = xN^{-2/3}(1-q)^{-1}\mu_q^{-1} = xN^{-2/3} (1+q^{1/2})^{2}$ for $x>C$ and adjusting the constant in the exponent gives
\begin{align*}
\P\left(T_N \leq (\mu_q-1)N - x \frac{q^{1/6}}{1-q}N^{1/3}\right) \leq \exp\left(-cx^{3/2}\right)
\end{align*}
when $q\geq x^3N^{-2}$ and $xN^{-2/3}\leq 1$ (since $q\geq \varepsilon^3$ and $\varepsilon\leq 1$ are required to apply Theorem~\ref{t.mean empirical law lower bound}), which are both implied by the hypothesis that $x\leq (q^{1/2}N)^{2/3}$. This completes the proof.
\end{proof}

To prove Theorem~\ref{t.mean empirical law lower bound}, we rely on a strategy of Ledoux explained in \cite[Section 5]{ledoux2005deviation}. It relies on getting strong estimates on polynomial moments of the mean empirical distribution $\nu_{q,N}$. The bounds we prove are the following.

\begin{theorem}\label{t.meixner poly moment bounds}
Let $X$ be distributed according to $\nu_{q,N}$ as defined in \eqref{e.nu_q,N definition}. There exist positive $C$, $\alpha_0$, and $N_0$ such that for any $N\geq N_0$ and $(k,q)$  satisfying $k\leq (\alpha_0N\wedge q^{-1/6}N^{2/3})$ and $q\in [k^{-2},1)$,
\begin{align*}
C^{-1}(q^{1/6}k)^{-3/2}(\mu_q N)^k \leq \E[X^k] \leq C(q^{1/6}k)^{-3/2}(\mu_q N)^k .
\end{align*}
\end{theorem}

Observe that the moments grow to first order like $(\mu_q N)^k$, which reflects that we expect $\nu_{q,N}$ to be supported on $[0,\mu_q N]$ (though more precisely there is a decaying-in-$N$ amount of mass beyond this point, which we will in fact upper bound in Proposition~\ref{p.crude upper tail bound}). The polynomial dependence on $k$ is what captures the behaviour of the tail of $\nu_{q,N}$ near this right edge, and this is the basic observation of Ledoux' argument. Indeed, the exponent of $-3/2$ for $k$ is what gives the $3/2$ exponent of $\varepsilon$ in Theorem~\ref{t.mean empirical law lower bound}.

Note also that we allow $k$ to go up to a rather large value; essentially of order $N$, since, when $q = \Theta(N^{-2})$, then $q^{-1/6}N^{2/3} = \Theta(N)$. We will in fact need the estimate to allow such large values of $k$ in the application, namely the proof of Theorem~\ref{t.mean empirical law lower bound}. We expand on this slightly  ahead in Remark~\ref{r.why such large k}. Finally we remark that the expression $q^{-1/6}N^{2/3}$ in the upper bound on $k$ is a technical feature, and may be an artifact, of the conversion we perform from factorial to polynomial moments (indeed, this expression does not appear in the bound on $k$ in the upcoming Theorem~\ref{t.factorial moment asymptotics} on factorial moment estimates).

We will prove Theorem~\ref{t.meixner poly moment bounds} in Sections~\ref{s.factorial bounds} and \ref{s.moment bounds}; the upper and lower bounds are separated into Propositions~\ref{p.stronger poly moment upper bound meixner} and \ref{p.polynomial lower bound meixner} respectively.

We conclude this section by using Theorem~\ref{t.meixner poly moment bounds} to implement Ledoux' argument to establish Theorem~\ref{t.mean empirical law lower bound}.

\begin{proof}[Proof of Theorem~\ref{t.mean empirical law lower bound}]
First, by the Cauchy-Schwarz inequality, we have
\begin{align}\label{e.cauchy-schwarz}
\E[X^{2k}\one_{X\geq \mu_qN(1-q^{1/6}\varepsilon)}] \leq \E[X^{4k}]^{1/2}\P\left(X\geq \mu_qN(1-q^{1/6}\varepsilon)\right)^{1/2}.
\end{align}
By Theorem~\ref{t.meixner poly moment bounds}, when $q\geq k^{-2}$, $4k\leq (\alpha_0N\wedge q^{-1/6}N^{2/3})$ and $N\geq N_0$ (conditions we assume in the rest of the proof),
\begin{align*}
\E[X^{4k}] \leq C_1(q^{1/6}k)^{-3/2}(\mu_q N)^{4k}.
\end{align*}
It is also easy to see that
\begin{equation}\label{e.cauchy schwarz next step}
\begin{split}
\E[X^{2k}\one_{X\geq \mu_qN(1-q^{1/6}\varepsilon)}]%
&= \E[X^{2k}] - \E[X^{2k}\one_{X< \mu_{q}N(1-q^{1/6}\varepsilon)}]\\
&\geq \E[X^{2k}] - \E[X^{k}]\left(\mu_{q}N(1-q^{1/6}\varepsilon)\right)^k.
\end{split}
\end{equation}
Now, from Theorem~\ref{t.meixner poly moment bounds}, under the same conditions on $(q,k)$ as above,
\begin{align*}
\E[X^{2k}] \geq C_2(q^{1/6}k)^{-3/2}(\mu_qN)^{2k},
\end{align*}
while, again from Theorem~\ref{t.meixner poly moment bounds} and the same conditions on $(q,k)$,
\begin{align*}
\E[X^{k}] \leq C_3(q^{1/6}k)^{-3/2}(\mu_q N)^{k}.
\end{align*}
Since overall it holds (under the condition that the expression in the parentheses is positive) from \eqref{e.cauchy-schwarz} and \eqref{e.cauchy schwarz next step} that
\begin{align*}
\P\bigl(X\geq \mu_qN(1-q^{1/6}\varepsilon)\bigr) \geq \E[X^{4k}]^{-1}\left(\E[X^{2k}] - \E[X^k](\mu_qN)^k(1-q^{1/6}\varepsilon)^k\right)^2,
\end{align*}
substituting the above recalled bounds on the moments of $X$ yields (cancelling out all the common factors of $(\mu_qN)^{4k}$ on the right-hand side)
\begin{align}
\P\bigl(X\geq \mu_qN(1-q^{1/6}\varepsilon)\bigr)\nonumber
&\geq C_1^{-1}(q^{1/6}k)^{3/2}\left[C_2(q^{1/6}k)^{-3/2} - C_3(q^{1/6}k)^{-3/2}(1-q^{1/6}\varepsilon)^k\right]^2\nonumber\\
&\geq C_1^{-1}(q^{1/6}k)^{-3/2}\left[C_2 - C_3\exp(-q^{1/6}\varepsilon k)\right]^2. \label{e.last CS inequality}
\end{align}
We need to set $k$ such that the expression in the square brackets is uniformly positive; it is sufficient if
\begin{align}\label{e.smallness condition in CS argument}
\frac{C_3}{C_2}\exp\left(-q^{1/6}\varepsilon k\right) < \frac{1}{2}.
\end{align}
We will verify that this condition is met with $k = Wq^{-1/6}\varepsilon^{-1}$ for a large absolute constant $W$ to be chosen: the expression on the LHS of the previous display becomes
\begin{align}\label{e.small condition after k choice}
\frac{C_3}{C_2} \exp\left(-W\right).
\end{align}
We pick $W$ such that $\frac{C_3}{C_2}\exp(-W) < \frac{1}{2}$. With this our choice of $k$ has been made, and we next verify that the conditions that $k$, $q$, $\varepsilon$ must satisfy to apply Theorem~\ref{t.meixner poly moment bounds} indeed hold. After that we will check that \eqref{e.smallness condition in CS argument} holds.

Recall that under the hypotheses of Theorem~\ref{t.mean empirical law lower bound} that we are proving, we have a lower bound on $\varepsilon$ of the form $N^{-2/3}$ times an absolute constant that we are free to choose and which we label $c_0^{-1}$, i.e., $\varepsilon$ satisfies $\varepsilon > c_0^{-1}N^{-2/3}$. We also have $q \geq \varepsilon^3$. We need to verify two conditions: (i) $q\geq k^{-2}$ and (ii) $k\leq \frac{1}{4}\min(\alpha_0N, q^{-1/6}N^{2/3})$.

For (i), it is easy to check that $q \geq k^{-2}$ is equivalent to $q\geq W^{-3}\varepsilon^3$ for our choice of $k$. If $W \geq 1$ (which we can ensure by raising $W$ if needed), then clearly our hypothesis that $q\geq \varepsilon^3$ implies $q\geq k^{-2}$.

Next we set $c_0>0$ small enough (depending on $W$) that the condition $\varepsilon > c_0^{-1}N^{-2/3}$ implies $k \leq \frac{1}{4}\min(\alpha_0N, q^{-1/6}N^{2/3})$ (i.e., that $\varepsilon > 4\max(W, \alpha_0^{-1})N^{-2/3}$, i.e., $c_0< \frac{1}{4}\min(W^{-1}, \alpha_0)$). This verifies that the upper bound (ii) above on $k$ holds, and thus that the hypotheses of Theorem~\ref{t.meixner poly moment bounds} are satisfied.

Using \eqref{e.smallness condition in CS argument} in \eqref{e.last CS inequality} and substituting $k=Wq^{-1/6}\varepsilon^{-1}$ in the same, we obtain
\begin{align*}
\P\bigl(X\geq \mu_qN(1-q^{1/6}\varepsilon)\bigr)
&\geq \tfrac{1}{4}C_1^{-1}C_2^2W^{-3/2}\varepsilon^{3/2}.
\end{align*}
Thus, overall, our choice of parameters yields that there exist positive absolute constants $c$ and $c_0$, such that for all $\varepsilon>c_0^{-1}N^{-2/3}$ and $q\geq \varepsilon^3$,
\begin{align*}
\P\bigl(X\geq \mu_qN(1-q^{1/6}\varepsilon)\bigr) \geq c\varepsilon^{3/2}.
\end{align*}
This completes the proof.
\end{proof}

\begin{remark}
Unfortunately, it does not seem possible to adapt the above argument to work with estimates on the factorial moments directly instead of first converting to polynomial moments, as such an adaptation would allow us to avoid the analysis to go from Theorem~\ref{t.factorial moment asymptotics} to Theorem~\ref{t.meixner poly moment bounds}. 

The reason is essentially that factorial moments do not work well with the Cauchy-Schwarz inequality: $\E[(X_k)] \leq \E[(X)_k^2]^{1/2}$, but we do not have any control on the RHS via estimates on factorial moments. This is in stark contrast to polynomial moments, where applying Cauchy-Schwarz results in another polynomial moment, which we do have estimates on.
\end{remark}

\begin{remark}\label{r.why such large k}
We observe that the proof used the polynomial moment of $X\sim\nu_{q,N}$ of order $q^{-1/6}\varepsilon^{-1}$. When $\varepsilon = N^{-2/3}$ (which falls within the parameter ranges we are interested in since such a value of $\varepsilon$ correspond to the KPZ fluctuation scale), the moments we are considering become of order $q^{-1/6}N^{2/3}$ (which further becomes $N$ if $q=\Theta(N^{-2})$), and this is why we need Theorem~\ref{t.meixner poly moment bounds} to allow $k$ to go up to $q^{-1/6}N^{2/3}$.
\end{remark}

\section{Sharp factorial moment bounds}\label{s.factorial bounds}

Here we prove Theorem~\ref{t.meixner poly moment bounds} on sharp bounds for the moments of the expected empirical distribution $\nu_{q,N}$ (as defined in \eqref{e.nu_q,N definition}) of the Meixner ensemble. We adopt the notation
$$(x)_k := x(x-1)\cdots(x-k+1).$$
While there is no explicit formula available for the polynomial moments of $\nu_{k,N}$, there \emph{is} one for the factorial moments, i.e., moments of the form $\E[(X)_k] = \E[X(X-1)\cdots(X-k+1)]$. Indeed letting $X$ be distributed according to $\nu_{q,N}$, \cite[Lemma~5.2]{ledoux2005distributions} states that
$$M^q(k,N) := \E[(X)_k] = \frac{q^k}{(1-q)^k}\sum_{i=0}^kq^{-i}\binom{k}{i}^2\cdot\sum_{\ell=i}^{N-1} \frac{(\ell+k-i)!}{(\ell-i)!}.$$
In fact, this simplifies \cite[eq. (3.5)]{cohen2020moments} to
\begin{align}\label{e.factorial moment formula}
M^q(k,N) = \frac{q^k}{(1-q)^k}\frac{1}{N}\cdot \frac{1}{k+1}\sum_{i=0}^kq^{-i}\binom{k}{i}^2\cdot \frac{(N+k-i)!}{(N-i-1)!}.
\end{align}

 Our approach is to use this formula to obtain asymptotics on the factorial moments, and then later convert them into polynomial moments.

\subsection{The factorial moment asymptotics}

\begin{theorem}\label{t.factorial moment asymptotics}
Let $X$ be distributed according to $\nu_{q,N}$ as defined in \eqref{e.nu_q,N definition}, and let $\mu_q = (1+q^{1/2})^2/(1-q)$ be as in \eqref{e.mu_q definition}. There exist positive constants $C$, $N_0$, and $k_0$ such that for all $N\geq N_0$, $k_0\leq k\leq \frac{1}{2}N$, and $q\in [k^{-2},1)$, $M^q(k,N)$ is upper and lower bounded by (and letting $\alpha:=k/N$)
\begin{align*}
(q^{1/6}k)^{-3/2}(\mu_qN)^k\exp\left(-N\left[(\mu_q-\alpha)\log\left(1-\frac{\alpha}{\mu_q}\right) + \alpha\right]\right)
\end{align*}
up to factors of $C\exp\left(Cq^{1/2}\frac{k^3}{N^2}\right)$ or its inverse. Further, the upper bound holds for all $q\in(0,1)$.
\end{theorem}

As we said, $\nu_{N,q}$ has right edge of support roughly $\mu_qN=\smash{\frac{(1+q^{1/2})^2}{1-q}}N$, and this is what gives that to first order, $M^q(k,N)N^{-k}$ grows like $\smash{\mu_q^k}$. The main task is to obtain the correct polynomial dependence on $k$ and $q$, namely $q^{-1/4}k^{-3/2}$, of the same.

It turns out that the exponential factor in the theorem is the quantity to go between factorial and polynomial moments. In Section~\ref{s.moment bounds}, when we use Theorem~\ref{t.factorial moment asymptotics} to prove Theorem~\ref{t.meixner poly moment bounds} on sharp upper and lower bounds on the polynomial moments, this factor will get canceled. In fact, the expression in the exponent that one gets by directly doing Laplace's method on the sum (as is our broad strategy to prove Theorem~\ref{t.factorial moment asymptotics} as we indicated in the introduction) is not the one recorded above. The recorded expression is instead obtained by identifying an additional cancellation as mentioned in the idea of proofs Section~\ref{s.proof ideas}; this cancellation is isolated as Lemma~\ref{l.factorial to poly cancellation} in Appendix~\ref{app.sum asymptotics}, where most of the technical estimates required for Theorem~\ref{t.factorial moment asymptotics} are done.

\begin{remark}\label{r.q^1/2 in error term}
Let us also emphasize the fact that the coefficient of $k^3/N^2$ in the error term is of order $q^{1/2}$ and not unit order, which would have been easier to obtain (and would suffice if uniform-in-$q$ bounds are not required). As we saw at the end of the previous section, the $q^{1/2}$ factor was important in obtaining the uniform lower bound in Theorem~\ref{t.mean empirical law lower bound}, as it ensures that $q^{1/2}k^3/N^2$ does not blow up as $q\to 0$ for the choice of $k$ (which recall grows as $q^{-1/6}$). 
\end{remark}

Before turning to giving the proof of Theorem~\ref{t.factorial moment asymptotics} in full (though as mentioned much of the technical estimates have been relegated to Appendix~\ref{app.sum asymptotics}), let us outline the strategy. First, we will use  Stirling's approximation and a precise form of the fact that $\binom{k}{i} \approx \exp(k H(i/k))$ (with $H(x) = -x\log x - (1-x)\log(1-x)$ the entropy function) to write the sum in \eqref{e.factorial moment formula} as, approximately and up to absolute constants,
\begin{equation}\label{e.sum for intuition}
\begin{split}
\MoveEqLeft[6]
\frac{N^kq^k}{(1-q)^k}\cdot k^{-1}\cdot\frac{1}{k+1}\sum_{i=0}^k\frac{k^2}{i(k-i)}\exp\biggl(i\log q^{-1} + 2k\cdot H(i/k)\\
&  + (k+1)\log\left(1+\frac{k-i}{N}\right) + \left(N-i-1\right)\log\left(1+\frac{k+1}{N-i-1}\right) -(k+1)\biggr).
 \end{split}
\end{equation}
Then, the idea is to obtain asymptotics for the sum using Laplace's method. Indeed, if we were to write $k$ as $\alpha N$ and regard the sum (along with the $(k+1)^{-1}$ factor as being approximately the integral
\begin{align*}
\int_0^1\frac{1}{x(1-x)}\exp\left(N f_\alpha(x)\right)\,\diff x,
\end{align*}
where 
\begin{align}\label{e.f_alpha first defn}
f_\alpha(x) = \alpha x\log q^{-1} + 2\alpha H(x) + \alpha \log(1+\alpha(1-x)) + (1-\alpha x)\log\left(1+\frac{\alpha}{1-\alpha x}\right) -\alpha,
\end{align}
then, letting $g(x) = (x(1-x))^{-1}$,
one would expect from the Laplace method heuristic that the sum in \eqref{e.sum for intuition} is approximately
\begin{align*}
(N|f_\alpha''(x_0)|)^{-1/2}\exp(Nf_\alpha(x_0))g(x_0),
\end{align*}
up to constants, where $x_0 = x_0(\alpha)$ is the maximizer of $f_\alpha$ on $[0,1]$. Evaluating $f_\alpha(x_0)$ and $g(x_0)$ will yield the claimed $q$ and $k$ dependencies in Theorem~\ref{t.factorial moment asymptotics}.

There are existing results in the literature, for example \cite{masoero2014laplace}, on Laplace's method for sums which also obtain the correct constant coefficient multiplying the previous display. However these are not directly useful to us: we need all our estimates to be uniform in the parameter $q$, which can be difficult to verify after applying black-box theorems. For this reason we perform the analysis explicitly ourselves; but we will not be concerned with obtaining the correct constant dependencies as these are not necessary for our ultimate applications. However, since there is not much novel or probabilistic content in these computations, we defer to Appendix~\ref{app.sum asymptotics} the proofs of these bounds (which are stated as Propositions~\ref{p.ledoux sum upper bound} and \ref{p.ledoux sum lower bound}).

\begin{proof}[Proof of Theorem~\ref{t.factorial moment asymptotics}]
In the following, we will include error terms of the form $\pm O(r(k,N))$ in front to emphasize that we allow the error to be positive or negative, as long as it is in magnitude bounded by $C r(k,N)$ for some constant $C$ not depending on $q$, $i$, $k$, $\alpha$, or $N$, in the range $(0< \alpha \leq \tfrac{1}{2})$). Similarly, a factor of $\Theta(1)$ indicates that the equality holds up to a factor of an absolute constant.

We will first obtain a bound on the summands in \eqref{e.factorial moment formula}. We start with bounding
$(N+k-i)!/(N-i-1)!$
using Stirling's approximation (non-asymptotic form) to obtain
\begin{align*}
\frac{(N+k-i)!}{(N-i-1)!} &= \frac{\sqrt{2\pi (N+k-i)}}{\sqrt{2\pi(N-i-1)}} \exp\Bigl[(N+k-i)\log(N+k-i)\\
&\qquad - (N+k-i) - (N-i-1)\log(N-i-1) + (N-i-1) \pm O(N^{-1})\Bigr]\\
&= \Theta(1)\cdot\exp\left[(N-i-1) \log\frac{N+k-i}{N-i-1} + (k+1)\bigl(\log(N+k-i)-1\bigr)\right]\\
&= \Theta(1)\cdot\exp\left[(N-i-1) \log\left(1+\frac{k+1}{N-i-1}\right) + (k+1)\bigl(\log(N+k-i)-1\bigr)\right].
\end{align*}
Substituting the above expression into that of $M^q(k,N)$ \eqref{e.factorial moment formula} and multiplying by $N^{-k}$, we obtain
\begin{align*}
M^q(k,N)N^{-k}
&= \frac{\Theta(1)q^k}{(1-q)^k(k+1)}\sum_{i=0}^kq^{-i}\binom{k}{i}^2 \exp\Bigl[(N-i-1) \log\left(1+\frac{k+1}{N-i-1}\right)\\
&\qquad\qquad + (k+1)\log (N+k-i) -(k+1)\log N -(k+1)\Bigr]\\
&= \frac{\Theta(1)q^k}{(1-q)^k(k+1)}\sum_{i=0}^kq^{-i}\binom{k}{i}^2 \exp\Bigl[(N-i-1) \log\left(1+\frac{k+1}{N-i-1}\right)\\
&\qquad\qquad + (k+1)\log\left(1 + \frac{k-i}{N}\right) -(k+1)\Bigr].
\end{align*}
Now since
\begin{align*}
\binom{k}{i} = \Theta(1) \sqrt{\frac{k}{i(k-i)}} \exp(kH(i/k)),
\end{align*}
where $H(p) = -p\log p -(1-p)\log(1-p)$ is the entropy function, we obtain that $M^q(k,N)N^{-k}$ is equal to
\begin{align*}
\MoveEqLeft[16]
\frac{\Theta(1)q^k}{(1-q)^k}\cdot k^{-1}\cdot\frac{1}{k+1}\sum_{i=0}^k\frac{k^2}{i(k-i)}\exp\Bigl[i\log q^{-1} + 2kH(i/k) + (N-i-1) \log\left(1+\frac{k+1}{N-i-1}\right)\\
&\qquad + (k+1)\log\left(1 + \frac{k-i}{N}\right) -(k+1)\Bigr].
\end{align*}
Rewriting the previous display a little, we obtain
\begin{equation}
\begin{split}
\MoveEqLeft[6]
\frac{\Theta(1)q^k}{(1-q)^k}\cdot k^{-1}\cdot\frac{1}{k+1}\sum_{i=0}^k\frac{1}{\tfrac{i}{k}(1-\tfrac{i}{k})}\exp\Bigl[N\Bigl\{\tfrac{i}{N}\log q^{-1} + \frac{2k}{N}\cdot H(i/k)\\
& + \left(1-\frac{i+1}{N}\right)\log\left(1+\frac{k+1}{N-i-1}\right) + \frac{k+1}{N}\log\left(1+\frac{k-i}{N}\right) -\frac{k+1}{N}\Bigr\}\Bigr].\label{e.factorial moment simplified upper bound}
\end{split}
\end{equation}
Note that the expression inside the curly brackets in the exponent is $f_\alpha(i/k)$ with $\alpha=k/N$ and $f_\alpha$ as defined in \eqref{e.f_alpha first defn}.
The sum is upper and lower bounded using Propositions~\ref{p.ledoux sum upper bound} and \ref{p.ledoux sum lower bound} (which perform the Laplace's method-type asymptotics mentioned above) in Appendix~\ref{app.sum asymptotics}, up to a factor of $C\exp(Cq^{1/2}k^3/N^2)$ or its inverse, and for $N\geq N_0$, $k_0\leq k \leq N$, and $q\geq k^{-2}$, by
\begin{align*}
q^{-1/4} k^{1/2}\left[\frac{\left(1+q^{1/2}\right)^2}{q}\right]^k\exp\left(-N\left[(\mu_q-\alpha)\log\left(1-\frac{\alpha}{\mu_q}\right)\right]\right).
\end{align*}
Substituting this into \eqref{e.factorial moment simplified upper bound} yields that $M^q(k,N)N^{-k}$ is equal to
\begin{align}\label{e.M(k,N) near final bound}
\Theta(1)\cdot k^{-3/2} q^{-1/4}\left[\frac{\left(1+q^{1/2}\right)^2}{1-q}\right]^k\exp\left(-N\left[(\mu_q-\alpha)\log\left(1-\frac{\alpha}{\mu_q}\right)\right] \pm O(1)q^{1/2}\frac{k^3}{N^2}\right),
\end{align}
which completes the proof after recalling that $\mu_q = (1+q^{1/2})^2/(1-q)$.
\end{proof}

\section{Translating from factorial to polynomial moments}\label{s.moment bounds}

Now we convert the factorial moment bounds of Theorem~\ref{t.factorial moment asymptotics} to polynomial bounds and thus prove Theorem~\ref{t.meixner poly moment bounds}. The basic idea is the following. Notice that the difference between the bounds in the two theorems is essentially the factor of $\exp(-N[(\mu_q-\alpha)\log(1-\frac{\mu_q}{\alpha})+\alpha])$ in the bounds for the factorial moments, where $\alpha = k/N$. The presence of this factor is a problem because it is much smaller than $O(1)$. So our goal is to show that that factor goes away when we move to the $k$\textsuperscript{th} polynomial moment.

Now, we can write the ratio $[X(X-1)\cdots (X-k+1)]/X^k$ as
$$\prod_{i=1}^{k-1}(1-iX^{-1}) = \exp\left(-\sum_{i=1}^k\log\left(1-\frac{i}{X}\right)\right).$$
If we believe that the support of $X$ essentially has upper boundary $\mu_q N$, then, when considering high powers of $X$, heuristically one should get the correct behavior when replacing $X$ by $\mu_q N$. Doing  so, we can recognize the sum in the previous display as $N$ times a Riemann sum of the integral $\int_0^\alpha \log(1-x/\mu_q)\,\diff x = (\alpha-\mu_q)\log(1-\frac{\alpha}{\mu_q})-\alpha$. This exactly cancels the extra factor we noted above.

Because this factor which needs to be cancelled will come up a number of times in this section, let us adopt some notation for it. For $\alpha\in(0,1)$, define $h_\alpha:[\alpha,\infty)\to\R$ by
\begin{align}\label{e.h_alpha}
h_\alpha(x) = (\alpha-x)\log\left(1-\frac{\alpha}{x}\right) - \alpha.
\end{align}
The proof we provide does not exactly make the above heuristic precise. While the conversion between factorial and polynomial expressions outlined above essentially holds for deterministic $x$ (see for example Lemma~\ref{l.polynomial moment upper bound} ahead), due to the randomness of $X$, the procedure is not able to provide sharp estimates on the $k$\textsuperscript{th} moment of $X$ in the entire range of $q$ and $k$ that we will require. What we will do instead is obtain sharp upper bounds on the upper tail of $X$ using the factorial moments, and then use the general formula (sometimes called the layer cake representation) $\E[X^k] = \int_0^\infty kx^{k-1}\P(X>x)\,\diff x$ and variants to obtain the sharp estimates on $\E[X^k]$ in the complete parameter range we need (as was discussed briefly in the introduction in Section~\ref{s.intro.meixner analysis}).

(In fact, for lower order moments it is possible to make the heuristic precise in a more direct way, without making use of the layer cake representation. For example, for moments of order $N^{2/3}$, the factor by which $\E[X^k]$ and $\E[(X)_k]$ differs becomes of unit order, and it is possible to approximate $h_\alpha(x)$ by $-\alpha^2/(2x) \pm O(\alpha^3)$ where recall $\alpha = k/N$, since $\alpha^3N = k^3/N^2 = \Theta(1)$ when $k=\Theta(N^{2/3})$. But because we require moments up to order $N$, the errors in these approximations become much too large and so we must keep all the expressions, such as $h_\alpha$, as they are and track all the cancellations.)

In the next Section~\ref{s.general moments lemmas} we collect some general statements connecting factorial and polynomial moments. In Section~\ref{s.upper bound on X upper tail} we use the estimates on factorial moments to obtain sharp upper bounds on $\P(X>x)$. Finally in Sections~\ref{s.general lemma application upper} and \ref{s.general to meixner, lower} we will use these estimates and the layer cake representation to obtain the upper and lower bounds respectively on $\E[X^k]$.

\subsection{General lemmas connecting factorial and polynomial moments}\label{s.general moments lemmas}

\begin{lemma}\label{l.polynomial moment upper bound}
Let $k, N, x\in\N$ with $k\leq x$ and let $\alpha = k/N$, $\beta=x/N$. Then, with $h_\alpha$ as in \eqref{e.h_alpha},
\begin{align*}
(x)_k \geq x^k\exp\Bigl(Nh_\alpha(\beta)\Bigr).
\end{align*}

\end{lemma}

\begin{proof}
We see that
\begin{align*}
(x)_k = x(x-1)\cdots(x-k+1)
= x^k \times \prod_{i=1}^{k-1}\left(1-\frac{i}{x}\right)
&=x^k\cdot\exp\left(\sum_{i=1}^{k-1}\log\left(1-\frac{i}{x}\right)\right).
\end{align*}
Now, the RHS is lower bounded by
\begin{align*}
x^k\cdot\exp\left(\sum_{i=1}^{k-1}\log\left(1-\frac{i}{\beta N}\right)\right) = x^k \cdot\exp\left(\sum_{i=1}^{\alpha N-1}\log\left(1-\frac{i}{\beta N}\right)\right).
\end{align*}
We recognize the sum in the exponent as $N$ times the left Riemann sum associated to the integral $\int_0^\alpha \log(1-x/\beta)\,\diff x = (\alpha-\beta)\log(1-\frac{\alpha}{\beta}) - \alpha$. Since $x\mapsto \log(1-x/\beta)$ is a decreasing function, the integral is a lower bound for the left Riemann sum, i.e., the previous display is lower bounded by
\begin{align*}
x^k \cdot\exp\left(N\left[(\alpha-\beta)\log\left(1-\frac{\alpha}{\beta}\right)-\alpha\right]\right).
\end{align*}
Rearranging completes the proof.
\end{proof}

\begin{lemma}\label{l.general polynomial moment lower bound}
Let $X$ be a non-negative integer-valued random variable. Then for any $k, N\in\N$ with $k\leq N/2$, $\alpha:= k/N$, and $\beta> \alpha$, with $h_\alpha$ as in \eqref{e.h_alpha}, 
$$\E[X^{k}] \geq \left(\E[(X)_k] - \E[(X)_k\one_{X> \beta N}]\right)\cdot\exp\Bigl(-Nh_\alpha(\beta)-\tfrac{1}{4}(\beta-\alpha)^{-1}\Bigr).$$
\end{lemma}

\begin{proof}
As in the previous proof, we see that, under the condition that $x\leq \beta N$,
\begin{align*}
(x)_k = x^k\times \prod_{i=1}^{k-1}\left(1-\frac{i}{x}\right)
= x^k\times \exp\left(\sum_{i=1}^{k-1}\log\left(1-\frac{i}{x}\right)\right)
\leq x^k\times \exp\left(\sum_{i=1}^{\alpha N-1}\log\left(1-\frac{i}{\beta N}\right)\right).
\end{align*}
We again recognize the sum as $N$ times the left Riemann sum of $\int_0^\alpha \log(1-x/\beta)\,\diff x = (\alpha-\beta)\log(1-\frac{\alpha}{\beta}) - \alpha$. The absolute value of the deviation between the Riemann sum of $f$ and the associated integral is bounded by $\smash{\sup_{x\in[0,\alpha]}|f'(x)|\alpha/N}$ (e.g. by bounding $|f(x)-f(i/N)|$ on $[i/N, (i+1)/N]$ by $\sup_{[0,\alpha]}|f'|/N$,  integrating the inequality on $[i/N, (i+1)/N]$, and summing over $i=1, \ldots \alpha N$), and we can evaluate this expression to be  $\smash{\alpha/((\beta-\alpha)N)}$ for our particular $f$. So we see that, when $x\leq \beta N$,
\begin{align}\label{e.deterministic factorial to polynomial upper bound}
(x)_k\leq x^k \exp\left(N\left[(\alpha-\beta)\log\left(1-\frac{\alpha}{\beta}\right) - \alpha\right] + (\beta-\alpha)^{-1}\alpha\right).
\end{align}

So with this, and since $X$ is non-negative, for any $\beta>\alpha$ and since $\alpha\leq \frac{1}{2}$,
\begin{align*}
\E[X^k] 
\geq \E[X^k\one_{X\leq \beta N}] 
&\geq \E\left[(X)_k\one_{X\leq \beta N}\right]\cdot\exp\left(N\left[(\beta-\alpha)\log\left(1-\frac{\alpha}{\beta}\right) + \alpha\right]-\tfrac{1}{2}(\beta-\alpha)^{-1}\right).
\end{align*}
Writing $(X)_k\one_{X\leq \beta N}$ as $(X)_k - (X)_k\one_{X> \beta N}$ completes the proof.
\end{proof}

In the next lemma we record some basic properties of $h_\alpha$ which will be useful ahead.

\begin{lemma}\label{l.h_alpha properties}
Recall $h_\alpha(x) = (\alpha-x)\log(1-\frac{\alpha}{x})-\alpha$ from \eqref{e.h_alpha}. Then $h_\alpha$ is increasing and, for $\alpha\in(0,\frac{1}{2}]$ and $x\geq y$,
\begin{align*}
h_\alpha(x) - h_\alpha(y) \leq \begin{cases}
\alpha^2y^{-2}(x-y) & y\geq 2\alpha\\
\alpha^2 & y\geq 1.
\end{cases}
\end{align*}
In particular, if $\alpha\in(0,\frac{1}{2}]$ and $x\geq y\geq 1$,
\begin{align*}
h_\alpha(x) - h_\alpha(y) \leq \alpha^2\bigl(y^{-2}(x-y)\wedge 1\bigr).
\end{align*}
\end{lemma}

\begin{proof}
That $h_\alpha$ is increasing is easily checked by differentiating.

Since $x\geq y$, $h_\alpha(x) \geq h_\alpha(y)$ as $h_\alpha$ is increasing. The lemma follows from the facts that (i) $h_\alpha$ is concave for $x>\alpha$, so that $h_\alpha(x) - h_\alpha(y) \leq (x-y) h_\alpha'(y)$, and noting by direct calculation that $h_\alpha'(y) = -\alpha/y-\log(1-\alpha/y) \leq \alpha^2/y^2$ for all $y\geq 2\alpha$ and $\alpha \in[0,\frac{1}{2}]$; and (ii) since $h$ is increasing and $y\geq 1$, $h_\alpha(x) - h_\alpha(y) \leq \lim_{z\to\infty} h_\alpha(z) - h_\alpha(1) = \alpha+(1-\alpha)\log(1-\alpha) \leq \alpha^2$ for all $\alpha\in[0,1]$.
\end{proof}

\subsection{Sharp upper bounds on the upper tail of $\nu_{q,N}$}\label{s.upper bound on X upper tail}

As we outlined earlier, for our ultimate goal of sharp asymptotics for $\E[X^k]$, we will need to obtain sharp upper bounds on expectations such as $\E[X^k]$ and $\E[(X)_k\one_{X\geq t}]$ (as appeared in Lemma~\ref{l.general polynomial moment lower bound}). For this we will need upper bounds on $\P(X\geq t)$ for appropriate $t$, which will then be used to estimate the expectations using the ``layer cake'' representation (i.e., a general form of the identity $\E[X^k] = \int_0^\infty kt^{k-1}\P(X\geq t)\,\diff t$).

In this section we obtain these sharp upper bounds on $\P(X\geq t)$, where $t = \mu_qN(1+q^{1/6}\varepsilon)$. The nature of the tail differs between $\varepsilon \in (-1,0)$ and $\varepsilon >0$; in the former case, it decays polynomially in $\varepsilon$ (indeed, Theorem~\ref{t.mean empirical law lower bound} asserts a lower bound of order $\varepsilon^{3/2}$, and we will prove a matching upper bound), while in the latter case it decays exponentially. These are captured in the next two propositions. We recall the definitions of $\mu_q$ from \eqref{e.mu_q definition} and the distribution $\nu_{q,N}$ as defined in \eqref{e.nu_q,N definition}.

\begin{proposition}\label{p.upper tail bound inner}
Let $X \sim\nu_{q,N}$. There exist positive constants $C$ and $N_0$ such that, for $N\geq N_0$, $\varepsilon\in(N^{-2/3},1)$ and $q\in[N^{-2},1)$,
\begin{align*}
\P\left(X\geq \mu_qN(1-q^{1/6}\varepsilon)\right) \leq C \varepsilon^{3/2}.
\end{align*}
\end{proposition}

\begin{proposition}\label{p.crude upper tail bound}
Let $X \sim\nu_{q,N}$. There exist positive constants $c$, $C$ and $N_0$ such that, for $N\geq N_0$, $\varepsilon > 0$, and $q\in[N^{-2}, 1)$, 
\begin{align*}
\P\left(X\geq \mu_qN(1+q^{1/6}\varepsilon)\right) \leq 
\begin{cases}
C \varepsilon^{-3/4} N^{-3/2}\exp\left(-c\varepsilon^{3/2}N\right) & \varepsilon\in(0,q^{1/3}]\\
C(q^{1/6}N)^{-3/2}(1+q^{1/6}\varepsilon)^{-cN} & \varepsilon \geq q^{1/3}.
\end{cases}
\end{align*}
When $\varepsilon\in[q^{1/3}, q^{-1/6}]$, we may replace $(1+q^{1/6}\varepsilon)^{-cN}$ by $\exp(-cq^{1/6}\varepsilon N)$ in the second inequality.
\end{proposition}

Observe the change in behaviour of the bound at $\varepsilon = q^{1/3}$; in particular, the coefficients in the exponent depend on $q$ when $\varepsilon > q^{1/3}$ and are no longer uniform. The reason for the change in the nature of the bound can be understood by recalling that the upper tail of $X$ is closely related to the upper tail of $T_N$, the LPP value in geometric LPP: $\P(T_N\geq t-N+1) \leq N\cdot \P(X\geq t)$ (we will prove this as well as use it in the proof of Theorem~\ref{t.uniform upper tail} on the upper tail of $T_N$ at the end of Section~\ref{s.upper bound on X upper tail} ahead). Setting $t=\mu_qN(1+q^{1/6}\varepsilon)$, we see that
$$\P\left(T_N\geq (\mu_q-1)N + \mu_qq^{1/6}\varepsilon N\right) \leq \P(X\geq \mu_qN(1+q^{1/6}\varepsilon)).$$
Now $\mu_q-1= O(q^{1/2}/(1-q))$ and, when $\varepsilon = q^{1/3}$, the deviation $\mu_qq^{1/6}\varepsilon N$ equals $O(q^{1/2}N/(1-q))$ as well, and so we are in the large deviation regime of geometric LPP. While in the moderate deviation regime we would expect universality (and hence the bounds are uniform in $q$), in the large deviation regime the bound should be expected to depend on the last passage percolation vertex distribution; in this case it is $\mrm{Geo}(q)$, and so it is not surprising that the bound depends on~$q$.

To prove Proposition~\ref{p.crude upper tail bound}, the basic idea will be to use Markov's inequality with the $k$\textsuperscript{th} factorial moment of $X$ for a $k$ which will be optimized over.

\begin{proof}[Proof of Proposition~\ref{p.upper tail bound inner}]
We may assume $\varepsilon <\frac{1}{2}$ by modifying the constant $C$. By Markov's inequality, Theorem~\ref{t.factorial moment asymptotics}, and Lemma~\ref{l.polynomial moment upper bound} (to lower bound the denominator), for any $N\geq N_0$, $k_0\leq k\leq \tfrac{1}{2}N$, and $q\in(0, 1)$ (and using that $k\leq N/2$, $\varepsilon < \frac{1}{2}$, and $q<1$ guarantees $\mu_qN(1-q^{1/6}\varepsilon)\geq k$ to apply Lemma~\ref{l.polynomial moment upper bound}),
\begin{align*}
\P\left(X\geq \mu_qN(1-q^{1/6}\varepsilon)\right)
&\leq \frac{\E[(X)_k]}{(\mu_qN(1-q^{1/6}\varepsilon))_k}\\
&\leq \frac{C(q^{1/6}k)^{-3/2}(\mu_qN)^k\exp\left(Nh_\alpha(\mu_q) + Cq^{1/2}\frac{k^3}{N^2}\right)}{(\mu_qN)^k(1-q^{1/6}\varepsilon)^k\exp\left(Nh_\alpha(\mu_q(1-q^{1/6}\varepsilon))\right)},
\end{align*}
where $\alpha=k/N$ and $h_\alpha(x) = (\alpha-x)\log(1-\frac{\alpha}{x})-\alpha$ as in \eqref{e.h_alpha}.
We will ultimately pick $\varepsilon<\frac{1}{2}$. Since $(1-x)^{-1} \leq \exp(2x)$ for $x\in[0,\frac{1}{2}]$, $\mu_q\geq 1$, and $|h_\alpha(x) - h_\alpha(y)|\leq \alpha^2y^{-2}(x-y)$ for $x\geq y\geq 2\alpha$ from Lemma~\ref{l.h_alpha properties}, we obtain
\begin{align*}
\P\left(X\geq \mu_qN(1-q^{1/6}\varepsilon)\right) \leq C(q^{1/6}k)^{-3/2}\exp\left(2q^{1/6}\varepsilon k+ Cq^{1/6}\varepsilon\frac{k^2}{N}+ Cq^{1/2}\frac{k^3}{N^2}\right).
\end{align*}
Setting $k=q^{-1/6}\varepsilon^{-1}$ (note that $k\leq N/2$ if $\varepsilon \geq 2^{2/3}N^{-2/3}$ since $q\geq N^{-2}$), and using that $k\leq N$, we obtain
\begin{align*}
\P\left(X\geq \mu_qN(1-q^{1/6}\varepsilon)\right)%
&\leq C\varepsilon^{3/2}\exp\left(C  + C\varepsilon^{-3}N^{-2}\right)
\leq C\varepsilon^{3/2},
\end{align*}
the last inequality using that $\varepsilon\geq N^{-2/3}$.
\end{proof}

\begin{proof}[Proof of Proposition~\ref{p.crude upper tail bound}]
We address the case of $\varepsilon\in[0,q^{1/3}]$ first. As in the previous proof, by Markov's inequality and Theorem~\ref{t.factorial moment asymptotics}, for any $N\geq N_0$, $k_0\leq k\leq \tfrac{1}{2}N$, and $q\in (0, 1)$,
\begin{align}
\P\left(X\geq \mu_qN(1+q^{1/6}\varepsilon)\right)
&\leq \frac{\E[(X)_k]}{(\mu_qN(1+q^{1/6}\varepsilon))_k}\nonumber\\
&\leq \frac{C(\mu_qN)^k(q^{1/6}k)^{-3/2}\exp\left(Nh_\alpha(\mu_q)+Cq^{1/2}\frac{k^3}{N^2}\right)}{(\mu_qN)^k(1+q^{1/6}\varepsilon)^k\exp\left(Nh_\alpha(\mu_q(1+q^{1/6}\varepsilon))\right)}.\label{e.markov bound earlier} 
\end{align}
Since $(1+x)^{-1} \leq \exp(-x/2)$ for $x\in[0, 1]$, and since $h_\alpha$ is increasing (as recorded in Lemma~\ref{l.h_alpha properties}), we obtain
\begin{align*}
\P\left(X\geq \mu_qN(1+q^{1/6}\varepsilon)\right)
&\leq C(q^{1/6}k)^{-3/2}\exp\left(-cq^{1/6}k\varepsilon + Cq^{1/2}\frac{k^3}{N^2}\right),
\end{align*}
the last inequality for $k\leq \frac{1}{2}cC^{-1}N$.
Setting $k=\delta \varepsilon^{1/2}q^{-1/6}N$ for a constant $\delta>0$ to be set shortly, and using that $\varepsilon\leq q^{1/3}$ (so that $k\leq \frac{1}{2}cC^{-1}N$ if $\delta<\frac{1}{2}cC^{-1}$), we obtain
\begin{align*}
\P\left(X\geq \mu_qN(1+q^{1/6}\varepsilon)\right) \leq C_{\delta}\varepsilon^{-3/4}N^{-3/2}\exp\left(-c\delta\varepsilon^{3/2}N + C\delta^3\varepsilon^{3/2}N\right).
\end{align*}
Picking $\delta$ to be an appropriately small absolute constant completes the proof in the case $\varepsilon\leq q^{1/3}$. 

Next we turn to the the case of $\varepsilon\geq q^{1/3}$. We look again at \eqref{e.markov bound earlier}. %
Using that $h_\alpha(\mu_q(1+q^{1/6}\varepsilon)) > h_\alpha(\mu_q)$ as $h_\alpha$ is increasing (from Lemma~\ref{l.h_alpha properties}) and taking $k=\alpha N$ in \eqref{e.markov bound earlier} for a small absolute constant $\alpha$ to be chosen, we conclude that
\begin{align*}
\P\left(X\geq \mu_q N(1+q^{1/6}\varepsilon)\right)
&\leq C(q^{1/6}N)^{-3/2}(1+q^{1/6}\varepsilon)^{-\alpha N}\exp\left(Cq^{1/2}\alpha^3N\right)\\
&\leq C(q^{1/6}N)^{-3/2}(1+q^{1/6}\varepsilon)^{-\alpha N/2}
\end{align*}
for all small enough $\alpha$ (using that $q^{1/6}\varepsilon \geq q^{1/2}$). In the case that $q^{1/6}\varepsilon\leq 1$, it holds that $(1+q^{1/6}\varepsilon)^{-1}\leq \exp(-cq^{1/6}\varepsilon)$, which provides the claimed bound in the remaining case. This completes the proof of Proposition~\ref{p.crude upper tail bound}.
\end{proof}

With these upper tail estimates of $X\sim\nu_{q,N}$ available, we can also quickly give the proof of Theorem~\ref{t.uniform upper tail} on the uniform upper tail of geometric LPP:

\begin{proof}[Proof of Theorem~\ref{t.uniform upper tail}]

As in the proof of Theorem~\ref{t.uniform lower tail}, we have that $T_N\stackrel{d}{=} \lambda_1-N+1$, where $(\lambda_1, \ldots, \lambda_N)$ is the Meixner ensemble. So for any $t\in \R$,
\begin{align}\label{e.upper tail to nu}
\P(T_N\geq t-N+1) = \P(\lambda_1\geq t) \leq \P\left(N^{-1}\sum_{i=1}^N \delta_{\lambda_i/N}([t, \infty)) \geq \frac{1}{N}\right) \leq N\nu_{q,N}([t,\infty)),
\end{align}
where that last inequality is Markov's inequality and recall $\nu_{q,N} = \E[N^{-1}\sum_{i=1}^N\delta_{\lambda_i/N}]$ is the expected empirical distribution of the Meixner ensemble defined in \eqref{e.nu_q,N definition}.

We set $t=\mu_qN(1+q^{1/6}\varepsilon)$ for an $\varepsilon\in(0,1)$ to be set later. Now if we let $X$ be distributed as $\nu_{q,N}$, we see that we want to estimate
\begin{align*}
N\nu_{q,N}([\mu_qN(1+q^{1/6}\varepsilon),\infty)) = N\cdot\P\Bigl(X\geq \mu_qN(1+q^{1/6}\varepsilon)\Bigr).
\end{align*}
By the first case of Proposition~\ref{p.crude upper tail bound}, the previous display is upper bounded by
\begin{align*}
C\varepsilon^{-3/4}N^{-1/2}\exp(-c\varepsilon^{3/2}N)
\end{align*}
when $\varepsilon < q^{1/3}$. Taking $\varepsilon = xN^{-2/3}$ and noting that $x < q^{1/3}N^{2/3}$ implies $\varepsilon < q^{1/3}$ completes the proof in this range of $x$. For $x>q^{1/3}N^{2/3}$, doing the above using the bound from Proposition~\ref{p.crude upper tail bound} for $\varepsilon>q^{1/3}$ completes the proof.
\end{proof}

\subsection{The upper bound on $\E[X^k]$}\label{s.general lemma application upper}
In this section we apply Lemmas~\ref{l.polynomial moment upper bound} to obtain the upper bound on $\E[X^k]$, and we will turn to the lower bounds in Section~\ref{s.general to meixner, lower}. The precise upper bound is the following. (Recall from \eqref{e.mu_q definition} that $\mu_q = (1+q^{1/2})^2/(1-q)$.)

\begin{proposition}\label{p.stronger poly moment upper bound meixner}
Let $X$ be distributed according to $\nu_{q,N}$ as defined in \eqref{e.nu_q,N definition} and $\mu_q$ be as in \eqref{e.mu_q definition}. There exist positive constants $C$, $k_0$, $\alpha_0$, and $N_0$ such that, for any $N\geq N_0$ and $(k,q)$ satisfying $k_0 \leq k\leq \min(\alpha_0N, q^{-1/6}N^{2/3})$ and $q\in [N^{-2},1)$, it holds that
\begin{align*}
\E[X^k] \leq C(q^{1/6}k)^{-3/2}(\mu_q N)^k\exp\left(Cq^{1/2}\frac{k^3}{N^2}\right).
\end{align*}

\end{proposition}

We note that, for the range of $k$ under consideration, we can absorb the factor $\exp(Cq^{1/2}k^3/N^2)$ into the constant factor $C$ (and thus match the statement of Theorem~\ref{t.meixner poly moment bounds}). We keep this factor in this statement merely to match the form of Proposition~\ref{p.polynomial lower bound meixner} ahead where it cannot be similarly absorbed, and because it will appear naturally in the proof.

\begin{proof}[Proof of Proposition~\ref{p.stronger poly moment upper bound meixner}]

We start with the breakup
\begin{align*}
\E[X^k]
&= \E[X^k\one_{X\leq \mu_qN(1-q^{1/6})}] + \E[X^k\one_{X> \mu_qN(1-q^{1/6})}]\\
&\leq (\mu_q N)^k\exp(-q^{1/6}k) + \E[X^k\one_{X> \mu_qN(1-q^{1/6})}].
\end{align*}
Next we estimate $\E[X^k\one_{X> \mu_qN(1-q^{1/6})}]$. 
Using the fundamental theorem of calculus and Fubini's theorem to write $\E[X^k\one_{X>s}] = \int_0^\infty kt^{k-1}\P(X>\max(s,t))\,\diff t$,
\begin{align*}
\E[X^k\one_{X> \mu_qN(1-q^{1/6})}]
&= \int_0^\infty kt^{k-1}\P\left(X> \max(\mu_qN(1-q^{1/6}), t)\right)\,\diff t\\
&= (\mu_q N(1-q^{1/6}))^k\P\left(X\geq \mu_q N(1-q^{1/6})\right) + \int_{\mu_q N(1-q^{1/6})}^\infty kt^{k-1}\P\left(X> t\right)\,\diff t\\
&\leq (\mu_qN)^{k}\exp(-q^{1/6}k)\\
&\qquad + (q^{1/6}k) (\mu_qN)^k \int_{-1}^{\infty} (1+q^{1/6}s)^{k-1} \P\left(X\geq \mu_qN(1+q^{1/6}s)\right)\,\diff s,
\end{align*}
performing a change of variable $t\mapsto \mu_qN(1+q^{1/6}s)$ in the last line and using $1-x\leq \exp(-x)$.

We focus on the second term in the previous display. The tail probability in the integrand behaves differently for $s\in(-1,0)$ and $s>0$, as captured in Propositions~\ref{p.upper tail bound inner} and \ref{p.crude upper tail bound}, and so we break up the integral into two parts on this basis. Doing so, and doing a change of variable $s\mapsto -s$ in the resulting first integral, we obtain that the second term in the previous display equals
\begin{equation}\label{e.kth moment upper bound break up}
\begin{split}
\MoveEqLeft[6]
(q^{1/6}k) (\mu_qN)^k \int_{0}^{1} (1-q^{1/6}s)^{k-1} \P\left(X\geq \mu_qN(1-q^{1/6}s)\right)\,\diff s\\
&+ (q^{1/6}k) (\mu_qN)^k \int_{0}^{\infty} (1+q^{1/6}s)^{k-1} \P\left(X\geq \mu_qN(1+q^{1/6}s)\right)\,\diff s.
\end{split}
\end{equation}
We start with the first term. Using that $1-x\leq \exp(-x)$ and Proposition~\ref{p.upper tail bound inner} (note that this proposition requires $\varepsilon > N^{-2/3}$, which necessitates the further break up below),
\begin{align*}
\MoveEqLeft[10]
(q^{1/6}k) (\mu_qN)^k \int_{0}^{1} (1-q^{1/6}s)^{k-1} \P\left(X\geq \mu_qN(1-q^{1/6}s)\right)\,\diff s\\
&= (q^{1/6}k) (\mu_qN)^k \int_{0}^{N^{-2/3}} (1-q^{1/6}s)^{k-1} \P\left(X\geq \mu_qN(1-q^{1/6}s)\right)\,\diff s\\
&\qquad + (q^{1/6}k) (\mu_qN)^k \int_{N^{-2/3}}^{1} (1-q^{1/6}s)^{k-1} \P\left(X\geq \mu_qN(1-q^{1/6}s)\right)\,\diff s\\
&\leq (q^{1/6}k) (\mu_qN)^k N^{-2/3} \P\left(X\geq \mu_q N(1-q^{1/6}N^{-2/3})\right)\\
&\qquad+ (q^{1/6}k) (\mu_qN)^k \int_{N^{-2/3}}^{1} s^{3/2}\exp(-q^{1/6}sk) \,\diff s\\
&\leq C(q^{1/6}k) (\mu_qN)^k \left[N^{-5/3} + (q^{1/6}k)^{-5/2}\right].
\end{align*}
Now we turn to the second term of \eqref{e.kth moment upper bound break up}. By applying the bound on the tail probability from Proposition~\ref{p.crude upper tail bound} (in the two cases of $s< q^{1/3}$ and $s>q^{1/3}$) and using that $1+x\leq \exp(x)$, we obtain that the second term of \eqref{e.kth moment upper bound break up} is upper bounded by
\begin{equation}\label{e.upper tail split up}
\begin{split}
\MoveEqLeft[16]
C(q^{1/6}k) (\mu_qN)^k \biggl[N^{-3/2}\int_{0}^{q^{1/3}} s^{-3/4}\exp\left(q^{1/6}ks-cs^{3/2}N\right) \,\diff s\\
& + q^{-1/4}N^{-3/2}\int_{q^{1/3}}^\infty (1+q^{1/6}s)^{k-1-cN} \,\diff s\biggr].
\end{split}
\end{equation}
We write $k$ in the exponent as $\alpha N$ and perform Laplace's method to bound the first integral, with $f(s) = q^{1/6}\alpha s - cs^{3/2}$, $s_0=\mrm{argmax}_s f(s) = \frac{4}{9}c^{-2} q^{1/3}\alpha^2$ and $f''(s_0) = -\frac{3}{4}cs_0^{-1/2}$, to obtain that the first term (i.e., the first integral times its complete coefficient) of the previous display is upper bounded by (also using that $\alpha = k/N$)
\begin{align}
\MoveEqLeft[6]
C(\mu_q N)^k q^{1/6}k \cdot N^{-3/2} (N|f''(s_0)|)^{-1/2}\cdot s_0^{-3/4}\cdot \exp(Nf(s_0))\nonumber\\
&= C(\mu_q N)^k q^{1/6}k \cdot N^{-3/2} (Nq^{-1/6}\alpha^{-1})^{-1/2}\cdot q^{-1/4}\alpha^{-3/2}\cdot \exp\left(Nq^{1/2}\frac{k^3}{N^2}\right)
\nonumber\\
&= 	C N^{-1}(\mu_q N)^k \exp\left(Cq^{1/2}\frac{k^3}{N^2}\right).
\end{align}
Under the condition that $k\leq cN/2$, the second term in \eqref{e.upper tail split up} is upper bounded (using that $k\leq N$ and $(1+q^{1/2})^{-1}\leq 1$ for the displayed inequality) by
\begin{align*}
C(q^{1/6}k)(\mu_qN)^k \cdot q^{-1/4}N^{-3/2}\cdot \frac{q^{-1/6}(1+q^{1/2})^{-cN/2}}{cN} \leq q^{-1/4}k^{-3/2}(\mu_qN)^k.
\end{align*}
Putting it all together, we have shown that
\begin{align*}
\E[X^k] \leq C(\mu_q N)^k\left[2\exp(-q^{1/6}k) + (q^{1/6}k)N^{-5/3} + 2(q^{1/6}k)^{-3/2} + N^{-1}\exp\left(Cq^{1/2}\frac{k^3}{N^2}\right)\right].
\end{align*}
Clearly $\exp(-q^{1/6}k) \leq C(q^{1/6}k)^{-3/2}$. The condition that $k\leq q^{-1/6}N^{2/3}$ implies that $(q^{1/6}k)N^{-5/3}\leq (q^{1/6}k)^{-3/2}$ and $N^{-1}\leq (q^{1/6}k)^{-3/2}$. Thus we obtain that the RHS in the previous display is upper bounded by 
$$C(q^{1/6}k)^{-3/2}(\mu_qN)^k\exp\left(Cq^{1/2}\frac{k^3}{N^2}\right),$$
completing the proof.
\end{proof}

\subsection{The lower bound on $\E[X^k]$} \label{s.general to meixner, lower}

Here is the lower bound statement we prove:

\begin{proposition}\label{p.polynomial lower bound meixner}
Let $X$ be distributed according to $\nu_{q,N}$ as defined in \eqref{e.nu_q,N definition} and $\mu_q$ be as in \eqref{e.mu_q definition}. Then there exist positive constants $c$, $C$, $\alpha_0$, $N_0$, and $k_0$ such that, for any $N\geq N_0$, $k_0 \leq k\leq \alpha_0N$, and $q\in[k^{-2},1)$,
\begin{align*}
\E[X^k] \geq c(q^{1/6}k)^{-3/2}(\mu_q N)^k\exp\left(-Cq^{1/2}\frac{k^3}{N^2}\right) .
\end{align*}

\end{proposition}

The basic idea is to combine Lemma~\ref{l.general polynomial moment lower bound} with Theorem~\ref{t.factorial moment asymptotics}. Recall that the lower bound from Lemma~\ref{l.general polynomial moment lower bound} has the term $\E[(X)_k] - \E[(X)_k\one_{X\geq \beta N}]$. While the first term can be lower bounded using Theorem~\ref{t.factorial moment asymptotics}, we do not currently have an estimate for the second term. To handle this we will again make use of the fundamental theorem of calculus and Fubini's theorem and make use of Proposition~\ref{p.crude upper tail bound} (upper bound on the upper tail of $X$), as in the proof of Proposition~\ref{p.polynomial lower bound meixner}.

Here too the fact that we need estimates on polynomial moments $\E[X^k]$ with $k$ up to order $N$ creates technical difficulties. For smaller values of $k$, e.g. up to order $N^{2/3}$, one can upper bound $\E[(X)_k\one_{X\geq \beta N}]$ by $\E[X^k\one_{X\geq \beta N}]$. The latter is less than $\E[X^{2k}]^{1/2}\P(X\geq \beta N)^{1/2}$ by Cauchy-Schwarz, and this can be bounded using the already proved upper bound on the polynomial moments and a crude upper bound on the tail probability. This strategy works for smaller $k$ because, when $X=\Theta(N)$ as is typical, $X^k$ and $(X)_k$ differ only by a constant factor; the same certainly does not hold when $k=\Theta(N)$. It is for this reason that we must follow the more delicate path outlined above, making use of a variant of the layer cake formula for $\E[(X)_k\one_{X\geq \beta N}]$ and essentially sharp upper tail bounds for $X$ throughout the tail.

\begin{proof}[Proof of Proposition~\ref{p.polynomial lower bound meixner}]
We will take $\beta = \mu_q(1+q^{1/2})$ and ultimately apply Lemma~\ref{l.general polynomial moment lower bound}. So we need to lower bound $\E[(X)_k] - \E[(X)_k\one_{X>\mu_qN(1+q^{1/2})}]$, in particular, to show that the second term is much smaller than the first. We start with lower bounding the first term using Theorem~\ref{t.factorial moment asymptotics}, which says that, for $N\geq N_0$, $k_0\leq k\leq \tfrac{1}{2}N$, and $q\in[k^{-2},1)$,
\begin{align}
\E[(X)_k]
&\geq C'(q^{1/6}k)^{-3/2}(\mu_qN)^k\exp\left(N\left[h_\alpha(\mu_q) -C'\alpha^3q^{1/2}\right]\right), \label{e.poly moment lower bound step}
\end{align}
where recall from \eqref{e.h_alpha} that $h_\alpha(x) = (\alpha-x)\log(1-\frac{\alpha}{x})-\alpha$ and $\alpha = k/N$.

We need to show that $\E[(X)_k\one_{X\geq \mu_q N(1+q^{1/2})}]$ is smaller than the previous display by at least a constant factor. The idea will be, as in the arguments in Section~\ref{s.general lemma application upper}, to use the fundamental theorem of calculus and Fubini's theorem to write this expectation as an integral against the tail probability. 
\begin{align*}
\E[(X)_k\one_{X \geq \mu_qN(1+q^{1/2})}]
&= \int_0^{\infty} \frac{\diff\,}{\diff x}(x)_k \cdot \P\left(X >  \max\bigl(x, \mu_q N(1+q^{1/2})\bigr)\right)\,\diff x\\
&= \int_0^{\mu_qN(1+q^{1/2})} \frac{\diff\,}{\diff x}(x)_k \cdot \P\left(X > \mu_q N(1+q^{1/2}\bigr)\right)\,\diff x\\
&\qquad + \int_{\mu_qN(1+q^{1/2})}^{\infty} \frac{\diff\,}{\diff x}(x)_k \cdot \P\left(X > x\right)\,\diff x\\
&= \bigl(\mu_qN(1+q^{1/2})\bigr)_k\cdot \P\bigl(X> \mu_q N(1+q^{1/2})\bigr)\\
&\qquad + \int_{\mu_qN(1+q^{1/2})}^{\infty} \frac{\diff\,}{\diff x}(x)_k \cdot \P\left(X > x\right)\,\diff x.
\end{align*}
Now, we can write the derivative of $(x)_k$ as $(\psi(x+1)-\psi(x-k+1))\cdot(x)_k$, where $\psi$ is the digamma function (as can be verified by using the product rule to differentiate $(x)_k$ and the recursive relation $\psi(x+1) = \psi(x) + x^{-1}$ that $\psi$ satisfies). The digamma function satisfies the inequalities, for $x>\frac{1}{2}$,
\begin{align*}
\psi(x) \in \left(\log(x-\tfrac{1}{2}), \log x\right),
\end{align*}
so that, when $x > k-\frac{1}{2}$,
\begin{align*}
\psi(x+1)-\psi(x-k+1) \leq \log(x+1) - \log(x-k+\tfrac{1}{2}) = -\log\left(1-\frac{k+\frac{1}{2}}{x+1}\right) \leq \frac{k+\frac{1}{2}}{x+1}.
\end{align*}
if $k < x/2$ (which holds in our situation since $k < \frac{1}{2}N$ and $x>\mu_qN$ with $\mu_q > 1$). 

We recall that, by the second case of Proposition~\ref{p.crude upper tail bound}, $\P(X>\mu_qN(1+q^{1/2})) \leq q^{-1/4}N^{-3/2}(1+q^{1/2})^{-cN}$.

Using these facts and performing the change of variable $x\mapsto \mu_q N(1+q^{1/2} y)$, we see that
\begin{align}
\MoveEqLeft[2]
\E[(X)_k\one_{X \geq \mu_qN(1+q^{1/2})}]\nonumber\\
&\leq C\left(\mu_q N(1+q^{1/2})\right)_k \cdot q^{-1/4}N^{-3/2}(1+q^{1/2})^{-cN}\nonumber\\
&\qquad + C\mu_q Nq^{1/2}\int_{1}^{\infty} \frac{k}{\mu_q N}(\mu_qN(1+q^{1/2}y))_k \cdot \P\left(X > \mu_qN(1+q^{1/2}y)\right)\,\diff y. \label{e.factorial moment upper tail bound}
\end{align}
We next estimate the falling factorial terms in the above expression. We know from \eqref{e.deterministic factorial to polynomial upper bound} that, for any $\varepsilon>0$ and $k=\alpha N$, (recalling $h_\alpha(x) = (\alpha-x)\log(1-\frac{\alpha}{x})-\alpha$ from \eqref{e.h_alpha})
\begin{align*}
\left(\mu_qN(1+\varepsilon)\right)_k
&\leq C(\mu_qN(1+\varepsilon))^k \exp\Bigl(N h_\alpha(\mu_q(1+\varepsilon))\Bigr)\\
&\leq C(\mu_qN)^k(1+\varepsilon)^k\exp\left(N\Bigl[h_\alpha(\mu_q) + C\alpha^2(\mu_q^{-1}\varepsilon\wedge 1)\Bigr]\right)\\
&\leq C(\mu_qN)^k(1+\varepsilon)^k\exp\left(N\Bigl[h_\alpha(\mu_q) + C\alpha^2(\varepsilon\wedge 1)\Bigr]\right),
\end{align*}
the second inequality using Lemma~\ref{l.h_alpha properties} and third using $\mu_q\geq 1$.

Substituting this bound on the falling factorial into \eqref{e.factorial moment upper tail bound} and using the bounds on the tail probability from the second case of Proposition~\ref{p.crude upper tail bound}, we obtain that $\E[(X)_k\one_{X \geq \mu_qN(1+C_1q^{1/2})}]$ is upper bounded by
\begin{equation}\label{e.factorial moment upper tail breakup}
\begin{split}
\MoveEqLeft[4]
\widetilde C(\mu_qN)^k \exp\left(N\left[(\alpha-\mu_q)\log\left(1-\frac{\alpha}{\mu_q}\right)- \alpha\right]\right)\cdot q^{-1/4}N^{-3/2}\\
& \times \biggl[ \exp\left( -N\left[c\log(1+q^{1/2}) - C\alpha^2q^{1/2}\right]\right)\\
&\qquad\qquad\quad + kq^{1/2} \int_{1}^{\infty} \exp\left(-N\left[c\log(1+q^{1/2}y) - C\alpha^2\left(1\wedge(q^{1/2}y)\right)\right]\right)\, \diff y\biggr].
\end{split}
\end{equation}
To ensure that the coefficient of $N$ in the exponential is negative in both the terms on the second and third lines of the previous display, we will need to restrict how big $\alpha$ can be. More precisely, since $\alpha < \alpha_0$ with $\alpha_0$ an absolute constant which we are free to set, we set it such that, for all $x\in[0,1]$
\begin{align*}
\frac{1}{2}c\log(1+x) > C\alpha_0^2x.
\end{align*}
Then the second line of \eqref{e.factorial moment upper tail breakup} is upper bounded by
\begin{align*}
(1+q^{1/2})^{-\frac{1}{2}cN} \leq \tfrac{1}{2},
\end{align*}
using that $q\geq N^{-2}$.

Similarly the third line of \eqref{e.factorial moment upper tail breakup} is upper bounded by
\begin{align*}
kq^{1/2} \int_{1}^\infty (1+q^{1/2}y)^{-\frac{1}{2}cN}\,\diff y = Ck \frac{(1+q^{1/2})^{-cN}}{N} \leq \tfrac{1}{2}.
\end{align*}
Thus overall we see that there is a $C$ such that
\begin{align*}
\E[(X)_k\one_{X \geq \mu_qN(1+C_1q^{1/2})}]
&\leq Cq^{-1/4}N^{-3/2}(\mu_qN)^k \exp\left(N\left[(\alpha-\mu_q)\log\left(1-\frac{\alpha}{\mu_q}\right)- \alpha\right]\right).
\end{align*}
Then, recalling the lower bound on $\E[(X)_k]$ from \eqref{e.poly moment lower bound step},
\begin{align*}
\frac{\E[(X)_k\one_{X\geq \mu_q N(1+C_1q^{1/2})}]}{\E[(X)_k]}
&\leq C(C')^{-1}\left[\frac{(q^{1/6}k)^{3/2}}{(q^{1/6}N)^{3/2}}\right] \leq C(C')^{-1}\alpha_0^{3/2}
\leq \frac{1}{2},
\end{align*}
the last inequality by reducing $\alpha_0$ further if necessary and using that $k\leq \alpha_0N$ and $q\geq k^{-2}$.

Now applying Lemma~\ref{l.general polynomial moment lower bound} with $\beta=\mu_q(1+q^{1/2})\geq 1$ and $\alpha\leq \frac{1}{2}$ (and noting that thus $(\beta-\alpha)^{-1}\leq 2$), using the previous display and \eqref{e.poly moment lower bound step}, we get
\begin{equation*}
\begin{split}
\E[X^k]
&\geq \frac{1}{2}\E[(X)_k]\cdot\exp\left(N\left[(\beta-\alpha)\log\left(1-\frac{\alpha}{\beta}\right) + \alpha\right]-\tfrac{1}{2}(\beta-\alpha)^{-1}\right)\\
&\geq c(q^{1/6}k^{-3/2})(\mu_qN)^k\exp\left(-C\alpha^3q^{1/2}N\right).\qedhere
\end{split}
\end{equation*}
\end{proof}

\section{Concentration inequalities and the proof of Theorem~\ref*{mt.q-pushtasep bound}}

In this section we combine Theorem~\ref{t.uniform lower tail} (on the uniform tail for geometric LPP) with the representation of the position $x_N(N)$ of the first particle in $q$-pushTASEP in terms of the LPP value in an infinite periodic strip of inhomogeneous geometric random variables, and so obtain an upper bound on the lower tail of $x_N(N)$. Recall from the proof outline given in Section~\ref{s.proof ideas} that the main idea is to lower bound the LPP value by a sum of independent LPP values, each one in an $N\times N$ square; the parameter of the geometric random variables is the same within each single such square, but varies across different ones.

For this argument we need one final ingredient: a concentration inequality for a sum of independent random variables that takes into account the possibly varying scales of the summands. Indeed, we will be considering a sum of geometric LPP values where the parameter of the geometric is $q^i$ for varying $i$; the scale of fluctuation for fixed $i$ is $q^{i/6}/(1-q^{i}) \approx i^{-1}(|\log q|)^{-1} q^{i/6}$. Such a concentration inequality is recorded next, and, as its proof is fairly routine, will be proven in Appendix~\ref{s.concentration}.

\begin{theorem}\label{t.concentration}
Let $I\in\N\cup\{\infty\}$ and suppose $X_1, \ldots, X_I$ are independent, and assume that there exists $C_1<\infty$ and $\rho_1, \ldots, \rho_I >0$ such that each $X_i$ satisfies
\begin{align*}
\P\left(X_i\geq t\right) \leq C_1\exp(-\rho_i t^{3/2})
\end{align*}
for all $t>0$. Let $\sigma_2 = \sum_{i=1}^I \rho_i^{-2}$ and $\sigma_{2/3} = \sum_{i=1}^I \rho_i^{-2/3}$. Then there exist positive absolute constants $C$ and $c$ such that, for $t>0$,
\begin{align*}
\P\left(\sum_{i=1}^I X_i \geq t + C\cdot C_1\sigma_{2/3}\right) \leq \exp\left(-c\sigma_2^{-1/2}t^{3/2}\right).
\end{align*}

\end{theorem}

Using Theorem~\ref{t.concentration} we may give the proof of the main result, Theorem~\ref{mt.q-pushtasep bound}. We will need a simple lower bound on the coefficient rescaling the fluctuations in the definition \eqref{e.definition of X^sc} of $X^{\mrm{sc}}_N$.

\begin{lemma}\label{l.flucutation constant bound}
For $q, u\in(0,1)$,
\begin{align*}
 (-\psi_q''(\log_q u))^{1/3}(\log q^{-1})^{-1} \geq \frac{u^{1/3}}{(1-q)^{1/3}(1-u)^{2/3}}.
 \end{align*}
  
\end{lemma}

\begin{proof}
By \cite[eq. (1.6)]{mansour2009some},
$\psi_q''(x) = (\log q)^3\cdot\sum_{n=1}^\infty \frac{n^2q^{nx}}{1-q^n}.$
Taking $x=\log_q u = \log u/\log q$ yields that
\begin{align*}
(1-q)(-\psi_q''(\log_q u))(\log q)^{-3} = \sum_{n=1}^\infty \frac{n^2u^n(1-q)}{1-q^n} = \sum_{n=1}^\infty \frac{n^2u^n}{1+q+ \ldots +q^{n-1}}.
\end{align*}
Since the denominator is upper bounded by $n$ and $\sum_{n=1}^\infty nu^n = u/(1-u)^2$, we obtain the lemma.
\end{proof}

This bound is not sharp. Indeed, since $(-\psi_q''(\log_q u))(\log q)^{-3} = \sum_{n=1}^\infty n^2u^n/(1-q)^n$, in the $q\to 0$ limit it equals $\sum_{n=1}^\infty n^2 u^n = u/(1-u)^3$ which is of larger order than $u/(1-u)^2$ in the $u\to 1$ regime.

\begin{proof}[Proof of Theorem~\ref{mt.q-pushtasep bound}]

We have to upper bound, for $\theta>\theta_0=\theta_0(q)$,
\begin{align}\label{e.main prob to upper bound earlier}
\P\left(x_N(N) \leq f_qN - (-\psi_q''(\log_q u))^{1/3}(\log q^{-1})^{-1}\theta N^{1/3} \right),
\end{align}
where we recall that $f_q$ is defined in \eqref{e.fq definition} as
\begin{align*}
f_q = 2\cdot\frac{\psi_q(\log_q u) + \log(1-q)}{\log q} + 1.
\end{align*}
By Lemma~\ref{l.flucutation constant bound}, with $\sigma_{u,q} = u^{1/3}(1-q)^{-1/3}(1-u)^{-2/3}$, \eqref{e.main prob to upper bound earlier} is upper bounded by
\begin{align}\label{e.main prob to upper bound}
\P\left(x_N(N) \leq f_qN - \sigma_{u,q}\theta N^{1/3} \right).
\end{align}
We define $\tilde f_q = f_q - 1$. By Theorem~\ref{t.q-pushTASEP to lpp}, we know that $\smash{L + N \stackrel{d}{=} x_N(N)}$, where $L$ is the LPP value from the topmost site to $\infty$ in the infinite periodic environment defined in Section~\ref{s.lpp}. Let $\smash{L_N^{(i)}}$ be the last passage time from the top to the bottom of the $i$\textsuperscript{th} large square on the vertical line from the top (which has \iid Geo($u^2q^{2i}$) random variables associated to each small square). Then clearly $\sum_{i=0}^\infty \smash{L_N^{(i)}} \leq L$, so
\begin{align*}
\eqref{e.main prob to upper bound} = \P\left(L \leq \tilde f_qN - \sigma_{u,q} \theta N^{1/3} \right)
\leq \P\left(\sum_{i=0}^\infty L_N^{(i)} \leq \tilde f_qN - \sigma_{u,q} \theta N^{1/3}\right).
\end{align*}
We next add and subtract the law of large numbers term of $L_{N}^{(i)}$, which as we see from Theorem~\ref{t.uniform lower tail} is $2N uq^i(1-uq^i)^{-1}$, as this is the term by which the random variables are centered to yield the tail bounds in the same theorem. So we can write the right-hand side of the previous display as
\begin{align}\label{e.prob with mean removed}
\P\left(\sum_{i=0}^\infty \left(L_N^{(i)}-2N\cdot\frac{uq^{i}}{1-uq^{i}}\right) \leq \tilde f_qN - \sum_{i=0}^\infty 2N\cdot\frac{uq^{i}}{1-uq^{i}}- \sigma_{u,q}\theta N^{1/3}\right).
\end{align}
We have already evaluated the LLN sum in the proof ideas section. So we recall from \eqref{e.first form of LLN} and \eqref{e.lln expression} that
\begin{align*}
\sum_{i=0}^\infty 2N\cdot\frac{uq^{i}}{1-uq^{i}} 
&= 2N\cdot \frac{\psi_q(\log_q(u)) + \log(1-q)}{\log q} = \tilde f_q N.
\end{align*}
Putting this back into \eqref{e.prob with mean removed}, we see that
\begin{align}\label{e.concentration prob to bound}
\eqref{e.prob with mean removed} = \P\left(\sum_{i=0}^\infty \left(L_N^{(i)}-2N\cdot\frac{uq^{i}}{1-uq^{i}}\right) \leq - \sigma_{u,q} \theta N^{1/3}\right).
\end{align}
In the remainder of the proof we will invoke the concentration bound from Theorem~\ref{t.concentration} to upper bound the previous display.

Now, $L_N^{(i)}$ is the LPP value in an $N\times N$ square with \iid geometric random variables of parameter $u^2q^{2i}$. We know from Theorem~\ref{t.uniform lower tail} that there exist positive constants $c$, $t_0$, and $N_0$ such that, when $u^2q^{2i} \in (0, 1)$, $t>t_0$, and $N>N_0$,
\begin{align*}
\P\left(L_N^{(i)} - 2N\cdot\frac{uq^{i}}{1-uq^{i}} \leq - t\cdot\frac{u^{1/3}q^{i/3}}{1-u^2q^{2i}}N^{1/3}\right) \leq \exp(-ct^{3/2}).
\end{align*}
Equivalently, there exist positive constants $c$, $C_1$, and $N_0$ such that, when $u^2q^{2i} \in (0, 1)$, $t>0$ (i.e., not  $t_0$), and $N>N_0$,
\begin{align*}
\P\left(L_N^{(i)} - 2N\cdot\frac{uq^{i}}{1-uq^{i}} \leq - t\cdot\frac{u^{1/3}q^{i/3}}{1-u^2q^{2i}}N^{1/3}\right) \leq C_1\exp(-ct^{3/2}).
\end{align*}
Letting $\sigma_{i,q,u}' = u^{1/3}q^{i/3}(1-u^2q^{2i})^{-1}$, we see that, for $t>0$,
\begin{align*}
\P\left(L_N^{(i)} - 2N\cdot\frac{uq^{i}}{1-uq^{i}} \leq -t N^{1/3}\right) \leq C_1\exp\left(-c(\sigma_{i,q,u}')^{-3/2} t^{3/2}\right).
\end{align*}
We next want invoke Theorem~\ref{t.concentration} with $I=\infty$ and (from the previous display) $\rho_i = (\sigma_{i,q,u}')^{-3/2}$. Now, since $q\in(0,1)$,
\begin{align*}
\sigma_2 &= \sum_{i=0}^\infty \rho_i^{-2} = \sum_{i=0}^\infty (\sigma_{i,q,u}')^3 = \sum_{i=0}^\infty \frac{uq^i}{(1-u^2q^{2i})^3} \leq \frac{u}{(1-q)(1-u^2)^3}\\
\text{and}\quad
\sigma_{2/3} &= \sum_{i=0}^\infty \rho_i^{-2/3} = \sum_{i=0}^\infty \sigma_{i,q,u}' = \sum_{i=0}^\infty \frac{u^{1/3}q^{i/3}}{1-u^2q^{2i}} \leq \frac{u^{1/3}}{(1-q^{1/3})(1-u^2)}.
\end{align*}

With these estimates, we obtain from Theorem~\ref{t.concentration} that there exist positive $C$ and $c$ such that, for all $t>0$ (and using that $1-u^2 = \Theta(1-u)$),
\begin{align*}
\MoveEqLeft[16]
\P\left(\sum_{i=0}^\infty \left(L_N^{(i)}-2N \cdot\frac{uq^{i}}{1-uq^{i}}\right) \leq -t N^{1/3} - C\frac{u^{1/3}}{(1-q^{1/3})(1-u)}\right)\\
&\leq \exp\left(-cu^{-1/2}(1-q)^{1/2}(1-u)^{3/2}t^{3/2}\right).
\end{align*}
So, putting in $t=2\sigma_{u,q}\theta = u^{1/3}(1-q)^{-1/3}(1-u)^{-2/3}\theta$, we obtain that, for $\theta>C(1-q)^{1/3}/((1-q^{1/3})(1-u)^{1/3})$ with $C$ as in the last display, \eqref{e.concentration prob to bound} is bounded as
\begin{equation*}
\P\left(\sum_{i=0}^\infty \left(L_N^{(i)}-2N\cdot\frac{uq^{i}}{1-uq^{i}}\right) \leq -\sigma_{u,q}\theta N^{1/3}\right) \leq \exp\left(-c(1-u)^{1/2}\theta^{3/2}\right).
\end{equation*}
Note that $(1-q)^{1/3}/(1-q^{1/3}) \leq C'\vee (\log q^{-1})^{-2/3}$ by writing $q=\exp(-\varepsilon)$ for some $C'$. Thus we obtain the desired bound \eqref{e.main prob to upper bound} with $\theta_0(q,u) = C(1-u)^{-1/3}(1\vee (\log q^{-1})^{-2/3})$.
\end{proof}

\appendix

\section{Proof of the LPP--$q$-Whittaker connection}\label{app.q-whittaker and lpp}

In this appendix we give the proof of Theorem~\ref{t.q-pushTASEP to lpp} relating the observable $x_N(T)$ to an infinite last passage problem in a periodic and inhomogeneous environment. As mentioned, the proof goes through an equivalence to the $q$-Whittaker measure, and we start by introducing it.

\subsection{$q$-Whittaker polynomials and measure}
\begin{definition}[$q$-Whittaker polynomial]
For a skew partition $\mu/\lambda$, the skew $q$-Whittaker
polynomial in $n$ variables $\mathscr P_{\mu/\lambda} (x_1,  \ldots  , x_n; q)$ is defined recursively by the branching rule
\begin{align*}
\mathscr P_{\mu/\lambda} (x_1,  \ldots  , x_n; q) = \sum_{\eta}\mathscr P_{\eta/\lambda} (x_1,  \ldots  , x_{n-1}; q)\mathscr P_{\mu/\eta} (x_n; q),
\end{align*}
where, for a single variable $z\in\C$ (recalling the $q$-binomial coefficient defined in \eqref{e.q-binomial coefficient}),
\begin{align*}
\mathscr P_{\mu/\eta}(z; q) = \one_{\eta\prec \mu} \prod_{i\geq 1} z^{\mu_i-\eta_i}\binom{\mu_i-\mu_{i+1}}{\mu_i-\eta_i}_{\!\!q}.
\end{align*}
For a partition $\mu$, the $q$-Whittaker polynomial $\mathscr P_\mu$ is given by the skew $q$-Whittaker polynomial $\mathscr P_{\mu/\lambda}$ with $\lambda$ taken to be the empty partition. The $q$-Whittaker polynomial is a special case ($t=0$) of the Macdonald polynomials, for which a comprehensive reference is \cite[Section VI]{macdonald1998symmetric}.
\end{definition}

For a partition $\mu$, we also define $\mathdutchcal b_{\mu}(q)$ by
\begin{align*}
\mathdutchcal{b}_{\mu}(q) = \prod_{i\geq 1} \frac{1}{(q;q)_{\mu_i-\mu_{i+1}}}.
\end{align*}

\begin{definition}[$q$-Whittaker measure]
The $q$-Whittaker measure $\mathbb{W}_{a;b}^{(q)}$, first introduced in \cite{borodin2014macdonald}, is the measure on the set of all partitions given by
\begin{align*}
\mathbb{W}_{a;b}^{(q)}(\mu) = \frac{1}{\Pi(a;b)} \mathdutchcal{b}_\mu(q) \mathscr P_{\mu}(a;q)\mathscr P_{\mu}(b;q),
\end{align*}
where $a = (a_1, \ldots, a_n)$ and $b = (b_1, \ldots, b_t)$ satisfy $a_i, b_j\in(0,1)$, and $\Pi(a;b)$ is a normalization constant given explicitly by
\begin{align*}
\Pi(a;b) = \prod_{i=1}^n\prod_{j=1}^t \frac{1}{(a_ib_j; q)_{\infty}}.
\end{align*}

\end{definition}

We may now record the important connection between $x_N(T)$ and the $q$-Whittaker measure which holds under general parameter choices and general times, when started from the narrow-wedge initial condition:

\begin{theorem}[Section~3.1 of \cite{matveev2016q}]\label{t.q-pushTASEP and whittaker}
Let $a$, $b$ be specializations of parameters respectively $(a_1,  \ldots  , a_N) \in (0, 1)^N$ and $(b_1,  \ldots  , b_T ) \in (0, 1)^T$. Let $\mu \sim \mathbb{W}^{(q)}_{a;b}$ and let $x(T)$ be a $q$-pushTASEP under initial conditions $x_k(0) = k$ for $k = 1,  \ldots  , N$. Then,
$$x_N(T) \stackrel{d}{=} \mu_1 + N.$$
\end{theorem}

\subsection{$q$-Whittaker to LPP}
Next we give the proof of Theorem~\ref{t.q-pushTASEP to lpp}, which was explained to us by Matteo Mucciconi. As indicated in Section~\ref{s.intro.q-pushtasep to LPP}, the proof we give relies heavily on the work \cite{imamura2021skew}. Before proceeding we introduce some terms that will be needed. First, a tableaux is a filling of a Young diagram with non-negative integers. It is called \emph{semi-standard} if the entries in the rows and columns are non-decreasing, from left to right and top to bottom respectively.
A \emph{vertically strict tableaux} is a tableaux of non-negative integers in which the columns are strictly increasing from top to bottom, but there is no constraint on the row entries. A \emph{skew tableaux} is a pair of partitions $(\lambda, \mu)$ such that the Young diagram of $\lambda$ contains that of $\mu$, and should be thought of as the boxes corresponding to $\lambda\setminus \mu$. A semi-standard skew tableaux is defined analogously to the semi-standard tableaux.

As we saw, Theorem~\ref{t.q-pushTASEP and whittaker} on the relation between $x_N(T)$ and the $q$-Whittaker measure reduces the proof of Theorem~\ref{t.q-pushTASEP to lpp} to proving the equality in distribution of the LPP value $L$ and the length of the top row of a Young diagram $\lambda$ sampled from the $q$-Whittaker measure. In fact, we will prove a stronger statement which relates all the row lengths of $\lambda$ to appropriate last passage percolation observables. For this, we let $L^{(j)}$ be the maximum weight over all collections of $j$ disjoint paths, one path starting from $(i,1)$ for each $1\leq i\leq j$ and all going to $\infty$ downwards, where the weight of a collection of paths is the sum of the weights of the individual paths.

\begin{theorem}\label{t.full q-whittaker to lpp}
For $1\leq j\leq \min(N,T)$, let $L^{(j)}$ be as defined above in the environment defined in Section~\ref{s.lpp} and let $\mu\sim \mathbb{W}^{(q)}_{a;b}$ with $a_i, b_j \in (0,1)$ for all $(i,j)\in\{1, \ldots, N\}\times\{1, \ldots, T\}$. Then, jointly across $1\leq j\leq \min(N,T)$,
$$L^{(j)} \stackrel{d}{=} \mu_1+ \ldots +\mu_j.$$
\end{theorem}

Theorem~\ref{t.q-pushTASEP to lpp} follows immediately from combining the $j=1$ case of Theorem~\ref{t.full q-whittaker to lpp} with Theorem~\ref{t.q-pushTASEP and whittaker}.

\begin{proof}[Proof of Theorem~\ref{t.full q-whittaker to lpp}]
We prove this in the case $T=N$. It is easy to see that the same proof applies to the $T< N$ case by setting $m_{(i,j);k} = 0$ (in the same notation as below) for $j=T+1, \ldots, N$, and similarly for $N<T$.

We will make use of two bijections. The first, known as the Sagan-Stanley correspondence and denoted by $\msf{SS}$, is a bijection between the set of $(M, \nu)$ and the collection of pairs $(P,Q)$ of semi-standard tableaux of general skew shape, where $M=(\smash{m_{(i,j);k}})_{1\leq i,j\leq N, k=0,1, \ldots}$ is a filling of the infinite strip by non-negative integers which are eventually all zero and $\nu$ is a partition. The second is a bijection $\Upsilon$ introduced in \cite{imamura2021skew} between the collection of such $(P,Q)$ and tuples of the form $(V,W; \kappa, \nu)$, where $V, W$ are vertically strict tableaux of shape $\mu$ (which is a function of $P$, $Q$) and $\kappa \in \mathcal K(\mu)$ is an ordered tuple of non-negative integers with certain constraints on the entries depending on $\mu$, encoded by the set $\mathcal K(\mu)$. We will not need the definition of $\mathcal K(\mu)$ for our arguments, but the interested reader is referred to \cite[eq. (1.22)]{imamura2021skew} for it. 

If one applies $\msf{SS}$ to $(M, \nu)$ and then $\Upsilon$ to the result, the $\nu$ in the resulting output $(V,W, \kappa; \nu)$ is the same as the starting one (see the line following \cite[Theorem~1.4]{imamura2021skew}). Thus the $\nu$ can be factored out, yielding a bijection between the set of $M$ and the set of $(V,W,\kappa)$; we call this $\tilde\Upsilon$, as in \cite{imamura2021skew}.

Now, \cite[Theorem~1.2]{imamura2021skew} asserts that $L^{(j)}(M)$ is equal to $\mu_1+ \ldots +\mu_j$, the sum of the lengths of the first $j$ rows in the partition $\mu$ from the previous paragraph, for all $1\leq j\leq N$ (this is a deterministic statement for $L$ defined with respect to any fixed entries of the environment). So we need to understand the distribution of $\mu$ under the map $\tilde \Upsilon$ when the entries $\smash{m_{(i,j);k}}$ of $M$ are distributed as independent $\mrm{Geo}(q^ka_ib_j)$, in particular, show that it is $\smash{\mathbb W^{(q)}_{a,b}}$.

For this task we will need certain weight preservation properties of $\tilde \Upsilon$ which we record in the next lemma.

\begin{lemma}\label{l.weight preservation}
For an infinite matrix  $M$ and $(V,W; \kappa)$ as above, define weight functions
\begin{align*}
W_1(M) &= \Bigl(\sum_k\sum_j m_{(i,j);k}\Bigr)_{1\leq i\leq N}\\
W_2(M) &= \Bigl(\sum_k\sum_i m_{(i,j);k}\Bigr)_{1\leq j\leq N}\\
W_3(M) &= \sum_k\sum_{i,j} k m_{(i,j);k}.
\end{align*}
and, with $\#(U,i)$ being the number of times the entry $i$ appears in the vertically strict tableaux $U$,
\begin{align*}
\widetilde W_1(V,W,\kappa) &= \bigl(\#(V,i)\bigr)_{1\leq i\leq N}\\
\widetilde W_2(V,W,\kappa) &= \bigl(\#(W,j)\bigr)_{1\leq j\leq N}\\
\widetilde W_3(V,W,\kappa) &= \mathcal H(V) + \mathcal H(W) + \sum_{i}\kappa_i,
\end{align*}
where $\mathcal H$ is the ``intrinsic energy function''. Its definition is complicated and not strictly needed for our purposes, so the interested reader is referred to \cite[Definition~7.4]{imamura2021skew} for a precise definition.

Then, if $M$ and $(V,W,\kappa)$ are in bijection via $\tilde \Upsilon$, it holds that, for $i=1,2,3$,
\begin{align*}
W_i(M) = \widetilde W_i(V,W,\kappa).
\end{align*}
\end{lemma}
We will prove this after completing the proof of Theorem~\ref{t.q-pushTASEP to lpp}. We wish to calculate $\P(L^{(k)}(M) = \sum_{i=1}^k\mu_i\ \text{ for all} 1\leq k\leq n)$ where $M$ is distributed according to independent geometric random variables as above, and show that this is equal to the first $k$ row lengths marginal of the $q$-Whittaker measure. We will instead show the stronger statement that the law of the shape of $V$ (or $W$) obtained by applying $\tilde \Upsilon$ to $M$ with $\smash{m_{(i,j);k}} \sim \mrm{Geo}(q^ka_ib_j)$ is the $q$-Whittaker measure. Then marginalizing to the lengths of the first $k$ rows will complete the proof. Denoting the $V$ obtained by applying $\tilde\Upsilon$ to $M$ by $V(M)$ and by $\mrm{vst}(\mu)$ the set of vertically-strict tableau of shape $\mu$,
\begin{align*}
\P\left(V(M) \in \mrm{vst}(\mu)\right)
&\propto \sum_{M:V(M) \in \mrm{vst}(\mu)} \prod_{i,j,k}(q^ka_ib_j)^{m_{(i,j);k}}\\
&= \sum_{M:V(M) \in \mrm{vst}(\mu)} q^{\sum_{i,j,k}km_{(i,j);k}} \prod_{i}a_i^{\sum_j m_{(i,j);k}}\prod_j b_j^{\sum_i m_{(i,j);k}}\\
&= \sum_{M:V(M) \in \mrm{vst}(\mu)}  q^{W_3(M)} a^{W_1(M)} b^{W_2(M)}.
\end{align*}
By applying the bijection $\tilde \Upsilon$ and Lemma~\ref{l.weight preservation}, and recalling the definitions of $\widetilde W_i$, we see that the previous line equals (letting $x^V = \prod_{i} x_i^{\#(V,i)}$)
\begin{align*}
\sum_{\substack{V, W\in \mrm{vst}(\mu) \\ \kappa\in\mathcal K(\mu)}}  q^{\widetilde W_3(V,W,\kappa)} a^{\widetilde W_1(V,W,\kappa)} b^{\widetilde W_2(V,W,\kappa)}
&= \sum_{\kappa\in\mathcal K(\mu)} q^{|\kappa|} \times \sum_{V \in \mrm{vst}(\mu)} q^{\mathcal H(V)}a^{V} \times \sum_{W\in\mrm{vst}(\mu)} q^{\mathcal H(W)}b^{W}.
\end{align*}
Now it is known that the first factor is $\mathdutchcal b_\mu(q)$ (see for example \cite[eq. (10.5)]{imamura2021skew}), while the second and third factors are $\mathscr P_\mu(a; q)$ and $\mathscr P_\mu(b; q)$ (see \cite[Proposition~10.1]{imamura2021skew}). Thus the RHS is the unnormalized probability mass function of $\smash{\mathbb W^{(q)}_{a,b}}$ at $\mu$, as desired. %
\end{proof}

\begin{proof}[Proof of Lemma~\ref{l.weight preservation}]
That $W_3(M) = \widetilde W_3(V,W, \kappa)$ follows by combining eq. (1.23) in \cite[Theorem~1.4]{imamura2021skew} (on the conservation of the quantity under $\Upsilon$) with eq. (4.15) in \cite[Theorem~4.11]{imamura2021skew} (on its conservation under $\msf{SS}$).

For $W_i(M)$ for $i=1,2$ we will similarly quote separate statements for its conservation under $\msf{SS}$ and $\Upsilon$. Under $\msf{SS}$, this is a consequence of \cite[Theorem~6.6]{sagan1990robinson} (which is also the source of \cite[Theorem~4.11]{imamura2021skew} mentioned in the previous paragraph), i.e., it holds that $W_1(M) = (\#(P, i))_{1\leq i\leq N}$ and $W_2(M) = (\#(Q, j))_{1\leq j\leq N}$. For $\Upsilon$ this preservation property is not recorded explicitly in \cite{imamura2021skew}, but it is easy to see it from its definition. Indeed, as described in \cite[Sections~1.2 and 3.3]{imamura2021skew}, the output $(V,W)$ of $\Upsilon$ is obtained as the asymptotic result of iteratively applying a map known as the \emph{skew RSK} map to $(P,Q)$, and it is immediate from the definition of this map that it does not change the number of times any entry $i$ appears in $P$ or $Q$ (only possibly the location of the entries and/or the shape of the tableaux). Thus this property carries over to $\Upsilon$.
\end{proof}

\section{Asymptotics for the sum in the factorial moments formula}\label{app.sum asymptotics}

Here we obtain upper and lower bounds (with the correct dependencies on $q$ and $k$) on the sum in \eqref{e.factorial moment simplified upper bound}, which we label $S$, i.e., (recall $H(x) = -x\log x-(1-x)\log(1-x)$)
\begin{equation}\label{e.S formula}
\begin{split}
\MoveEqLeft[4]
S := \sum_{i=0}^k\frac{1}{\tfrac{i}{k}(1-\tfrac{i}{k})}\exp\biggl[N\biggl\{\tfrac{i}{N}\log q^{-1} + \frac{2k}{N}\cdot H(i/k)\\
& + \left(1-\frac{i+1}{N}\right)\log\left(1+\frac{k+1}{N-i-1}\right) + \frac{k+1}{N}\log\left(1+\frac{k-i}{N}\right) -\frac{k+1}{N}\biggr\}\biggr].
\end{split}
\end{equation}
As indicated, the idea behind the analysis is simply Laplace's method, but it must be done carefully and explicitly here since we need to obtain the estimates uniformly in $q$. We start with the upper bound, and turn to the lower bound in Section~\ref{s.laplace lower bound}.

\subsection{The upper bound}
\begin{proposition}\label{p.ledoux sum upper bound}
There exist positive constants $C$, $k_0$, and $N_0$ such that for all $N\geq N_0$, $k_0\leq k\leq N$, and $q\in(0,1)$, there exists $x_0 = 1-\Theta(q^{1/2})$ such that
$$S
\leq C q^{-1/4} k^{1/2}\left[\frac{\left(1+q^{1/2}\right)^2}{q}\right]^k \exp\left(N\left[(\alpha-\mu_q)\log\left(1-\frac{\alpha}{\mu_q}\right) - \alpha  + Cq^{1/2}\alpha^3\right]\right) .$$
\end{proposition}

\begin{proof}%
Recall that $\alpha$ is defined by $k=\alpha N$. Also recall the definition of $f_\alpha, g:[0,1]\to \R$ from \eqref{e.f_alpha first defn} by
\begin{equation}\label{e.f definition}
\begin{split}
f_\alpha(x) &= \alpha x\log q^{-1} + 2\alpha H(x) + \alpha \log(1+\alpha(1-x)) + (1-\alpha x)\log\left(1+\frac{\alpha}{1-\alpha x}\right) -\alpha\\
g(x) &= \left(x(1-x)\right)^{-1}.
\end{split}
\end{equation}
Then the sum $S$ is
\begin{align*}
S = \sum_{i=1}^{k-1}g(i/k)\exp\Bigl(Nf_\alpha(i/k) + )(1)\Bigr);
\end{align*}
the $O(1)$ term is to account for the stray $\pm 1$ we have ignored when going from the definition of $S$ to its form in terms of $f_\alpha$.

The first step is to identify the location $x_0$ where $f_\alpha$ is maximized.  \begin{align*}
f_\alpha'(x)
&= \alpha \log q^{-1} + 2\alpha \log\left(\frac{1-x}{x}\right) - \alpha\log\left(1+\frac{\alpha}{1-\alpha x}\right) + (1-\alpha x)\cdot \frac{1}{1+ \frac{\alpha}{1-\alpha x}} \cdot \frac{\alpha}{(1-\alpha x)^2}\cdot \alpha\\
&\qquad + \alpha\cdot\frac{1}{1+\alpha(1-x)}\cdot(-\alpha) \\
&= \alpha \log q^{-1} + 2\alpha \log\left(\frac{1-x}{x}\right) - \alpha\log\left(1+\frac{\alpha}{1-\alpha x}\right).
\end{align*}
Therefore we see that $f_\alpha'(x) = 0$ is equivalent to
\begin{align}\label{e.maximizer relation}
x^{-1} -1 = q^{1/2}\left(1+\frac{\alpha}{1-\alpha x}\right)^{1/2}.
\end{align}
Observe that the LHS tends to $\infty$ as $x \to 0$ and equals $0$ when $x=1$. Further the LHS is decreasing while the RHS is increasing in $x$, and both are continuous in $x$. Thus there is a unique $x_0\in(0,1)$ satisfying \eqref{e.maximizer relation}.

Further, the same observations yield (by evaluating the RHS of \eqref{e.maximizer relation} at $x=0$ and $x=1$ and solving for $x$ on the LHS) that
\begin{align}\label{e.x_0 range}
x_0\in \left[\frac{1}{1+q^{1/2}(1+\alpha)^{1/2}}, \frac{1}{1+q^{1/2}(1-\alpha)^{-1/2}}\right] =: I_{q,\alpha}.
\end{align}
We further see that the size of $I_{q,\alpha}$ is
\begin{align}\label{e.size of x_0 range}
O\left(q^{1/2}\left[\left(1-\alpha\right)^{-1/2} - (1+\alpha)^{1/2}\right]\right) = O\left(q^{1/2}\alpha^2\right).
\end{align}
As a result, if $x\in I_{q,\alpha}$, and since $|f_\alpha'(x)| = O(\alpha)$ for such $x$, the mean value theorem implies that $|f_\alpha(x) - f_\alpha(x_0)| = O(q^{1/2}\alpha^3)$.

\subsection*{Evaluating $f_\alpha(x_0)$}
Using these estimates we may obtain an estimate of $f_\alpha(x_0)$. Let us first evaluate the first three terms from \eqref{e.f definition} at 
$$x=x^* := \left(1+q^{1/2}\exp(\tfrac{1}{2}\alpha)\right)^{-1}.$$
(Note that since $1+x\leq \exp(x) \leq (1-x)^{-1}$, $x^*\in I_{q,\alpha}$.)  We see that the three terms evaluated at $x^*$ equal
\begin{align*}
\MoveEqLeft
\alpha x^*\log q^{-1} - 2\alpha (x^*\log x^* + (1-x^*)\log(1-x^*)) + \alpha \log\left(1+\alpha(1-x^*)\right)\\
&= \alpha x^*\log q^{-1} - 2\alpha \left[x^*\log x^* + (1-x^*)\log \frac{q^{1/2}\exp(\tfrac{1}{2}\alpha)}{1+q^{1/2}\exp(\tfrac{1}{2}\alpha)}\right] + \alpha \log\left(1+\alpha(1-x^*)\right)\\
&= \alpha x^*\log q^{-1} - 2\alpha \left[x^*\log x^* + \tfrac{1}{2}(1-x^*)\log \left(q\exp(\alpha)\right) + (1-x^*)\log \frac{1}{1+q^{1/2}\exp(\tfrac{1}{2}\alpha)}\right]\\
&\qquad + \alpha \log\left(1+\alpha(1-x^*)\right).
\end{align*}
We recognize $(1+q^{1/2}\exp(\alpha/2))^{-1}$ to be $x^*$, which results in a cancellation with the $x^*\log x^*$ term. With this, and writing $\log(1+\alpha(1-x^*)) = \alpha(1-x^*) \pm O(\alpha^2(1-x^*)^2)$, we see that the previous display equals
\begin{align*}
\MoveEqLeft
\alpha x^*\log q^{-1} - 2\alpha \left[\log x^* + \tfrac{1}{2}(1-x^*)\log q + \tfrac{1}{2}(1-x^*)\alpha\right] + \alpha \cdot(\alpha(1-x^*)) \pm O(\alpha^3(1-x^*)^2)\\
&= \alpha x^*\log q^{-1} - 2\alpha\log x^*  - \alpha (1-x^*)\log q \pm O(\alpha^3(1-x^*)^2)\\
&= \alpha \log\left(q^{-1}(1+q^{1/2}\exp(\tfrac{1}{2}\alpha))^2\right) \pm O\left(\alpha^3q\right)\\
&= \alpha \log\left(q^{-1}(1+q^{1/2})^2\right) + \alpha^2\frac{q^{1/2}}{1+q^{1/2}} \pm O(\alpha^3q^{1/2}),
\end{align*}
where we performed a Taylor expansion of $\exp(\alpha/2)$ as well as of the logarithm in the last step.

Next we must move from this calculation done at $x^*$ to $x_0$. Observe that the derivative with respect to $x$ of both the first three terms in \eqref{e.f definition} of $f_\alpha$ (which we will refer to as $\smash{f_\alpha^{(1)}}$ below) gives an expression with a common factor of $\alpha$. By the mean value theorem, we therefore have that, for some $y\in[x_0, x^*]$ (or $y\in[x^*, x_0]$ depending on which is greater),
\begin{align*}
|f^{(1)}_\alpha(x_0) - f^{(1)}_\alpha(x^*)| = |x_0-x^*|\cdot|(f^{(1)}_\alpha)'(y)| = O(1)\alpha^3q^{1/2},
\end{align*}
the last equality using that $|(f^{(1)}_\alpha)'(y)| = O(\alpha)$ and $|x_0-x^*| = O(\alpha^2q^{1/2})$ from \eqref{e.size of x_0 range} and since $x_0, x^*\in I_{q,\alpha}$.

Putting back in the the remaining two terms of $f_\alpha$ not included in $f_\alpha^{(1)}$, overall we have shown that
\begin{align}
f_\alpha(x_0)
&= \alpha \log\left(q^{-1}(1+q^{1/2})^2\right) + (1-\alpha x_0)\log\left(1+\frac{\alpha}{1-\alpha x_0}\right) -\alpha + \alpha^2\frac{q^{1/2}}{1+q^{1/2}} \pm O\left(\alpha^3q^{1/2}\right).\nonumber
\end{align}
This can be simplified further using the following cancellation, which we prove in Section~\ref{app.factorial to poly cancellation}.

\begin{lemma}\label{l.factorial to poly cancellation}
Let $x_0 = x_0(\alpha)$ be the maximizer of $f_\alpha$ over $[0,1]$. For $\alpha\in(0,\frac{1}{2}]$ (and adopting the convention that $O(1)$ refers to a quantity whose value is upper bounded by an absolute constant)
\begin{equation}
\begin{split}\label{e.factorial to poly terms}
\MoveEqLeft[16]
(1-\alpha x_0)\log\left(1+\frac{\alpha}{1-\alpha x_0}\right) + \alpha^2\cdot\frac{q^{1/2}}{1+q^{1/2}} + (\mu_q-\alpha)\log\left(1-\frac{\alpha}{\mu_q}\right)=  \pm O(1)q^{1/2}\alpha^3.
\end{split}
\end{equation}
\end{lemma}%

This yields that
\begin{align}\label{e.f_alpha(x_0) expression}
f_\alpha(x_0) = \alpha \log\left(q^{-1}(1+q^{1/2})^2\right) + (\alpha-\mu_q)\log\left(1-\frac{\alpha}{\mu_q}\right) -\alpha \pm O\left(\alpha^3q^{1/2}\right).
\end{align}

\subsection*{Performing Laplace's method}
For future reference we also record that
\begin{align}\label{e.f second derivative}
f_\alpha''(x) = -\frac{2\alpha}{x(1-x)} - \frac{\alpha^3}{(1-\alpha x)^2 + \alpha(1-\alpha x)} \implies f_\alpha''(x_0) = -\Theta(\alpha q^{-1/2})
\end{align}
since $x_0 = 1-\Theta(q^{1/2})$.
To apply Laplace's method, we need to have bounds on $f_\alpha$ over its domain, which we will obtain by Taylor approximations. We will expand to third order as we need to include the just calculated second order term precisely. So next we bound the third derivative of $f_\alpha$.

Fix $c>0$. We observe that there exists $C$ such that for $x\in[\frac{1}{4}, 1-cq^{1/2}]$ and $\alpha\in(0,\frac{1}{2}]$,
\begin{equation}\label{e.f third derivative}
f_\alpha'''(x) = -2\alpha\left(\frac{1}{(1-x)^2} - \frac{1}{x^2}\right) - \frac{\alpha^4\left(2(1-\alpha x)+\alpha\right)}{\left( (1-\alpha x)^2 + \alpha (1-\alpha x)\right)^2} \geq -C\alpha q^{-1}.
\end{equation}
So by Taylor's theorem, we see that, if $\varepsilon>0$ is such that $x_0-\varepsilon q^{1/2}\in[\frac{1}{4}, 1-cq^{1/2}]$, then for some $y\in[x_0-\varepsilon q^{1/2}, x_0]$,
\begin{align*}
f_\alpha(x_0-\varepsilon q^{1/2})
&= f_\alpha(x_0) - \varepsilon q^{1/2} f_\alpha'(x_0) + \tfrac{1}{2}\varepsilon^2qf_\alpha''(x_0) - \tfrac{1}{6}\varepsilon^3q^{3/2} f'''(y)\\
&\leq f_\alpha(x_0) - \varepsilon q^{1/2} f_\alpha'(x_0) + \tfrac{1}{2}\varepsilon^2qf_\alpha''(x_0) + \tfrac{1}{6}C\alpha q^{1/2}\varepsilon^3\\
&= f_\alpha(x_0) + \tfrac{1}{2}\varepsilon^2qf_\alpha''(x_0) + \tfrac{1}{6}C\alpha q^{1/2}\varepsilon^3
\end{align*}
since $f_\alpha'(x_0) = 0$. We may pick $\varepsilon_0$ a constant depending only on $C$ such that, if $0<\varepsilon<\varepsilon_0$, then (recalling that $f_\alpha''(x_0)<0$)
$$\tfrac{1}{6}C\alpha q^{1/2}\varepsilon < \tfrac{1}{4}q|f_\alpha''(x_0)|;$$
that $\varepsilon_0$ can be taken to not depend on $q$ or $\alpha$ follows from the fact that $|f_\alpha''(x_0)|$ can be lower bounded by an absolute constant times $\alpha q^{-1/2}$, due to \eqref{e.f second derivative}. Similarly, since $f'''(x) < C\alpha$ for $x>\frac{1}{4}$, it holds for $\varepsilon>0$ that (since $x_0>\frac{1}{4}$ always)
\begin{align*}
f_\alpha(x_0 + \varepsilon q^{1/2}) \leq f_\alpha(x_0) + \tfrac{1}{2}\varepsilon^2 q f_\alpha''(x_0) + \tfrac{1}{6}C\alpha \varepsilon^3 q^{3/2} \leq f_\alpha(x_0) + \tfrac{1}{4}\varepsilon^2 q f_\alpha''(x_0),
\end{align*}
the last inequality for $\varepsilon < \varepsilon_0$ for some absolute constant $\varepsilon_0$ by similar reasoning as above.

So, for $-\varepsilon_0 < \varepsilon < \varepsilon_0$,
\begin{align}\label{e.f quadratic around maxima}
f_\alpha(x_0+\varepsilon q^{1/2}) \leq f_\alpha(x_0) - \tfrac{1}{4}\varepsilon^2q|f_\alpha''(x_0)|.
\end{align}
The above controls $f_\alpha$ inside $[x_0-\varepsilon q^{1/2}, x_0+\varepsilon q^{1/2}]$, where $0<\varepsilon < \varepsilon_0$. We will also need control outside this interval, which we turn to next.
We observe that, since $f_\alpha$ is concave on $(0,1)$ and $f_\alpha'(x_0) = 0$, it holds for $0<\varepsilon<\varepsilon_0$ and $x\in(0,x_0-\varepsilon q^{1/2}]$ that
\begin{align}
f_\alpha(x)
&\leq f_\alpha(x_0 - \varepsilon q^{1/2}) + (x-x_0-\varepsilon q^{1/2}) f'(x_0-\varepsilon q^{1/2})\nonumber\\
&\leq f_\alpha(x_0) - \tfrac{1}{4}\varepsilon^2q|f''(x_0)| - (x_0-x)f'(x_0-\varepsilon q^{1/2})\nonumber\\
&\leq f_\alpha(x_0) - (x_0-x)C \varepsilon \alpha. \label{e.f bound on left interval}
\end{align}
the last inequality using again that $|f_\alpha''(x_0)| = \Theta(\alpha q^{-1/2})$ and using Taylor's theorem for $f_\alpha'$ around $x_0$ to obtain $f_\alpha'(x_0-\varepsilon q^{1/2}) = f_\alpha'(x_0) - \varepsilon q^{1/2} f_\alpha''(y)$ for some $y\in[x_0-\varepsilon q^{1/2}, x_0]$ and using  \eqref{e.f second derivative} to then obtain that $f_\alpha'(x_0-\varepsilon q^{1/2}) \geq C\varepsilon \alpha$.

Similarly, since $f_\alpha$ is concave, $f'(x_0) = 0$, and \eqref{e.f quadratic around maxima}, it follows for $x\in[x_0+\varepsilon q^{1/2}, 1]$ that
\begin{align}
f_\alpha(x) \leq f_\alpha(x_0 + \varepsilon q^{1/2}) \leq f_\alpha(x_0) - \tfrac{1}{4} q\varepsilon^2 |f''_\alpha(x_0)|\leq f_\alpha(x_0) - C \alpha q^{1/2}\varepsilon^2.  \label{e.f bound on right interval}
\end{align}
We next analyze the actual sum $S$. We break up $S$ into three subsums $S_1$, $S_2$, and $S_3$. Let $0 < \varepsilon< \varepsilon_0$ be fixed. $S_1$ corresponds to $i = 1$ to $i=\lceil k(x_0-\varepsilon q^{1/2})\rceil$, $S_2$ to $i=\lceil k(x_0-\varepsilon q^{1/2})\rceil+1$ to $\lceil k(x_0+\varepsilon q^{1/2})\rceil$, and $S_3$ to $\lceil k(x_0+\varepsilon q^{1/2})\rceil +1$ to $k-1$.

We start by bounding $S_1$ using the above groundwork. Recall that $g(x) = (x(1-x))^{-1}$. Recall that $g(i/k) = k^2/(i(k-i))$. We drop the $\lceil\rceil$ in the notation for convenience. Using \eqref{e.f bound on left interval}, that $q\leq 1$, and that $\alpha N=k$,
\begin{align*}
S_1 = \sum_{i=1}^{k(x_0-\varepsilon q^{1/2})+1} g(i/k)\exp\left(Nf_\alpha(i/k)\right)
&\leq  e^{Nf_\alpha(x_0)}\sum_{i=1}^{k(x_0-\varepsilon q^{1/2})+1} \frac{k^2}{i(k-i)}\exp\Bigl(-C\varepsilon (kx_0-i)\Bigr)\\
&= e^{Nf_\alpha(x_0)}\!\!\sum_{j=k\varepsilon q^{1/2} -1}^{kx_0 - 1} \frac{k^2}{(kx_0-j)(k(1-x_0) +j)}\exp\Bigl(-C\varepsilon j\Bigr).
\end{align*}
It is easy to see that the sum is bounded by $Cq^{-1/2}$ for an absolute constant $C$ depending on $\varepsilon$, as this is the behaviour near $j=k\varepsilon q^{1/2}$ (using that $x_0 = 1-\Theta(q^{1/2})$). At the same time, it is easy to check that for $k\geq 2$ and $1\leq j\leq k-1$, $g(i/k)\leq 2k$, so that $S_1$ is also upper bounded by $Ck\exp(Nf_\alpha(x_0))$. Thus 
$$S_1 \leq C\min(q^{-1/2}, k)e^{Nf_\alpha x_0} \leq C q^{-1/4}k^{1/2}e^{Nf_\alpha(x_0)}.$$

Next we bound $S_3$. 
We apply \eqref{e.f bound on right interval} and use again that $g(i/k) \leq 2k$ and $x_0=1-\Theta(q^{1/2})$ to see that
\begin{align*}
S_3 = \sum_{i=k(x_0+\varepsilon q^{1/2})+1}^{k-1} g(i/k)\exp\left(Nf_\alpha(i/k)\right)
&\leq e^{Nf_\alpha(x_0)}\cdot 2k\cdot k(1-x_0-\varepsilon q^{1/2})\exp(-C\varepsilon^2 N \alpha q^{1/2})\\
&\leq C e^{Nf_\alpha(x_0)}\cdot k^2\cdot q^{1/2}\exp(-Ckq^{1/2}).
\end{align*}
Now $k^2q^{1/2} = q^{-1/4}k^{1/2}\cdot (kq^{1/2})^{3/2}$ and, since $x\mapsto x^{3/2}\exp(-cx)$ is uniformly bounded over $x\geq 0$, this implies from the previous display that
\begin{align*}
S_3 \leq Cq^{-1/4}k^{1/2}e^{Nf_\alpha(x_0)}.
\end{align*}

Finally we turn to the main sum, $S_2$, which consists of the range $\frac{i}{k}\in[x_0-\varepsilon q^{1/2}, x_0+\varepsilon q^{1/2}]$. We first want to say that, for $x$ in the same range, $g(x) \leq Cg(x_0)$ for some absolute constant $C$. Observe that $g$ blows up near $1$ (and $x_0$ can be arbitrarily close to $1$), and it is to avoid this and thereby be able to control $g$ on the mentioned interval that its upper boundary is of order $q^{1/2}$ above $x_0$.

\begin{lemma}\label{l.g control}
There exists an absolute constant $C$ such that for $x\in[x_0-\varepsilon q^{1/2}, x_0+\varepsilon q^{1/2}]$, $g(x) \leq Cg(x_0)$.
\end{lemma}

\begin{proof}
We have to upper bound $g(x)/g(x_0) = \smash{\frac{x_0(1-x_0)}{x(1-x)}}$. When $x\in[x_0-\varepsilon q^{1/2}, x_0]$, this ratio is upper bounded by $x_0/x$; since $x\geq x_0-\varepsilon q^{1/2}$, and $x_0\geq \frac{1}{4}$ always, if $\varepsilon<\frac{1}{8}$ say, the ratio is uniformly upper bounded. When $x\in[x_0, x_0+\varepsilon q^{1/2}]$, this ratio is upper bounded by $(1-x_0)/(1-x)$; since $1-x\geq 1-x_0-\varepsilon q^{1/2} \geq \frac{1}{2}(1-x_0)$ (using that $\varepsilon q^{1/2} \leq \frac{1}{2}(1-x_0) = cq^{1/2}$ whenever $\varepsilon$ is small enough), the ratio is again upper bounded by a constant.
\end{proof}

So we see, from Lemma~\ref{l.g control} and \eqref{e.f quadratic around maxima} that $S_2$ equals
\begin{align*}
\MoveEqLeft[9]
\sum_{i=k(x_0-\varepsilon q^{1/2})+1}^{k(x_0+\varepsilon q^{1/2})} g(i/k)\exp\left(Nf_\alpha(i/k)\right)\\
&\leq Cg(x_0) \sum_{i=k(x_0-\varepsilon q^{1/2})+1}^{k(x_0+\varepsilon q^{1/2})} \exp\left[N\left(f_\alpha(x_0) - c\left(\frac{i}{k}-x_0\right)^2|f_\alpha''(x_0)|\right)\right]\\
&\leq Cg(x_0)e^{Nf_\alpha(x_0)} \sum_{i=-\infty}^{\infty} \exp\left[-cNk^{-2}\left(i-kx_0\right)^2|f_\alpha''(x_0)|\right]\\
&= Cg(x_0)e^{Nf_\alpha(x_0)} \sum_{i=-\infty}^{\infty} \exp\left[-c\alpha^{-1}k^{-1}i^2|f_\alpha''(x_0)|\right].
\end{align*}
Recall that $f_\alpha''(x_0)=-\Theta(\alpha q^{-1/2})$, and so the coefficient of $i^2$ is of order $c(q^{1/2}k)^{-1}$. We bound the above series using Proposition~\ref{p.quadratic sum bound} ahead, which says, with $\gamma = c\alpha^{-1}k^{-1}|f_\alpha''(x_0)| = \Theta((q^{1/2}k)^{-1})$, and using that $g(x_0) = \Theta(q^{-1/2})$,
\begin{align*}
S_2
\leq Cg(x_0) (\alpha^{-1}k^{-1}|f_\alpha''(x_0)|)^{-1/2} e^{Nf_\alpha(x_0)}
&\leq C q^{-1/2}\cdot k^{1/2} q^{1/4}\cdot e^{Nf_\alpha(x_0)}
= Cq^{-1/4}k^{1/2}e^{Nf_\alpha(x_0)}.
\end{align*}
Note that this estimate holds for all $q>0$ and dos not require $q\geq k^{-2}$.
Thus overall we have shown that
$$S= S_1+S_2+S_3 
 \leq Cq^{-1/4}k^{1/2}e^{Nf_\alpha(x_0)}.$$
Using the expression for $f_\alpha(x_0)$ from \eqref{e.f_alpha(x_0) expression} and recalling $\alpha N = k$ completes the proof.
\end{proof}

The following is the bound on the discrete Gaussian sum, more precisely a Jacobi theta function, which we used in the proof. It can be proved using the Poisson summation formula and straightforward bounds.

\begin{proposition}[page 157 of \cite{stein2011fourier}]\label{p.quadratic sum bound}
There exists $C$ and, for any $M>0$, a constant $c_M>0$ such that for $0<\gamma\leq M$ (for the first inequality) and $\gamma>0$ (for the second),
\begin{align*}
c_M\gamma^{-1/2}\leq \sum_{i\in\Z} e^{-\gamma i^2} \leq C\gamma^{-1/2}.
\end{align*}
\end{proposition}

\subsection{The lower bound} \label{s.laplace lower bound}
Recall the definition of $S$ from \eqref{e.S formula}.

\begin{proposition}\label{p.ledoux sum lower bound}
There exist positive constants $C$ and $k_0$ such that for all $q\in[k^{-2},1)$ and $k_0\leq k\leq N$,
$$S  \geq C^{-1} q^{-1/4} k^{1/2}\left[\frac{(1+q^{1/2})^2}{q}\right]^k\exp\left(N\left[(\alpha-\mu_q)\log\left(1-\frac{\alpha}{\mu_q}\right) -\alpha  -C\alpha^3q^{1/2}\right]\right) .$$
\end{proposition}

\begin{proof}
As in the proof of Proposition~\ref{p.ledoux sum upper bound}, we focus around the point $Nx_0$, where $x_0 = (1+q^{1/2}\exp(\frac{1}{2}\alpha))^{-1} \pm O(q^{1/2}\alpha^2)$ satisfies \eqref{e.maximizer relation}. So, for $\varepsilon,\varepsilon'>0$ to be chosen,
\begin{align*}
S
&\geq \sum_{i=k(x_0-\varepsilon)}^{k(x_0+\varepsilon')} \frac{1}{\tfrac{i}{k}(1-\tfrac{i}{k})}\exp\biggl[N\biggl\{\tfrac{i}{N}\log q^{-1} + \frac{2k}{N}\cdot H(i/k)\\
&\qquad\qquad + \left(1-\frac{i+1}{N}\right)\log\left(1+\frac{k+1}{N-i-1}\right) + \frac{k+1}{N}\log\left(1+\frac{k-i}{N}\right) -\frac{k+1}{N}\biggr\}\biggr]\\
&= \sum_{i=k(x_0-\varepsilon)}^{k(x_0+\varepsilon')} g(i/k)\exp\left[Nf_\alpha(i/k)\right],
\end{align*}
where $g(x) = (x(1-x))^{-1}$ and $f_\alpha$ is as defined in \eqref{e.f definition}.
We break into two cases depending on whether $x_0>\frac{1}{2}$ or $x_0\leq \frac{1}{2}$. We start with the first case.

In this case the basic issue is that $x_0$ can be arbitrarily close to $1$. Thus, since the contribution of the sum from $i=kx_0$ to $i=k-1$ will be small anyway, we will ignore it and set $\varepsilon'=0$. We will also set $\varepsilon$ to be an absolute constant. With these values, we want to show that, for some $c>0$ and all $x\in[x_0-\varepsilon, x_0]$, it holds that $g(x) \geq cg(x_0)$. This is easy to verify by upper bounding $g(x)/g(x_0)$ for $x$ in the same range using the expression for $g(x)$, the value of $\varepsilon$, and that $x_0\in[\frac{1}{4}, 1]$. So, for some absolute constant $c>0$,
$$S\geq cg(x_0)\sum_{i=k(x_0-\varepsilon)}^{kx_0} \exp\left[Nf_\alpha(i/k)\right].$$
Now as in the proof of Proposition~\ref{p.ledoux sum upper bound}, we Taylor expand $f_\alpha$ around $x_0$ to obtain a lower bound on $f_\alpha(x)$ for $x\in[x_0-\varepsilon, x_0]$:
\begin{align*}
f_\alpha(x) &\geq f_\alpha(x_0) - \frac{(x-x_0)^2}{2}|f_\alpha''(x_0)| + \frac{1}{6}\inf_{y\in[x, x_0]} (x-x_0)^3 f_\alpha'''(y)\\
&= f_\alpha(x_0) - \frac{(x-x_0)^2}{2}|f_\alpha''(x_0)| + \frac{1}{6}(x-x_0)^3 \sup_{y\in[x,x_0]}f_\alpha'''(y);
\end{align*}
the second equality by noting that, since $x-x_0 < 0$, the infimum is equivalent to maximizing $f_\alpha'''(y)$ over $y\in[x,x_0]$. 
Now, if $x\geq \frac{1}{2}$, then $y\geq \frac{1}{2}$ in the previous display and $f_\alpha'''(y)<0$ by the explicit formula \eqref{e.f third derivative}, and so the last term in the previous display is non-negative, i.e., can be lower bounded by zero. If $x<\frac{1}{2}$, we can ensure that $\varepsilon<\frac{1}{4}$, so that $x_0 < \frac{3}{4}$. Then we see that each of $|f_\alpha''(x_0)|$ and $|f_\alpha'''(x)|$ are uniformly bounded by $C\alpha$ where $C$ is independent of $q$, so, by picking $\varepsilon$ small enough also independent of $q$, we can ensure that $|x-x_0|f_\alpha'''(x) < Cf_\alpha''(x_0)$ for some absolute constant $C$.

Thus over all, we have shown that, in the case that $x_0\geq \frac{1}{2}$, there exists $\varepsilon>0$ and $C>0$ independent of $q$ such that, for $x\in[x_0-\varepsilon, x_0]$,
\begin{align*}
f_\alpha(x)\geq f_\alpha(x_0) - C(x-x_0)^2|f_\alpha''(x_0)|.
\end{align*}
Using this we see that
\begin{align*}
S
&\geq cg(x_0)e^{Nf_\alpha(x_0)}\sum_{i=k(x_0-\varepsilon)}^{kx_0} \exp\left(-CN|f_\alpha''(x_0)|\left(\frac{i}{k}-x_0\right)^2\right)\\
&= \frac{1}{2}cg(x_0)e^{Nf_\alpha(x_0)}\sum_{i=k(x_0-\varepsilon)}^{k(x_0+\varepsilon)} \exp\left(-CN|f_\alpha''(x_0)|\left(\frac{i}{k}-x_0\right)^2\right).
\end{align*}
Now it is easy to see that $\sum_{i:|i-kx_0| > k\varepsilon} \exp\left(-CNk^{-2}|f_\alpha''(x_0)|(i-kx_0)^2\right) \leq C\exp(-cN|f_\alpha''(x_0)|\varepsilon^2)$, so the previous display is lower bounded by (using Proposition~\ref{p.quadratic sum bound} in the third line)
\begin{align*}
\MoveEqLeft
cg(x_0)e^{Nf_\alpha(x_0)}\left[\sum_{i=-\infty}^{\infty} \exp\left(-CNk^{-2}|f_\alpha''(x_0)|\left(i-kx_0\right)^2\right) - C\exp(-cN|f_\alpha''(x_0)|\varepsilon^2)\right]\\
&=cg(x_0)e^{Nf_\alpha(x_0)}\left[\sum_{i=-\infty}^{\infty} \exp\left(-C\alpha^{-1}k^{-1}|f_\alpha''(x_0)|i^2\right) - C\exp(-cN|f_\alpha''(x_0)|\varepsilon^2)\right]\\
&\geq cg(x_0)e^{Nf_\alpha(x_0)} \left[\alpha^{1/2}k^{1/2}|f_\alpha''(x_0)|^{-1/2} - C\exp(-c\varepsilon^2N|f_\alpha''(x_0)|)\right]\\
&\geq cg(x_0)e^{Nf_\alpha(x_0)} \left[q^{1/4}k^{1/2} - C\exp(-c\varepsilon^2kq^{-1/2})\right],
\end{align*}
using that $|f_\alpha''(x_0)| = \Theta(\alpha q^{-1/2})$ and $\alpha N=k$. This is in turn lower bounded by $cq^{-1/4}k^{1/2}e^{Nf_\alpha(x_0)}$ since $g(x_0) = \Theta(q^{-1/2})$ since $x_0 = 1-\Theta(q^{1/2})$.

In the case that $x_0\leq \frac{1}{2}$, the same proof works after noting that, since also $x_0\geq \frac{1}{4}$ for all $q, \alpha$ (as can be observed from \eqref{e.x_0 range}), quantities like $|f''(x_0)|$, $f'''(x)$, and $g(x)$ are all bounded above and below by absolute constants for $x\in[x_0-\varepsilon, x_0+\varepsilon]$ (assuming $\varepsilon<\frac{1}{8}$ say) and $\alpha\in(0,\frac{1}{2}]$.
\end{proof}

\subsection{The factorial to polynomial moment cancellation}\label{app.factorial to poly cancellation}

Here we give the proof of Lemma~\ref{l.factorial to poly cancellation}. The proof goes by doing a full series expansion of the expression in $\alpha$ around $0$. Performing a Taylor expansion to second order in $\alpha$ of the expression under consideration for fixed $x$, and uniformly (over $\alpha\in(0,\frac{1}{2})$ and $x=x_0=1-\Theta(q^{1/2})$) bounding the third derivative error term by $Cq^{1/2}$ is also possible but somewhat messier.

\begin{proof}[Proof of Lemma~\ref{l.factorial to poly cancellation}]
To establish this we will utilize the series expansion for the logarithm as well as for $(1-x)^{-k}$.
We start with the first term: for any $x$,
\begin{align*}
(1-\alpha x)\log\left(1+\frac{\alpha}{1-\alpha x}\right)
&= \sum_{j=1}^{\infty}(-1)^{j-1} \frac{\alpha^j}{j(1-\alpha x)^{j-1}}\\
&= \alpha + \sum_{j=2}^{\infty}(-1)^{j-1}\frac{\alpha^j}{j}\sum_{i=0}^{\infty} (\alpha x)^{i}\binom{j+i-2}{i}.
\end{align*}
Collecting the $\alpha^\ell$ terms together, the previous line equals
\begin{align*}
\alpha + \sum_{\ell=2}^{\infty}\alpha^{\ell}\sum_{i=0}^{\ell -2} \frac{(-1)^{\ell-i-1}}{\ell-i}x^{i}\binom{\ell-2}{i}
&= \alpha -\tfrac{1}{2}\alpha^2 + \sum_{\ell=3}^{\infty}\alpha^{\ell}\sum_{i=0}^{\ell -2} \frac{(-1)^{\ell-i-1}}{\ell-i}x^{i}\binom{\ell-2}{i}.
\end{align*}
Invoking Lemma~\ref{l.coefficient identity} ahead, the sums in the previous display equals
\begin{align*}
\sum_{\ell=3}^{\infty}\frac{\alpha^{\ell}}{\ell(\ell-1)}\left[(-1+x)^{\ell} + \ell(-1+x)^{\ell-1} - x^{\ell}\right]
&= \sum_{\ell=3}^{\infty}\frac{\alpha^{\ell}}{\ell(\ell-1)}\left[(-1+x)^{\ell-1}(-1+x+\ell) - x^{\ell}\right].
\end{align*}
Now turning to the third term in \eqref{e.factorial to poly terms},
\begin{align*}
(\mu_q-\alpha)\log\left(1-\frac{\alpha}{\mu_q}\right)
&= -(\mu_q-\alpha)\sum_{j=1}^\infty \frac{\alpha^j}{j\mu_q^j}
= -\alpha + \frac{\alpha^2}{2\mu_q}  + \sum_{\ell=3}^\infty \frac{\alpha^\ell}{\ell(\ell-1) \mu_q^{\ell-1}}.
\end{align*}
Thus we see overall that the LHS of \eqref{e.factorial to poly terms} equals (where $x_0$ is the solution of \eqref{e.maximizer relation})
\begin{align} \label{e.coefficient cancellation sum}
\MoveEqLeft
\alpha^2\left(-\frac{1}{2} + \frac{1}{2\mu_q} + \frac{q^{1/2}}{1+q^{1/2}}\right) + \sum_{\ell=3}^\infty \frac{\alpha^{\ell}}{\ell(\ell-1)} \left[(-1+x_0)^{\ell-1}(-1+x_0+\ell) - x_0^{\ell} + \mu_q^{-(\ell-1)}\right].
\end{align}
We focus on the coefficient of $\alpha^2$ first, which, using that $\mu_q = (1+q^{1/2})/(1-q^{1/2})$, simplifies to $0$. 

Now we turn to the sum in \eqref{e.coefficient cancellation sum}. We observe that $\mu_q= 1+\Theta(q^{1/2})$ and recall that $x_0 = 1-\Theta(q^{1/2})$ (from \eqref{e.x_0 range}), so that the expression in the square brackets is equal to
\begin{align*}
\Theta(1)(-1)^{\ell-1}\ell q^{(\ell-1)/2} - 1 + \Theta(\ell q^{1/2}) + 1 -\Theta(\ell q^{1/2})= \pm O(\ell q^{1/2}),
\end{align*}
implying that the sum in \eqref{e.coefficient cancellation sum} is
$\pm O(1)\sum_{\ell=3}^\infty \frac{q^{1/2}\alpha^\ell}{(\ell-1)} 
= \pm O(1)q^{1/2}\alpha^3.$
This completes the proof.
\end{proof}

\begin{lemma}\label{l.coefficient identity}
It holds for any $x\in\R$ and $\ell\geq 2$ that
\begin{align*}
\ell(\ell-1)\sum_{i=0}^{\ell -2} \frac{(-1)^{i+1}}{\ell-i}x^{i}\binom{\ell-2}{i} = (-1+x)^{\ell} + \ell(-1+x)^{\ell-1} - x^\ell.
\end{align*}
\end{lemma}

\begin{proof}
We do the change of variable $i\mapsto \ell-i-2$ and use $\binom{n}{k} = \frac{n}{k}\binom{n-1}{k-1}$ twice to write the LHS as 
\begin{align*}
\ell(\ell-1)\sum_{i=0}^{\ell -2} \frac{(-1)^{\ell-i-1}}{i+2}x^{\ell-i-2}\binom{\ell-2}{i}=
\sum_{i=0}^{\ell -2} (-1)^{i+1}(i+1)x^{\ell-i-2}\binom{\ell}{i+2}.
\end{align*}

Now we do a change of variable ($i\mapsto i-2$) and add and subtract the terms corresponding to $i=0$ and $i=1$ in the new indexing to obtain
\begin{align*}
\sum_{i=0}^{\ell} (-1)^{i+1}x^{\ell-i}(i-1)\binom{\ell}{i} - x^{\ell}.
\end{align*}
Multiplying out over the $(i-1)$ factor and using the binomial theorem on each of the resulting terms yields the claim.
\end{proof}

\section{Concentration inequality proofs}\label{s.concentration}

Here we prove the concentration inequality Theorem~\ref{t.concentration}. As is typical for such inequalities, the main step is to obtain a bound on the moment generating function.

\begin{proposition}\label{p.exponential moment bound}
Suppose $X$ is such that
\begin{align}\label{e.cubic tail}
\P(X\geq t) \leq C_1\exp(-\rho t^{3/2})
\end{align}
for some $\rho>0$, and all $t\geq 0$. Then, there exist positive absolute constants $C$ and $c >0$ such that, for $\lambda>0$,
$$\E\left[e^{\lambda X}\right] \leq \exp\left\{C\cdot C_1 \left[\lambda\rho^{-2/3} + \lambda^{3}\rho^{-2}\right]\right\}.$$
\end{proposition}

\begin{proof}
By rescaling $X$ as $\rho^{2/3}X$ it is enough to prove the proposition with $\rho =1$. First, we see that
\begin{align*}
\E[e^{\lambda X}] = \E\left[\lambda \int_{-\infty}^X  e^{\lambda x}\, \diff x\right]
&= \E\left[\lambda \int_{-\infty}^\infty  e^{\lambda x} \one_{X\geq x}\, \diff x\right]
= \lambda \int_{-\infty}^\infty e^{\lambda x} \cdot \P(X\geq x)\, \diff x.
\end{align*}
Next, we break up the integral into two at $0$ and use the hypothesis \eqref{e.cubic tail} for the second term:
\begin{align*}
\lambda\int_{-\infty}^\infty e^{\lambda x} \cdot \P(X\geq x)\, \diff x
&= \lambda\int_{-\infty}^{0} e^{\lambda x} \cdot \P(X\geq x)\, \diff x + \int_{0}^\infty e^{\lambda x} \cdot \P(X\geq x)\, \diff x\\
&\leq \lambda\int_{-\infty}^{0} e^{\lambda x}\,\diff x +  C_1\lambda\int_{0}^\infty e^{\lambda x - x^{3/2}}\, \diff x\\
&= 1 +  C_1\lambda\int_{0}^\infty e^{\lambda x - x^{3/2}}\, \diff x.
\end{align*}
We focus on the second integral now, and make the change of variables $y = \lambda^{-2} x \iff x=\lambda^{2} y$, to obtain
\begin{align*}
C_1\lambda\int_{0}^\infty e^{\lambda x - x^{3/2}}\, \diff x
= C_1\lambda^{3}\int_{0}^\infty e^{\lambda^{3}(y - y^{3/2})}\, \diff y.
\end{align*}
We upper bound this essentially using Laplace's method. We first note that, for all $y>0$, it holds that $y-y^{3/2}\leq 1-\frac{1}{2}y^{3/2}$. So
\begin{align*}
C_1\lambda^{3}\int_{0}^\infty e^{\lambda^{3}(y - y^{3/2})}\, \diff y
\leq C_1\lambda^{3}\cdot \int_{0}^\infty e^{\lambda^{3}(1 - \frac{1}{2}y^{3/2})}\, \diff y
&= C_1\lambda^{3} \cdot 2^{2/3}\Gamma(5/3)\lambda^{-2}e^{\lambda^3}\\
&= C\cdot C_1 \lambda e^{\lambda^3}.
\end{align*}

for positive absolute constants $C$ and $c$. Putting all the above together yields that
\begin{align*}
\E[e^{\lambda X}]
&\leq 1 + C\cdot C_1 \lambda e^{c\lambda^{3}}.
\end{align*}
Now we break into two cases depending on whether $\lambda$ is less than or greater than 1: if $\lambda \leq 1$, then, since $1+x\leq \exp(x)$ and $C\cdot C_1 \exp(c\lambda) \leq C\cdot C_1\exp(c)$, we obtain, for an absolute constant $\tilde C$,
\begin{align*}
\E[e^{\lambda X}] \leq 1 + C\cdot C_1\exp(c)\lambda \leq \exp\left(\tilde C\cdot C_1\lambda\right).
\end{align*}
On the other hand if $\lambda\geq 1$, we observe that, since $x\leq \exp(x)$, and by increasing the coefficient in the exponent,
\begin{align*}
1 + C\cdot C_1 \lambda^3e^{c\lambda^{3}} \leq e^{c'\lambda^{3}}.
\end{align*}
It is easy to see that we may take $c'$ to depend linearly on $C_1$. Thus overall, since $\max(a,b)\leq a+b$ when $a,b\geq 0$, we obtain, for some universal constant $C$,
\begin{equation*}
\E\left[e^{\lambda X}\right] \leq \exp\left(C\cdot C_1(\lambda + \lambda^3)\right).\qedhere
\end{equation*}
\end{proof}

\begin{proof}[Proof of Theorem~\ref{t.concentration}]
Following the proof of the Chernoff bound, we exponentiate inside the probability (with $\lambda>0$ to be chosen shortly) and apply Markov's inequality:
\begin{align*}
\P\left(\sum_{i=1}^I X_i \geq t + C\cdot C_1\sigma_{2/3}\right)
&= \P\left(e^{\lambda\sum_{i=1}^I X_i} \geq e^{\lambda t + \lambda C\cdot C_1\cdot\sigma_{2/3}}\right)
\leq e^{-\lambda t - \lambda C\cdot C_1\cdot\sigma_{2/3}}\prod_{i=1}^I \E\left[e^{\lambda X_i}\right].
\end{align*}
Using Proposition~\ref{p.exponential moment bound}, this is bounded by
\begin{align*}
\exp\left(-\lambda t - \lambda C\cdot C_1\sigma_{2/3} + C\cdot C_1\left[\lambda\sum_{i=1}^I\rho_i^{-2/3} + \lambda^{3}\sum_{i=1}^I\rho_i^{-2}\right]\right)
&= \exp\left(-\lambda t + C\cdot C_1\cdot \lambda^{3}\sigma_2\right),
\end{align*}
the penultimate line using Proposition~\ref{p.exponential moment bound} for $\lambda > 0$ to be chosen soon. Optimizing over $\lambda$ and setting it to $c't^{1/2}\sigma_2^{-1/2}$ for some small constant $c'>0$ now yields that,
\begin{equation*}
\P\left(\sum_{i=1}^I X_i \geq t + C\cdot C_1\cdot \sigma_{2/3}\right) \leq \exp\left(-ct^{3/2}\sigma_2^{-1/2}\right). \qedhere
\end{equation*}
\end{proof}

\bibliographystyle{alpha}
\bibliography{q-pushtasep-lower-tail}

\end{document}